\documentclass[peerreview,a4paper,12pt]{IEEEtran}

\usepackage{amsthm}
\usepackage{amsmath}
\usepackage{dsfont}
\usepackage{amsfonts,amssymb,amsmath}
\usepackage{graphicx}
\usepackage{epsfig}
\usepackage[applemac]{inputenc}
\usepackage{latexsym}

\newtheorem{mydef}{Definition}
\newtheorem{mycor}{Corollary}
\newtheorem{myprop}{Proposition}
\newtheorem{mythe}{Theorem}
\newtheorem{mylem}{Lemma}

\DeclareMathOperator*{\DMC}{DMC}
\DeclareMathOperator*{\Proj}{Proj}

\DeclareMathOperator*{\supp}{supp}
\DeclareMathOperator*{\conv}{co}
\DeclareMathOperator*{\Imag}{Im}
\DeclareMathOperator*{\CE}{CE}
\DeclareMathOperator*{\rank}{rank}
\DeclareMathOperator*{\can}{can}
\DeclareMathOperator*{\cl}{cl}

\DeclareMathOperator*{\MP}{MP}
\DeclareMathOperator*{\BSC}{BSC}
\DeclareMathOperator*{\BEC}{BEC}

\begin{document}

\sloppy

%% Paper Title
%% You can use linebreaks \\ within to get better formatting as
%% desired. 
\title{Topological Structures on DMC spaces}

%% Author names and affiliations:
%%
%% Avoiding spaces at the end of the author lines is not a problem with
%% conference papers because we don't use \thanks or \IEEEmembership.
%%
%% For several authors with only one affiliation:
%%
% \author{
%   \IEEEauthorblockN{Hui-Ting Chang and Stefan M.~Moser}
%   \IEEEauthorblockA{Department of Electrical and Computer Engineering\\
%     National Chiao Tung University (NCTU)\\
%     Hsinchu, Taiwan\\
%     Email: \{email-of-hui-ting,email-of-stefan\}@ieee.org} 
% }
%%
%% For up to three affiliations:
%%

%\author{
%Rajai Nasser\\
%EPFL, Lausanne, Switzerland\\
%rajai.nasser@epfl.ch
%
%\and
%
%Joseph M. Renes\\
%ETH Zurich, Switzerland\\
%renes@itp.phys.ethz.ch
%}

\author{
Rajai Nasser\\
EPFL, Lausanne, Switzerland\\
rajai.nasser@epfl.ch
%\thanks{This paper was presented in part at the IEEE International Symposium on Information Theory, Hong Kong, June 2015.}
}

%%
%% For over three affiliations, or if they all won't fit within the width
%% of the page, use this alternative format:
%%
% \author{
%   \IEEEauthorblockN{
%     Michael Shell\IEEEauthorrefmark{1},
%     Homer Simpson\IEEEauthorrefmark{2},
%     James Kirk\IEEEauthorrefmark{3}, 
%     Montgomery Scott\IEEEauthorrefmark{3} and
%     Eldon Tyrell\IEEEauthorrefmark{4}}
%   \IEEEauthorblockA{
%     \IEEEauthorrefmark{1}School of Electrical and Computer Engineering\\
%     Georgia Institute of Technology, Atlanta, Georgia 30332--0250\\ 
%     Email: see http://www.michaelshell.org/contact.html}
%   \IEEEauthorblockA{
%     \IEEEauthorrefmark{2}Twentieth Century Fox, Springfield, USA\\
%     Email: homer@thesimpsons.com}
%   \IEEEauthorblockA{
%     \IEEEauthorrefmark{3}Starfleet Academy, San Francisco, California 96678-2391\\
%     Telephone: (800) 555--1212, Fax: (888) 555--1212}
%   \IEEEauthorblockA{
%     \IEEEauthorrefmark{4}Tyrell Inc., 123 Replicant Street, Los Angeles, California 90210--4321}
% }

%% Use for special paper notices
%\IEEEspecialpapernotice{(Invited Paper)}

%% To balance the two columns, you should reduce the text-height of
%% the last page using the following command:
%%%%%%%%%%%%%%%%%%%%%%%%%%%%%%%%%%%%%%%%%%%%%%%%%%%%%%%%%%%%%%%%%%%%%
%\addtolength{\textheight}{-9.35cm}
%%%%%%%%%%%%%%%%%%%%%%%%%%%%%%%%%%%%%%%%%%%%%%%%%%%%%%%%%%%%%%%%%%%%%
%% with an appropriate value. This command must be place on the second
%% last page, i.e., for a one-page abstract here, for a two-page
%% abstract right after the \maketitle command.

%% Create the title:
\maketitle

%% Abstract: 
%% For the final version of the accepted paper, please make sure you
%% remove the comment "THIS PAPER IS ELIGIBLE FOR THE STUDENT PAPER
%% AWARD."
%%
\begin{abstract}
Two channels are said to be equivalent if they are degraded from each other. The space of equivalent channels with input alphabet $\mathcal{X}$ and output alphabet $\mathcal{Y}$ can be naturally endowed with the quotient of the Euclidean topology by the equivalence relation. We show that this topology is compact, path-connected and metrizable. A topology on the space of equivalent channels with fixed input alphabet $\mathcal{X}$ and arbitrary but finite output alphabet is said to be natural if and only if it induces the quotient topology on the subspaces of equivalent channels sharing the same output alphabet. We show that every natural topology is $\sigma$-compact, separable and path-connected. On the other hand, if $|\mathcal{X}|\geq 2$, a Hausdorff natural topology is not Baire and it is not locally compact anywhere. This implies that no natural topology can be completely metrized if $|\mathcal{X}|\geq 2$. The finest natural topology, which we call the strong topology, is shown to be compactly generated, sequential and $T_4$. On the other hand, the strong topology is not first-countable anywhere, hence it is not metrizable. We show that in the strong topology, a subspace is compact if and only if it is rank-bounded and strongly-closed. We provide a necessary and sufficient condition for a sequence of channels to converge in the strong topology. We introduce a metric distance on the space of equivalent channels which compares the noise levels between channels. The induced metric topology, which we call the noisiness topology, is shown to be natural. We also study topologies that are inherited from the space of meta-probability measures by identifying channels with their Blackwell measures. We show that the weak-$\ast$ topology is exactly the same as the noisiness topology and hence it is natural. We prove that if $|\mathcal{X}|\geq 2$, the total variation topology is not natural nor Baire, hence it is not completely metrizable. Moreover, it is not locally compact anywhere. Finally, we show that the Borel $\sigma$-algebra is the same for all Hausdorff natural topologies.
\end{abstract}

\section{Introduction}

A topology on a given set is a mathematical structure that allows us to formally talk about the neighborhood of a given point of the set. This makes it possible to define continuous mappings and converging sequences. Topological spaces generalize metric spaces which are mathematical structures that specify distances between the points of the space. Links between information theory and topology were investigated in \cite{TopInInfInTop}. In this paper, we aim to construct meaningful topologies and metrics for the space of equivalent channels sharing a common input alphabet.

Let $\mathcal{X}$ and $\mathcal{Y}$ be two fixed finite sets. Every discrete memoryless channel (DMC) with input alphabet $\mathcal{X}$ and output alphabet $\mathcal{Y}$ can be determined by its transition probabilities. Since there are $|\mathcal{X}|\times|\mathcal{Y}|$ such probabilities, the space of all channels from $\mathcal{X}$ to $\mathcal{Y}$ can be seen as a subset of $\mathbb{R}^{|\mathcal{X}|\times|\mathcal{Y}|}$. Therefore, this space can be naturally endowed with the Euclidean metric, or any other equivalent metric. A generalization of this topology to infinite input and output alphabets was considered in \cite{Schwarte}.

There are a few drawbacks for this approach. For example, consider the case where $\mathcal{X}=\mathcal{Y}=\mathbb{F}_2:=\{0,1\}$. The binary symmetric channels $\BSC(\epsilon)$ and $\BSC(1-\epsilon)$ have non-zero Euclidean distance if $\epsilon\neq\frac{1}{2}$. On the other hand, $\BSC(\epsilon)$ and $\BSC(1-\epsilon)$ are completely equivalent from an operational point of view: both channels have exactly the same probability of error under optimal decoding for any fixed code. Moreover, any sub-optimal decoder for one channel can be transformed to a sub-optimal decoder for the other channel without changing the probability of error nor the computational complexity. This is why it makes sense, from an information-theoretic point of view, to identify equivalent channels and consider them as one point in the space of ``equivalent channels".

The limitation of the Euclidean metric is clearer when we consider channels with different output alphabets. For example, $\BSC\left({\frac{1}{2}} \right)$ and $\BEC\left(1\right)$ are completely equivalent but they do not have the same output alphabet, and so there is no way to compare them with the Euclidean metric because they do not belong to the same space.

The standard approach to solve this problem is to find a ``canonical sufficient statistic" and find a representation of each channel in terms of this sufficient statistic. This makes it possible to compare channels with different output-alphabets. One standard sufficient statistic that has been widely used for binary-input channels is the log-likelihood ratio. Each binary-input channel can be represented as a density of log-likelihood ratios (called $L$-density in \cite{RichardsonUrbanke}). This representation makes it possible to ``topologize" the space of ``equivalent" binary-input channels by considering the topology of convergence in distribution \cite{RichardsonUrbanke}. A similar approach can be adopted for non-binary-input channels (see \cite{RathiUrbanke} and \cite{BannatanBurshtein}). Another (equivalent) way to ``topologize" the space of equivalent channels is by using the Le Cam deficiency distance \cite{LeCam}.

The current formulation of this topology cannot be generalized to the quantum setting because it is given in terms of posterior probabilities which do not have quantum analogues. Therefore, if we want to generalize this topology to the space of equivalent quantum channels and equivalent classical-quantum channels, it is crucial to find a formulation for this topology that does not explicitly depend on posterior probabilities.

Another issue (which is secondary and only relevant for conceptual purposes) is that the current formulation of this topology does not allow us to see it as a ``natural topology". Consider a fixed output alphabet $\mathcal{Y}$ and let us focus on the space of ``equivalent channels" from $\mathcal{X}$ to $\mathcal{Y}$. Since this space is the quotient of the space of channels from $\mathcal{X}$ to $\mathcal{Y}$, which is naturally topologized by the Euclidean metric, it seems that the most natural topology on this space is the quotient of the Euclidean topology by the equivalence relation. This motivates us to consider a topology on the space of ``equivalent channels" with input alphabet $\mathcal{X}$ and arbitrary but finite output alphabet as \emph{natural} if and only if it induces the quotient topology on the subspaces of ``equivalent channels" from $\mathcal{X}$ to $\mathcal{Y}$ for any output alphabet $\mathcal{Y}$. A legitimate question to ask now is whether the $L$-density topology is natural in this sense or not.

In this paper, we study general and particular natural topologies on DMC spaces. In Section II, we provide a brief summary of the basic concepts and theorems in general topology. The measure-theoretic notations that we use are introduced in section III. The space of channels from $\mathcal{X}$ to $\mathcal{Y}$ and its topology is studied in Section IV. We formally define the equivalence relation between channels in section V. It is shown that the equivalence class of a channel can be determined by the distribution of its posterior probability distribution. This is the standard generalization of $L$-densities to non-binary-input channels. This distribution is called the Blackwell measure of the channel. In section VI, we study the space of equivalent channels from $\mathcal{X}$ to $\mathcal{Y}$ and the quotient topology.

In Section VII, we define the space of equivalent channels with input alphabet $\mathcal{X}$ and we study the properties of general natural topologies. The finest natural topology, which we call the \emph{strong topology} is studied in Section VIII. A metric for the space of equivalent channels is proposed in section IX. The induced topology by this metric is called the \emph{noisiness topology}. In section X, we study the topologies that are inherited from the space of meta-probability measures by identifying equivalent channels with their Blackwell measures. We show that the weak-$\ast$ topology (which is the standard generalization of the $L$-density topology to non-binary-input channels) is exactly the same as the noisiness topology. The total variation topology is also investigated in section X. The Borel $\sigma$-algebra of Hausdorff natural topologies is studied in section XI.

The continuity (under the topologies introduced here) of mappings that are relevant to information theory (such as capacity, mutual information, Bhattacharyya parameter, probability of error of a fixed code, optimal probability of error of a given rate and blocklength, channel sums and products, etc ...) is studied in \cite{RajContTop}. We also study the polarization process of Ar{\i}kan \cite{Arikan} and its convergence under various topologies in \cite{RajPolarConvTop}.

\section{Preliminaries}

In this section, we recall basic definitions and well known theorems in general topology. The reader who is already familiar with the basic concepts of topology may skip this section and refer to it later if necessary. Proofs of all non-referenced facts can be found in any standard textbook on general topology (e.g., \cite{kelley1975general}). Definitions and theorems that may not be widely known can be found in Sections \ref{subsecQuotient}, \ref{subsecSequential} and \ref{subsecCompactlyGenerated}.

\subsection{Set-theoretic notations}

For every integer $n>0$, we denote the set $\{1,\ldots,n\}$ as $[n]$.

The set of mappings from a set $A$ to a set $B$ is denoted as $B^A$.

Let $A$ be a subset of $B$. The \emph{indicator mapping} $\mathds{1}_{A,B}:B\rightarrow\{0,1\}$ of $A$ in $B$ is defined as:
$$\mathds{1}_{A,B}(x)=\mathds{1}_{x\in A}=\begin{cases}1\quad&\text{if}\;x\in A,\\0\quad&\text{otherwise}.\end{cases}$$
If the superset $B$ is clear from the context, we simply write $\mathds{1}_A$ to denote the indicator mapping of $A$ in $B$.

The \emph{power set} of $B$ is the set of subsets of $B$. Since every subset of $B$ can be identified with its indicator mapping, we denote the power set of $B$ as $2^B:=\{0,1\}^B$.

A collection $\mathcal{A}\subset 2^B$ of subsets of $B$ is said to be \emph{finer} than another collection $\mathcal{A}'\subset 2^B$ if $\mathcal{A}'\subset\mathcal{A}$. If this is the case, we also say that $\mathcal{A}'$ is \emph{coarser} than $\mathcal{A}$.

Let $(A_i)_{i\in I}$ be a collection of arbitrary sets indexed by $I$. The \emph{disjoint union} of $(A_i)_{i\in I}$ is defined as $\displaystyle \coprod_{i\in I} A_i=\bigcup_{i\in I}(A_i\times\{i\})$. For every $i\in I$, the $i^{th}$-\emph{canonical injection} is the mapping $\phi_i:A_i\rightarrow \displaystyle\coprod_{j\in I} A_j$ defined as $\phi_i(x_i)=(x_i,i)$. If no confusions can arise, we can identify $A_i$ with $A_i\times\{i\}$ through the canonical injection. Therefore, we can see $A_i$ as a subset of $\displaystyle\coprod_{j\in I} A_j$ for every $i\in I$.

A \emph{relation} $R$ on a set $T$ is a subset of $T\times T$. For every $x,y\in T$, we write $x R y$ to denote $(x,y)\in R$.

A relation is said to be \emph{reflexive} if $x R x$ for every $x\in T$. It is \emph{symmetric} if $x R y$ implies $y R x$ for every $x,y\in T$. It is \emph{anti-symmetric} if $x R y$ and $y R x$ imply $x=y$ for every $x,y\in T$. It is \emph{transitive} if $x R y$ and $y R z$ imply $x R z$ for every $x,y,z\in T$.

An \emph{order relation} is a relation that is reflexive, anti-symmetric and transitive. An \emph{equivalence relation} is a relation that is reflexive, symmetric and transitive.

Let $R$ be an equivalence relation on $T$. For every $x\in T$, the set $\hat{x}=\{y\in T:\; x R y\}$ is the \emph{$R$-equivalence class} of $x$. The collection of $R$-equivalence classes, which we denote as $T/R$, forms a partition of $T$, and it is called the \emph{quotient space of $T$ by $R$}. The mapping $\Proj_R:T\rightarrow T/R$ defined as $\Proj_R(x)=\hat{x}$ for every $x\in T$ is the \emph{projection mapping onto $T/R$}.

\subsection{Topological spaces}

A \emph{topological space} is a pair $(T,\mathcal{U})$, where $\mathcal{U}\subset 2^T$ is a collection of subsets of $T$ satisfying:
\begin{itemize}
\item $\o\in\mathcal{U}$ and $T\in\mathcal{U}$.
\item The intersection of a finite collection of members of $\mathcal{U}$ is also a member of $\mathcal{U}$.
\item The union of an arbitrary collection of members of $\mathcal{U}$ is also a member of $\mathcal{U}$.
\end{itemize}
If $(T,\mathcal{U})$ is a topological space, we say that $\mathcal{U}$ is a \emph{topology} on $T$.

The power set $2^T$ of $T$ is clearly a topology. It is called the \emph{discrete topology} on $T$.

If $\mathcal{A}$ is a an arbitrary collection of subsets of $T$, we can construct a topology on $T$ starting from $\mathcal{A}$ as follows:
$$\bigcap_{\substack{\mathcal{A}\subset \mathcal{V}\subset 2^T,\\\mathcal{V}\;\text{is a topology on}\; T}} \mathcal{V}.$$
This is the coarsest topology on $T$ that contains $\mathcal{A}$. It is called the \emph{topology on $T$ generated by $\mathcal{A}$}.

Let $(T,\mathcal{U})$ be a topological space. The subsets of $T$ that are members of $\mathcal{U}$ are called the \emph{open sets} of $T$. Complements of open sets are called \emph{closed sets}. We can easily see that the closed sets satisfy the following:
\begin{itemize}
\item $\o$ and $T$ are closed.
\item The union of a finite collection of closed sets is closed.
\item The intersection of an arbitrary collection of closed sets is closed.
\end{itemize}

Let $A$ be an arbitrary subset of $T$. The \emph{closure} $\cl(A)$ of $A$ is the smallest closed set containing $A$: $$\displaystyle \cl(A)=\bigcap_{\substack{A\subset F\subset T,\\F\;\text{is closed}}} F.$$
The \emph{interior} $A^{\circ}$ of $A$ is the largest open subset of $A$: $$\displaystyle A^{\circ}=\bigcup_{\substack{U\subset A,\\U\;\text{is open}}} U.$$

If $A\subset T$ and $\cl(A)=T$, we say that $A$ is \emph{dense} in $T$.

$(T,\mathcal{U})$ is said to be \emph{separable} if there exists a countable subset of $T$ that is dense in $T$.

A subset $O$ of $T$ is said to be a \emph{neighborhood} of $x\in T$ if there exists an open set $U\in\mathcal{U}$ such that $x\in U\subset O$.

A \emph{neighborhood basis} of $x\in T$ is a collection $\mathcal{O}$ of neighborhoods of $x$ such that for every neighborhood $O$ of $x$, there exists $O'\in\mathcal{O}$ such that $O'\subset O$.

We say that $(T,\mathcal{U})$ is \emph{first-countable} if every point $x\in T$ has a countable neighborhood basis.

A collection of open sets $\mathcal{B}\subset \mathcal{U}$ is said to be a \emph{base} for the topology $\mathcal{U}$ if every open set $U\in\mathcal{U}$ can be written as the union of elements of $\mathcal{B}$.

We say that $(T,\mathcal{U})$ is a \emph{second-countable} space if the topology $\mathcal{U}$ has a countable base.

It is a well known fact that every second-countable space is first-countable and separable.

We say that a sequence $(x_n)_{n\geq 0}$ of elements of $T$ \emph{converges} to $x\in T$ if for every neighborhood $O$ of $x$, there exists $n_0\geq 0$ such that for every $n\geq n_0$, we have $x_n\in O$. We say that $x$ is a \emph{limit} of the sequence $(x_n)_{n\geq 0}$. Note that the limit does not need to be unique if there is no constraint on the topology.

\subsection{Separation axioms}

$(T,\mathcal{U})$ is said to be a \emph{$T_1$-space} if for every $x,y\in T$, there exists an open set $U\in\mathcal{U}$ such that $x\in U$ and $y\notin U$. It is easy to see that $(T,\mathcal{U})$ is $T_1$ if and only if all singletons are closed.

$(T,\mathcal{U})$ is said to be a \emph{Hausdorff} space (or \emph{$T_2$-space}) if for every $x,y\in T$, there exist two open sets $U,V\in\mathcal{U}$ such that $x\in U$, $y\in V$ and $U\cap V=\o$.

If $(T,\mathcal{U})$ is Hausdorff, the limit of every converging sequence is unique.

$(T,\mathcal{U})$ is said to be \emph{regular} if for every $x\in T$ and every closed set $F$ not containing $x$, there exist two open sets $U,V\in\mathcal{U}$ such that $x\in U$, $F\subset V$ and $U\cap V=\o$.

$(T,\mathcal{U})$ is said to be \emph{normal} if for every two disjoint closed sets $A$ and $B$, there exist two open sets $U,V\in\mathcal{U}$ such that $A\subset U$, $B\subset V$ and $U\cap V=\o$.

If $(T,\mathcal{U})$ is normal, disjoint closed sets can be separated by disjoint closed neighborhoods. I.e., for every two disjoint closed sets $A$ and $B$, there exist two open sets $U,U'\in\mathcal{U}$ and two closed sets $K,K'$ such that $A\subset U\subset K$, $B\subset U'\subset K'$ and $K\cap K'=\o$.

$(T,\mathcal{U})$ is said to be a \emph{$T_3$-space} if it is both $T_1$ and regular.

$(T,\mathcal{U})$ is said to be a \emph{$T_4$-space} if it is both $T_1$ and normal.

It is easy to see that $T_4\Rightarrow T_3\Rightarrow T_2\Rightarrow T_1$.

\subsection{Relativization}

If $(T,\mathcal{U})$ is a topological space and $A$ is an arbitrary subset of $T$, then $A$ inherits a topology $\mathcal{U}_A$ from $(T,\mathcal{U})$ as follows:
$$\mathcal{U}_A=\{A\cap U:\;U\in\mathcal{U}\}.$$
It is easy to check that $\mathcal{U}_A$ is a topology on $A$.

If $(T,\mathcal{U})$ is first-countable (respectively second-countable, or Hausdorff), then $(A,\mathcal{U}_A)$ is first-countable (respectively second-countable, or Hausdorff).

If $(T,\mathcal{U})$ is normal and $A$ is closed, then $(A,\mathcal{U}_A)$ is normal.

The union of a countable number of separable subspaces is separable.

\subsection{Continuous mappings}

Let $(T,\mathcal{U})$ and $(S,\mathcal{V})$ be two topological spaces. A mapping $f:T\rightarrow S$ is said to be \emph{continuous} if for every $V\in\mathcal{V}$, we have $f^{-1}(V)\in\mathcal{U}$.

$f:T\rightarrow S$ is an \emph{open} mapping if $f(U)\in\mathcal{V}$ whenever $U\in \mathcal{U}$.
$f:T\rightarrow S$ is a \emph{closed} mapping if $f(F)$ is closed in $S$ whenever $F$ is closed in $T$.

A bijection $f:T\rightarrow S$ is a \emph{homeomorphism} if both $f$ and $f^{-1}$ are continuous. In this case, for every $A\subset T$, $A\in\mathcal{U}$ if and only if $f(A)\in\mathcal{V}$. This means that $(T,\mathcal{U})$ and $(S,\mathcal{V})$ have the same topological structure and share the same topological properties.

\subsection{Compact spaces and sequentially compact spaces}

$(T,\mathcal{U})$ is a \emph{compact} space if every open cover of $T$ admits a finite sub-cover. I.e., if $(U_i)_{i\in I}$ is a collection of open sets such that $\displaystyle T=\bigcup_{i\in I} U_i$ then there exists $n>0$ and $i_1,\ldots,i_n\in I$ such that $\displaystyle T=\bigcup_{j=1}^{n} U_{i_j}$.

If $(T,\mathcal{U})$ is compact, then every closed subset of $T$ is compact (with respect to the inherited topology).

If $f:T\rightarrow S$ is a continuous mapping from a compact space $(T,\mathcal{U})$ to an arbitrary topological space $(S,\mathcal{V})$, then $f(T)$ is compact.

If $A$ is a compact subset of a Hausdorff topological space, then $A$ is closed.

$(T,\mathcal{U})$ is said to be \emph{locally compact} if every point has at least one compact neighborhood. A compact space is automatically locally compact.

If $(T,\mathcal{U})$ is Hausdorff and locally compact, then for every point $x\in T$ and every neighborhood $O$ of $x$, $O$ contains a compact neighborhood of $x$.

A compact Hausdorff space is always normal.

$(T,\mathcal{U})$ is a \emph{$\sigma$-compact} space if it is the union of a countable collection of compact subspaces.

$(T,\mathcal{U})$ is \emph{countably compact} if every countable open cover of $T$ admits a finite sub-cover. This is a weaker condition compared to compactness.

$(T,\mathcal{U})$ is said to be \emph{sequentially compact} if every sequence in $T$ has a converging subsequence. In general, compactness does not imply sequential compactness nor the other way around.

\subsection{Connected spaces}

$(T,\mathcal{U})$ is a \emph{connected} space if it satisfies one of the following equivalent conditions:
\begin{itemize}
\item $T$ cannot be written as the union of two disjoint non-empty open sets.
\item $T$ cannot be written as the union of two disjoint non-empty closed sets.
\item The only subsets of $T$ that are both open and closed are $\o$ and $T$.
\item Every continuous mapping from $T$ to $\{0,1\}$ is constant, where $\{0,1\}$ is endowed with the discrete topology.
\end{itemize}

$(T,\mathcal{U})$ is \emph{path-connected} if every two points of $T$ can be joined by a continuous path. I.e., for every $x,y\in T$, there exists a continuous mapping $f:[0,1]\rightarrow T$ such that $f(0)=x$ and $f(1)=y$, where $[0,1]$ is endowed with the well known Euclidean topology\footnote{See Section \ref{SubsecMetricSpaces} for the definition of the Euclidean metric and its induced topology}.

A path-connected space is connected but the converse is not true in general.

A subset $A$ of $T$ is said to be connected (respectively path-connected) if $(A,\mathcal{U}_A)$ is connected (respectively path-connected).

If $(A_i)_{i\in I}$ is a collection of connected (respectively path-connected) subsets of $T$ such that $\displaystyle \bigcap_{i\in I} A_i\neq \o$, then $\displaystyle \bigcup_{i\in I} A_i$ is connected (respectively path-connected).

\subsection{Product of topological spaces}

Let $\{(T_i,\mathcal{U}_i)\}_{i\in I}$ be a collection of topological spaces indexed by $I$. Let $\displaystyle T=\prod_{i\in I} T_i$ be the product of this collection. For every $j\in I$, the $j^{th}$-\emph{canonical projection} is the mapping $\Proj_j:T\rightarrow T_j$ defined as $\Proj_j\big((x_i)_{i\in I}\big)=x_j$.

The \emph{product topology} $\displaystyle \mathcal{U}:=\bigotimes_{i\in I}\mathcal{U}_i$ on $T$ is the coarsest topology that makes all the canonical projections continuous. It can be shown that $\mathcal{U}$ is generated by the collection of sets of the form $\displaystyle\prod_{i\in I} U_i$, where $U_i\in \mathcal{U}_i$ for all $i\in I$, and $U_i\neq T_i$ for only finitely many $i\in I$.

The product of $T_1$ (respectively, Hausdorff, regular, $T_3$, compact, connected, or path-connected) spaces is $T_1$ (respectively, Hausdorff, regular, $T_3$, compact, connected, or path-connected).

\subsection{Disjoint union}

Let $\{(T_i,\mathcal{U}_i)\}_{i\in I}$ be a collection of topological spaces indexed by $I$. Let $\displaystyle T=\coprod_{i\in I} T_i$ be the disjoint union of this collection. The \emph{disjoint union topology} $\displaystyle \mathcal{U}:=\bigoplus_{i\in I}\mathcal{U}_i$ on $T$ is the finest topology which makes all the canonical injections continuous. It can be shown that $U\in\mathcal{U}$ if and only if $U\cap T_i\in\mathcal{U}_i$ for every $i\in I$.

A mapping $f:T\rightarrow S$ from $(T,\mathcal{U})$ to a topological space $(S,\mathcal{V})$ is continuous if and only if it is continuous on $T_i$ for every $i\in I$.

The disjoint union of $T_1$ (respectively Hausdorff) spaces is $T_1$ (respectively Hausdorff). The disjoint union of two or more non-empty spaces is always disconnected.

Products are distributive with respect to the disjoint union, i.e., if $(S,\mathcal{V})$ is a topological space then $S\times \left(\displaystyle\coprod_{i\in I} T_i\right)=\displaystyle\coprod_{i\in I}\left(S\times T_i\right)$ and $\mathcal{V}\otimes\left(\displaystyle\bigoplus_{i\in I}\mathcal{U}_i\right)=\displaystyle\bigoplus_{i\in I}\left(\mathcal{V}\otimes\mathcal{U}_i\right)$.

\subsection{Quotient topology}

\label{subsecQuotient}
Let $(T,\mathcal{U})$ be a topological space and let $R$ be an equivalence relation on $T$. The \emph{quotient topology} on $T/R$ is the finest topology that makes the projection mapping $\Proj_R$ continuous. It is given by
$$\mathcal{U}/R=\left\{\hat{U}\subset T/R:\;\textstyle\Proj_R^{-1}(\hat{U})\in \mathcal{U}\right\}.$$

\begin{mylem}
\label{lemQuotientFunction}
Let $f:T\rightarrow S$ be a continuous mapping from $(T,\mathcal{U})$ to $(S,\mathcal{V})$. If $f(x)=f(x')$ for every $x,x'\in T$ satisfying $x R x'$, then we can define a \emph{transcendent mapping} $f:T/R\rightarrow S$ such that $f(\hat{x})=f(x')$ for any $x'\in\hat{x}$. $f$ is well defined on $T/R$ . Moreover, $f$ is a continuous mapping from $(T/R,\mathcal{U}/R)$ to $(S,\mathcal{V})$.
\end{mylem}

If $(T,\mathcal{U})$ is compact (respectively, connected, or path-connected), then $(T/R,\mathcal{U}/R)$ is compact (respectively, connected, or path-connected).

$T/R$ is said to be \emph{upper semi-continuous} if for every $\hat{x}\in T/R$ and every open set $U\in\mathcal{U}$ satisfying $\hat{x}\subset U$, there exists an open set $V\in\mathcal{U}$ such that $\hat{x}\subset V\subset U$, and $V$ can be written as the union of members of $T/R$.

The following Lemma characterizes upper semi-continuous quotient spaces:
\begin{mylem}
\label{lemUpperSemiCont}
\cite{kelley1975general} $T/R$ is upper semi-continuous if and only if $\Proj_R$ is a closed mapping.
\end{mylem}

The following theorem is very useful to prove many topological properties for the quotient space:
\begin{mythe}
\label{theQuotientTop}
\cite{kelley1975general} Let $(T,\mathcal{U})$ be a topological space, and let $R$ be an equivalence relation on $T$ such that $T/R$ is upper semi-continuous and $\hat{x}$ is a compact subset of $T$ for every $\hat{x}\in T/R$. If $(T,\mathcal{U})$ is Hausdorff (respectively, regular, locally compact, or second-countable) then $(T/R,\mathcal{U}/R)$ is Hausdorff (respectively, regular, locally compact, or second-countable).
\end{mythe}

\subsection{Metric spaces}

\label{SubsecMetricSpaces}

A \emph{metric space} is a pair $(M,d)$, where $d:M\times M\rightarrow\mathbb{R}^+$ satisfies:
\begin{itemize}
\item $d(x,y)=0$ if and only if $x=y$ for every $x,y\in M$.
\item Symmetry: $d(x,y)=d(y,x)$ for every $x,y\in M$.
\item Triangle inequality: $d(x,z)\leq d(x,y)+d(y,z)$ for every $x,y,z\in M$.
\end{itemize}
If $(M,d)$ is a metric space, we say that $d$ is a \emph{metric} (or \emph{distance}) on $M$.

The \emph{Euclidean metric} on $\mathbb{R}^n$ is defined as $d(x,y)=\sqrt{\displaystyle\sum_{i=1}^n(x_i-y_i)^2}$, where $x=(x_i)_{1\leq i\leq n}$ and $y=(y_i)_{1\leq i\leq n}$.

$\mathbb{R}^n$ is second countable. Moreover, a subset of $\mathbb{R}^n$ is compact if and only if it is bounded and closed.

For every $x\in M$ and every $\epsilon>0$, we define the \emph{open ball} of center $x$ and radius $\epsilon$ as:
$$B_{\epsilon}(x)=\{y\in M:\; d(x,y)<\epsilon\}.$$

The \emph{metric topology} $\mathcal{U}_d$ on $M$ \emph{induced by $d$} is the coarsest topology on $M$ which makes $d$ a continuous mapping from $M\times M$ to $\mathbb{R}^+$. It is generated by all the open balls.

The metric topology is always $T_4$ and first-countable. Moreover, $(M,\mathcal{U}_d)$ is separable if and only if it is second-countable.

Since every metric space is Hausdorff, we can see that every subset of a compact metric space is closed if and only if it is compact.

Every $\sigma$-compact metric space is second-countable.

For metric spaces, compactness and sequential compactness are equivalent.

A function $f:M_1\rightarrow M_2$ from a metric space $(M_1,d_1)$ to a metric space $(M_2,d_2)$ is said to be \emph{uniformly continuous} if for every $\epsilon>0$, there exists $\delta>0$ such that for every $x,x'\in M_1$ satisfying $d_1(x,x')<\delta$ we have $d_2(f(x),f(x'))<\epsilon$.

If $f:M_1\rightarrow M_2$ is a continuous mapping from a compact metric space $(M_1,d_1)$ to an arbitrary metric space $(M_2,d_2)$, then $f$ is uniformly continuous.

A topological space $(T,\mathcal{U})$ is said to be \emph{metrizable} if there exists a metric $d$ on $T$ such that $\mathcal{U}$ is the metric topology on $T$ induced by $d$.

The disjoint union of metrizable spaces is always metrizable.

The following theorem shows that all separable metrizable spaces are characterized topologically:
\begin{mythe}
\label{theMetrSep}
\cite{kelley1975general} A topological space $(T,\mathcal{U})$ is metrizable and separable if and only if it is Hausdorff, regular and second countable.
\end{mythe}

\subsection{Complete metric spaces}

A sequence $(x_n)_{n\geq 0}$ is said to be a Cauchy sequence in $(M,d)$ if for every $\epsilon>0$, there exists $n_0\geq 0$ such that for every $n_1,n_2\geq n_0$ we have $d(x_{n_1},x_{n_2})<\epsilon$.

Every converging sequence is Cauchy, but the converse is not true in general.

A metric space is said to be \emph{complete} if every Cauchy sequence converges in it.

A closed subset of a complete space is always complete.

A complete subspace of an arbitrary metric space is always closed.

Every compact metric space is complete, but the converse is not true in general.

For every metric space $(M,d)$, there exists a superspace $(\overline{M},\overline{d})$ containing $M$ such that:
\begin{itemize}
\item $(\overline{M},\overline{d})$ is complete.
\item $M$ is dense in $(\overline{M},\overline{d})$.
\item $\overline{d}(x,y)=d(x,y)$ for every $x,y\in M$.
\end{itemize}
The space $(\overline{M},\overline{d})$ is said to be \emph{a completion} of $(M,d)$.

\subsection{Polish spaces and Baire spaces}
A topological space $(T,\mathcal{U})$ that is both separable and completely metrizable (i.e., has a metrization that is complete) is called a \emph{Polish space}.

A topological space is said to be a \emph{Baire space} if the intersection of countably many dense open subsets is dense. The following facts can be found in \cite{PolishBaire}:
\begin{itemize}
\item Every completely metrizable space is Baire.
\item Every compact Hausdorff space is Baire.
\item Every open subset of a Baire space is Baire.
\end{itemize}

\subsection{Sequential spaces}

\label{subsecSequential}

Sequential spaces were introduced by Franklin \cite{SequentialSpace} to answer the following question: Assume we know all the converging sequences of a topological space. Is this enough to uniquely determine the topology of the space? \emph{Sequential spaces} are the most general category of spaces for which converging sequences suffice to determine the topology.

Let $(T,\mathcal{U})$ be a topological space. A subset $U\subset T$ is said to be \emph{sequentially open} if for every sequence $(x_n)_{n\geq 0}$ that converges to a point of $U$ lies eventually in $U$, i.e., there exists $n_0\geq 0$ such that $x_n\in U$ for every $n\geq n_0$. Clearly, every open subset of $T$ is sequentially open, but the converse is not true in general.

A topological space $(T,\mathcal{U})$ is said to be \emph{sequential} if every sequentially open subset of $T$ is open.

A mapping $f:T\rightarrow S$ from a sequential topological space $(T,\mathcal{U})$ to an arbitrary topological space $(S,\mathcal{V})$ is continuous if and only if for every sequence $(x_n)_{n\geq 0}$ in $T$ that converges to $x\in T$, the sequence $(f(x_n))_{n\geq 0}$ converges to $f(x)$ in $(S,\mathcal{V})$ \cite{SequentialSpace}.

The following facts were shown in \cite{SequentialSpace}:
\begin{itemize}
\item Every first-countable space is sequential. Therefore, every metrizable space is sequential.
\item The quotient of a sequential space is sequential.
\item All closed and open subsets of a sequential space are sequential.
\item Every countably compact sequential Hausdorff space is sequentially compact.
\item A topological space is sequential if and only if it is the quotient of a metric space.
\end{itemize}

\subsection{Compactly generated spaces}
\label{subsecCompactlyGenerated}

A topological space $(T,\mathcal{U})$ is \emph{compactly generated} if it is Hausdorff and for every subset $F$ of $T$, $F$ is closed if and only if $F\cap K$ is closed for every compact subset $K$ of $T$. Equivalently, $(T,\mathcal{U})$ is \emph{compactly generated} if it is Hausdorff and for every subset $U$ of $T$, $U$ is open in $T$ if and only if $U\cap K$ is open in $K$ for every compact subset $K$ of $T$.

The following facts can be found in \cite{CompactlyGenerated}:
\begin{itemize}
\item All locally compact Hausdorff spaces are compactly generated.
\item All first-countable Hausdorff spaces are compactly generated. Therefore, every metrizable space is compactly generated.
\item A Hausdorff quotient of a compactly generated space is compactly generated.
\item If $(T,\mathcal{U})$ is compactly generated and $(S,\mathcal{V})$ is Hausdorff locally compact, then $(T\times S,\mathcal{U}\otimes\mathcal{V})$ is compactly generated.
\end{itemize}

\section{Measure-theoretic notations}

In this section, we introduce the measure-theoretic notations that we are using. We assume that the reader is familiar with the basic definitions and theorems of measure theory.

\subsection{Probability measures}

If $\mathcal{A}\subset 2^M$ is a collection of subsets of $M$, we denote the $\sigma$-algebra that is \emph{generated} by $\mathcal{A}$ as $\sigma(\mathcal{A})$.

The set of probability measures on $(M,\Sigma)$ is denoted as $\mathcal{P}(M,\Sigma)$. If the $\sigma$-algebra $\Sigma$ is known from the context, we simply write $\mathcal{P}(M)$ to denote the set of probability measures.

If $P\in\mathcal{P}(M,\Sigma)$ and $\{x\}$ is a measurable singleton, we simply write $P(x)$ to denote $P(\{x\})$.

For every $P_1,P_2\in\mathcal{P}(M,\Sigma)$, the \emph{total variation distance} between $P_1$ and $P_2$ is defined as:
$$\|P_1-P_2\|_{TV}=\sup_{A\in\Sigma}|P_1(A)-P_2(A)|.$$
The space $\mathcal{P}(M,\Sigma)$ is a complete metric space under the total variation distance.

\subsection{Probabilities on finite sets}

We always endow finite sets with their finest $\sigma$-algebra, i.e., the power set. In this case, every probability measure is completely determined by its value on singletons, i.e., if $P$ is a measure on a finite set $\mathcal{X}$, then for every $A\subset\mathcal{X}$, we have
$$P(A)=\sum_{x\in A}P(x).$$

If $\mathcal{X}$ is a finite set, we denote the set of probability distributions on $\mathcal{X}$ as $\Delta_{\mathcal{X}}$. Note that $\Delta_{\mathcal{X}}$ is an $(|\mathcal{X}|-1)$-dimensional simplex in $\mathbb{R}^{\mathcal{X}}$. We always endow $\Delta_{\mathcal{X}}$ with the total variation distance and its induced topology. For every $p_1,p_2\in\Delta_{\mathcal{X}}$, we have:
$$\|p_1-p_2\|_{TV}=\frac{1}{2}\sum_{x\in\mathcal{X}}|p_1(x)-p_2(x)|=\frac{1}{2}\|p_1-p_2\|_1.$$
Note that the total variation topology on $\Delta_{\mathcal{X}}$ is the same as the one inherited from the Euclidean topology of $\mathbb{R}^{\mathcal{X}}$ by relativisation. Since $\Delta_{\mathcal{X}}$ is a closed and bounded subset of $\mathbb{R}^{\mathcal{X}}$, it is compact.

\subsection{Borel sets and the support of a measure}

Let $(T,\mathcal{U})$ be a Hausdorff topological space. The \emph{Borel $\sigma$-algebra} of $(T,\mathcal{U})$ is the $\sigma$-algebra generated by $\mathcal{U}$. We denote the Borel $\sigma$-algebra of $(T,\mathcal{U})$ as $\mathcal{B}(T,\mathcal{U})$. If the topology $\mathcal{U}$ is known from the context, we simply write $\mathcal{B}(T)$ to denote the Borel $\sigma$-algebra. The sets in $\mathcal{B}(T)$ are called the \emph{Borel sets} of $T$.

The \emph{support} of a probability measure $P\in\mathcal{P}(T,\mathcal{B}(T))$ is the set of all points $x\in T$ for which every neighborhood has a strictly positive measure:
$$\supp(P)=\{x\in T:\;P(O)>0\;\text{for every neighborhood}\;O\;\text{of}\;x\}.$$
If $P$ is a probability measure on a Polish space, then $P\big(T\setminus\supp(P)\big)=0$.

\subsection{Convergence of probability measures and the weak-$\ast$ topology}

\label{subsecConvMeasuresTopology}

We have many notions of convergence of probability measures. If the measurable space does not have a topological structure, we have two notions of convergence:
\begin{itemize}
\item The \emph{total-variation convergence}: we say that a sequence $(P_n)_{n\geq 0}$ of probability measures in $\mathcal{P}(M,\Sigma)$ converges in total variation to $P\in \mathcal{P}(M,\Sigma)$ if and only if $\displaystyle\lim_{n\to\infty}\|P_n-P\|_{TV}=0$.
\item The \emph{strong convergence}: we say that a sequence $(P_n)_{n\geq 0}$ in $\mathcal{P}(M,\Sigma)$ strongly converges to $P\in \mathcal{P}(M,\Sigma)$ if and only if $\displaystyle\lim_{n\to\infty}P_n(A)=P(A)$ for every $A\in \Sigma$.
\end{itemize}

Clearly, total-variation convergence implies strong convergence. The converse is not true in general. However, if we are working in the Borel $\sigma$-algebra of a Polish space $T$ and $(P_n)_{n\geq 0}$ strongly converges to a finitely supported probability measure $P$, then
\begin{align*}
\|P_n-P\|_{TV}&= \sup_{B\in\mathcal{B}(T)}|P_n(B)-P(B)| \\
&\leq \sup_{B\in\mathcal{B}(T)}\bigg(\big|P_n\big(B\setminus \supp(P)\big)-P\big(B\setminus \supp(P)\big)\big|+\sum_{x\in\supp(P)}|P_n(x)-P(x)|\bigg)\\
&= \sup_{B\in\mathcal{B}(T)}\bigg(\big|P_n\big(B\setminus \supp(P)\big)\big|+\sum_{x\in\supp(P)}|P_n(x)-P(x)|\bigg)\\
&\leq |P_n\big(T\setminus \supp(P)\big)\big|+\sum_{x\in\supp(P)}|P_n(x)-P(x)|\\
&= |P_n\big(T\setminus \supp(P)\big)-P\big(T\setminus \supp(P)\big)\big|+\sum_{x\in\supp(P)}|P_n(x)-P(x)|\stackrel{n\to\infty}{\longrightarrow} 0,
\end{align*}
which implies that $(P_n)_{n\geq 0}$ also converges to $P$ in total variation. Therefore, in a Polish space, total variation convergence and strong convergence to finitely supported probability measures are equivalent.

Let $(T,\mathcal{U})$ be a Hausdorff topological space. We say that a sequence $(P_n)_{n\geq 0}$ of probability measures in $\mathcal{P}(T,\mathcal{B}(T))$ weakly-$\ast$ converges to $P\in \mathcal{P}(T,\mathcal{B}(T))$ if and only if for every bounded and continuous function $f$ from $T$ to $\mathbb{R}$, we have $$\displaystyle\lim_{n\to\infty}\int_{T}f\cdot dP_n=\int_{T}f\cdot dP.$$
Note that many authors call this notion ``weak convergence" rather than weak-$\ast$ convergence. We will refrain from using the term ``weak convergence" in order to be consistent with the functional analysis notation.

The \emph{weak-$\ast$ topology} on $\mathcal{P}(T,\mathcal{B}(T))$ is the coarsest topology which makes the mappings $$\displaystyle P\rightarrow\int_{\Delta_{\mathcal{X}}}f\cdot dP$$ continuous over $\mathcal{P}(T,\mathcal{B}(T))$, for every bounded and continuous function $f$ from $T$ to $\mathbb{R}$.

\subsection{Metrization of the weak-$\ast$ topology}

If $(T,\mathcal{U})$ is a Polish space (i.e., separable and completely metrizable), the weak-$\ast$ topology on $\mathcal{P}(T,\mathcal{B}(T))$ is also Polish \cite{WassersteinMetric}. There are many known metrizations for the weak-$\ast$ topology. One metrization that is particularly convenient for us is the Wasserstein metric.

The $1^{st}$-Wasserstein distance on $\mathcal{P}(T,\mathcal{B}(T))$ is defined as
$$W_1(P,P')=\inf_{\gamma\in\Gamma(P,P')}\int_{T\times T} d(x,x') \cdot d\gamma(x,x'),$$
where $\Gamma(P,P')$ is the collection of all probability measures on $T\times T$ with marginals $P$ and $P'$ on the first and second factors respectively, and $d$ is a metric on $T$ that induces the topology $\mathcal{U}$. $\Gamma(P,P')$ is also called the set of \emph{couplings} of $P$ and $P'$.

If $d$ is bounded and $(T,d)$ is separable and complete, then $W_1$ metrizes the weak-$\ast$ topology \cite{WassersteinMetric}. If $(T,\mathcal{U})$ is compact, then $(\mathcal{P}(T),W_1)$ is also compact \cite{WassersteinMetric}.

If $\displaystyle D=\sup_{x,x'\in T} d(x,x')$ is the diameter of $(T,d)$, then $W_1(P,P')\leq D\|P-P'\|_{TV}$ \cite{WassersteinMetric}. In other words, the Wasserstein metric is controlled by total variation.

\subsection{Meta-probability measures}

\label{subsecMetaProbDef}

Let $\mathcal{X}$ be a finite set. A \emph{meta-probability measure} on $\mathcal{X}$ is a probability measure on the Borel sets of $\Delta_{\mathcal{X}}$. It is called a meta-probability measure because it is a probability measure on the space of probability distributions on $\mathcal{X}$.

We denote the set of meta-probability measures on $\mathcal{X}$ as $\mathcal{MP}(\mathcal{X})$. Clearly, $\mathcal{MP}(\mathcal{X})=\mathcal{P}(\Delta_{\mathcal{X}})$.

A meta-probability measure ${\MP}$ on $\mathcal{X}$ is said to be \emph{balanced} if it satisfies
$$\int_{\Delta_{\mathcal{X}}}p\cdot d{\MP}(p)=\pi_{\mathcal{X}},$$
where $\pi_{\mathcal{X}}$ is the uniform probability distributions on $\mathcal{X}$.

We denote the set of all balanced meta-probability measures on $\mathcal{X}$ as $\mathcal{MP}_b(\mathcal{X})$. The set of all balanced and finitely supported meta-probability measures on $\mathcal{X}$ is denoted as $\mathcal{MP}_{bf}(\mathcal{X})$.

\section{The space of channels from $\mathcal{X}$ to $\mathcal{Y}$}

A discrete memoryless channel $W$ is a 3-tuple $W=(\mathcal{X},\;\mathcal{Y},\;p_W)$ where $\mathcal{X}$ is a finite set that is called the \emph{input alphabet} of $W$, $\mathcal{Y}$ is a finite set that is called the \emph{output alphabet} of $W$, and $p_W:\mathcal{X} \times \mathcal{Y} \rightarrow [0,1]$  is a function satisfying $\forall x\in\mathcal{X},\;\displaystyle\sum_{y\in\mathcal{Y}} p_W(x,y)=1$.

For every $(x,y)\in\mathcal{X}\times\mathcal{Y}$, we denote $p_W(x,y)$ as $W(y|x)$, which we interpret as the conditional probability of receiving $y$ at the output, given that $x$ is the input.

Let $\DMC_{\mathcal{X},\mathcal{Y}}$ be the set of all channels having $\mathcal{X}$ as input alphabet and $\mathcal{Y}$ as output alphabet.

For every $W,W'\in\DMC_{\mathcal{X},\mathcal{Y}}$, define the distance between $W$ and $W'$ as follows:
$$d_{\mathcal{X},\mathcal{Y}}(W,W')=\frac{1}{2} \max_{x\in\mathcal{X}}\sum_{y\in\mathcal{Y}}|W'(y|x)-W(y|x)|.$$

It is easy to check the following properties of $d_{\mathcal{X},\mathcal{Y}}$:
\begin{itemize}
\item $0\leq d_{\mathcal{X},\mathcal{Y}}(W,W')\leq 1$.
\item $d_{\mathcal{X},\mathcal{Y}}:\DMC_{\mathcal{X},\mathcal{Y}}\times\DMC_{\mathcal{X},\mathcal{Y}}\rightarrow\mathbb{R}^+$ is a metric distance on $\DMC_{\mathcal{X},\mathcal{Y}}$.
\end{itemize}
Throughout this paper, we always associate the space $\DMC_{\mathcal{X},\mathcal{Y}}$ with the metric distance $d_{\mathcal{X},\mathcal{Y}}$ and the metric topology $\mathcal{T}_{\mathcal{X},\mathcal{Y}}$ induced by it.

For every $x\in \mathcal{X}$, the mapping $y\rightarrow W(y|x)$ is a probability distributions on $\mathcal{Y}$. Therefore, every channel $W$ can be seen as a collection of probability distributions on $\mathcal{Y}$, and the collection is indexed by $x\in\mathcal{X}$. This allows us to identify the space $\DMC_{\mathcal{X},\mathcal{Y}}$ with $\displaystyle(\Delta_{\mathcal{Y}})^{\mathcal{X}}=\prod_{x\in\mathcal{X}}\Delta_{\mathcal{Y}}$, where $\Delta_{\mathcal{Y}}$ is the set of probability distributions on $\mathcal{Y}$. It is easy to see that the topology given by the metric $d_{\mathcal{X},\mathcal{Y}}$ on $\DMC_{\mathcal{X},\mathcal{Y}}$ is the same as the product topology on $(\Delta_{\mathcal{Y}})^{\mathcal{X}}$, which is also the same as the topology inherited from the Euclidean topology of $\mathbb{R}^{\mathcal{X}\times\mathcal{Y}}$ by relativization.

It is known that $\Delta_{\mathcal{Y}}$ is a closed and bounded subset of $\mathbb{R}^{\mathcal{Y}}$. Therefore, $\Delta_{\mathcal{Y}}$ is compact, which implies that $(\Delta_{\mathcal{Y}})^{\mathcal{X}}$ is compact. We conclude that the metric space $\DMC_{\mathcal{X},\mathcal{Y}}\equiv (\Delta_{\mathcal{Y}})^{\mathcal{X}}$ is compact. Moreover, since $\Delta_{\mathcal{Y}}$ a convex subset of $\mathbb{R}^{\mathcal{Y}}$, it is path-connected, hence $\DMC_{\mathcal{X},\mathcal{Y}}\equiv (\Delta_{\mathcal{Y}})^{\mathcal{X}}$ is path-connected as well.

If $W\in\DMC_{\mathcal{X},\mathcal{Y}}$ and $V\in\DMC_{\mathcal{Y},\mathcal{Z}}$, we define the composition $V\circ W\in\DMC_{\mathcal{X},\mathcal{Z}}$ of $W$ and $V$ as follows:
$$(V\circ W)(z|x)=\sum_{y\in\mathcal{Y}}V(z|y)W(y|x),\;\;\forall x\in\mathcal{X},\;\forall z\in\mathcal{Z}.$$
It is easy to see that the mapping $(W,V)\rightarrow V\circ W$ from $\DMC_{\mathcal{X},\mathcal{Y}}\times \DMC_{\mathcal{Y},\mathcal{Z}}$ to $\DMC_{\mathcal{X},\mathcal{Z}}$ is continuous.

For every mapping $f:\mathcal{X}\rightarrow \mathcal{Y}$, define the \emph{deterministic channel} $D_f\in\DMC_{\mathcal{X},\mathcal{Y}}$ as follows:
$$D_f(y|x)=\begin{cases}1\quad&\text{if}\;y=f(x),\\0\quad&\text{otherwise}.\end{cases}$$
It is easy to see that if $f:\mathcal{X}\rightarrow \mathcal{Y}$ and $g:\mathcal{Y}\rightarrow \mathcal{Z}$, then $D_g\circ D_f=D_{g\circ f}$.

\section{Equivalent channels and their representation}

\label{secMetaProb}

Let $W\in \DMC_{\mathcal{X},\mathcal{Y}}$ and $W'\in \DMC_{\mathcal{X},\mathcal{Z}}$ be two channels having the same input alphabet. We say that $W'$ is \emph{degraded} from $W$ if there exists a channel $V\in \DMC_{\mathcal{Y},\mathcal{Z}}$ such that $W'=V\circ W$. $W$ and $W'$ are said to be \emph{equivalent} if each one is degraded from the other. In the rest of this section, we describe one way to check whether two given channels are equivalent.

Let $\Delta_{\mathcal{X}}$ and $\Delta_{\mathcal{Y}}$ be the space of probability distributions on $\mathcal{X}$ and $\mathcal{Y}$ respectively. Define $P_W^o\in\Delta_{\mathcal{Y}}$ as $$\displaystyle P_W^o(y)=\frac{1}{|\mathcal{X}|}\sum_{x\in\mathcal{X}}W(y|x),\;\;\forall y\in\mathcal{Y}.$$ This can be interpreted as the probability distribution of the output when the input is uniformly distributed in $\mathcal{X}$. The \emph{image} of $W$ is the set of output-symbols $y\in\mathcal{Y}$ having strictly positive probabilities:
$$\Imag(W)=\{y\in\mathcal{Y}:\; P_W^o(y)>0\}.$$

For every $y\in\Imag(W)$, define $W^{-1}_y\in\Delta_{\mathcal{X}}$ as follows:
$$\; W^{-1}_y(x)=\displaystyle\frac{W(y|x)}{|\mathcal{X}|P_W^o(y)},\;\;\forall x\in\mathcal{X}.$$
$W^{-1}_y(x)$ can be interpreted as the posterior probability of $x$, given that the output is $y$, and assuming a uniform prior distribution on the input. In other words, if $X$ is a random variable uniformly distributed in $\mathcal{X}$ and $Y$ is the output of the channel $W$ when $X$ is the input, then:
\begin{itemize}
\item $P_W^o(y)=P_Y(y)$ for every $y\in\mathcal{Y}$.
\item $W^{-1}_y(x)=\displaystyle P_{X|Y}(x|y)$ for every $(x,y)\in\mathcal{X}\times\Imag(W)$.
\end{itemize}

Let $(x,y)\in\mathcal{X}\times\mathcal{Y}$. If $P_W^o(y)=P_Y(y)>0$, we have
$$W(y|x)=P_{Y|X}(y|x)=\frac{P_{X,Y}(x,y)}{P_X(x)}=|\mathcal{X}| P_Y(y)P_{X|Y}(x|y)=|\mathcal{X}|P_W^o(y)W_y^{-1}(x).$$
On the other hand, if $P_W^o(y)=0$, then we must have $W(y|x)=0$. We conclude that $P_W^o$ and the collection $\{W_y^{-1}\}_{y\in \Imag(W)}$ uniquely determine $W$.

The \emph{Blackwell measure}\footnote{In an earlier version of this work, I called $\MP_W$ the \emph{posterior meta-probability distribution} of $W$. Maxim Raginsky thankfully brought to my attention the fact that $\MP_W$ is called \emph{Blackwell measure}.} (denoted ${\MP}_W$) of $W$ is a probability distribution on $\Delta_{\mathcal{X}}$ having masses $P_W^o(y)$ on $W^{-1}_y$ for each $y\in\Imag(W)$:
$${\textstyle{\MP}_W}(B)=\sum_{\substack{y\in\Imag(W),\\W^{-1}_y\in B}}P_W^o(y),\;\;\forall B\in\mathcal{B}(\Delta_{\mathcal{X}}).$$
Another way to express $\MP_W$ is as follows:
$${\MP}_W=\sum_{y\in\Imag(W)}P_W^o(y)\cdot\delta_{W_y^{-1}},$$
where $\delta_{W_y^{-1}}$ is a Dirac measure centered at $W_y^{-1}\in\Delta_{\mathcal{X}}$.

${\MP}_W$ can be interpreted as follows: after the receiver obtains the output of the channel, he can compute the posterior probabilities of the input as the conditional probability distribution of the input given the output symbol that he received. But before receiving the output symbol, the receiver does not know what he we will receive. He just has different probabilities for different possible output symbols. Therefore, the posterior probability distribution that will be computed by the receiver is itself random, and so we need a meta-probability measure to describe it. ${\MP}_W$ is exactly this meta-probability measure.

Since $\Imag(W)$ is finite, the support of ${\MP}_W$ is finite and it consists of all points in $\Delta_{\mathcal{X}}$ having strictly positive mass:
$$\supp({\textstyle{\MP}_W})=\{p\in\Delta_{\mathcal{X}}:\;{\textstyle{\MP}_W(p)>0}\}.$$

The \emph{rank} of $W$ is the size of the support of its Blackwell measure: $$\rank(W)=|\supp({\MP}_W)|.$$

Notice that for every $x\in\mathcal{X}$, we have
\begin{align*}
\int_{\Delta_{\mathcal{X}}}p(x)\cdot d{\textstyle{\MP}_W}(p)&=\sum_{p\in\supp({\MP}_W)}{\textstyle{\MP}_W}(p)\cdot p(x)=\sum_{y\in\Imag(W)}P_W^o(y)W^{-1}_y(x)\\
&=\sum_{y\in\Imag(W)}\frac{1}{|\mathcal{X}|}W(y|x)\stackrel{(a)}{=}\sum_{y\in\mathcal{Y}}\frac{1}{|\mathcal{X}|}W(y|x)=\frac{1}{|\mathcal{X}|},
\end{align*}
where (a) follows from the fact that $W(y|x)=0$ for every $y\notin \Imag(W)$. Therefore, we can write
\begin{equation}
\int_{\Delta_{\mathcal{X}}}p\cdot d{\textstyle{\MP}_W}(p)=\pi_{\mathcal{X}},
\end{equation}
where $\pi_{\mathcal{X}}$ is the uniform probability distribution on $\mathcal{X}$. This shows that ${\MP}_W$ is a balanced meta-probability measure.

The following proposition characterizes the Blackwell measures of DMCs with input alphabet $\mathcal{X}$:
\begin{myprop}
\label{propCharacPostMetaProb}
\cite{torgersen} A meta-probability measure ${\MP}$ on $\mathcal{X}$ is the Blackwell measure of some DMC with input alphabet $\mathcal{X}$ if and only if $\MP$ is balanced and finitely supported.
\end{myprop}
\begin{proof}
This proposition is known \cite{torgersen}, but we provide a proof for completeness.

The above discussion shows that if ${\MP}$ is the Blackwell measure of some channel with input alphabet $\mathcal{X}$, then it is balanced and finitely supported.

Now assume that ${\MP}$ is balanced and finitely supported, and let $\mathcal{Y}=\supp({\MP})$. Define the channel $W\in\DMC_{\mathcal{X},\mathcal{Y}}$ as $W(p|x)=|\mathcal{X}|{\MP}(p) p(x)$ for every $x\in\mathcal{X}$ and every $p\in\mathcal{Y}=\supp({\MP})$. For every $x\in\mathcal{X}$, we have:
$$\sum_{p\in\mathcal{Y}}W(p|x)=\sum_{p\in\supp({\MP})}|\mathcal{X}|p(x){\MP}(p)=|\mathcal{X}|\int_{\Delta_\mathcal{X}}p(x)\cdot d{\MP}(p)=|\mathcal{X}|\pi_{\mathcal{X}}(x)=1.$$
Therefore, $W$ is a valid channel. For every $p\in\mathcal{Y}$, we have
$$P_W^o(p)=\frac{1}{|\mathcal{X}|}\sum_{x\in \mathcal{X}}W(p|x)=\frac{1}{|\mathcal{X}|}\sum_{x\in \mathcal{X}}|\mathcal{X}|p(x){\MP}(p)=\sum_{x\in \mathcal{X}}p(x){\MP}(p)={\MP}(p)>0,$$
which implies that $\Imag(W)=\mathcal{Y}$. For every $(x,p)\in\mathcal{X}\times\mathcal{Y}$ we have:
$$W_p^{-1}(x)=\frac{W(p|x)}{|\mathcal{X}|P_W^o(p)}=\frac{|\mathcal{X}|{\MP}(p)p(x)}{{|\mathcal{X}|\MP}(p)}=p(x).$$
Therefore, $W_p^{-1}=p$ for every $p\in\mathcal{Y}$. For every Borel subset $B$ of $\Delta_{\mathcal{X}}$, we have:
$${\textstyle{\MP}_W}(B)=\sum_{\substack{p\in\Imag(W),\\W^{-1}_p\in B}}P_W^o(p)=\sum_{\substack{p\in\supp({\MP}),\\p\in B}}{\MP}(p)={\MP}(B).$$
We conclude that ${\MP}_W={\MP}$.
\end{proof}

\vspace*{3mm}
In \cite{RichardsonUrbanke}, equivalent representations for binary memoryless symmetric (BMS) channels (namely $L$, $D$ and $G$ densities) were provided. A necessary and sufficient condition for the degradation of a BMS channel $W'$ with respect to another BMS channel $W$ was given in \cite{RichardsonUrbanke} in terms of the $|D|$-densities of $W$ and $W'$. It immediately follows from this condition that two BMS channels are equivalent if and only if they have the same $|D|$-densities. One can deduce from this that two BMS channels (with finite output alphabets) are equivalent if and only if they have the same Blackwell measure. The following proposition shows that this is also true for channels with arbitrary (but finite) input and output alphabets:

\begin{myprop}
\label{propCharacOutputEquiv}
\cite{torgersen} Let $\mathcal{X},\mathcal{Y}$ and $\mathcal{Z}$ be three finite sets. Two channels $W\in\DMC_{\mathcal{X},\mathcal{Y}}$ and $W'\in\DMC_{\mathcal{X},\mathcal{Z}}$ are equivalent if and only if ${\MP}_W={\MP}_{W'}$.
\end{myprop}
\begin{proof}
This proposition is known \cite{torgersen}, but we provide a proof in Appendix \ref{appCharacOutputEquiv} for completeness.
\end{proof}

\begin{mycor}
\label{corNonEquiv}
If $W\in\DMC_{\mathcal{X},\mathcal{Y}}$ and $\rank(W)>|\mathcal{Z}|$, then $W$ is not equivalent to any channel in $\DMC_{\mathcal{X},\mathcal{Z}}$.
\end{mycor}
\begin{proof}
Since $\rank(W')=|\supp({\MP}_{W'})|\leq|\mathcal{Z}|$ for every $W'\in\DMC_{\mathcal{X},\mathcal{Z}}$, it is impossible for $W$ to be equivalent to any channel $W'$ in $\DMC_{\mathcal{X},\mathcal{Z}}$.
\end{proof}

\begin{mycor}
If $|\mathcal{X}|=1$, all channels with input alphabet $\mathcal{X}$ are equivalent.
\end{mycor}

\section{Space of equivalent channels from $\mathcal{X}$ to $\mathcal{Y}$}

\subsection{The $\DMC_{\mathcal{X},\mathcal{Y}}^{(o)}$ space}

\label{subsecDMCXYo}
Let $\mathcal{X}$ and $\mathcal{Y}$ be two finite sets. Define the relation $R_{\mathcal{X},\mathcal{Y}}^{(o)}$ on $\DMC_{\mathcal{X},\mathcal{Y}}$ as follows:
$$\forall W,W'\in \textstyle\DMC_{\mathcal{X},\mathcal{Y}},\;\; WR_{\mathcal{X},\mathcal{Y}}^{(o)}W'\;\;\Leftrightarrow\;\;W\;\text{is equivalent to}\;W'.$$

It is easy to see that $R_{\mathcal{X},\mathcal{Y}}^{(o)}$ is an equivalence relation on $\DMC_{\mathcal{X},\mathcal{Y}}$.

\begin{mydef}
The space of equivalent channels with input alphabet $\mathcal{X}$ and output alphabet $\mathcal{Y}$ is the quotient of the space of channels from $\mathcal{X}$ to $\mathcal{Y}$ by the equivalence relation:
$$\textstyle\DMC_{\mathcal{X},\mathcal{Y}}^{(o)}=\DMC_{\mathcal{X},\mathcal{Y}}/R_{\mathcal{X},\mathcal{Y}}^{(o)}.$$
We define the topology $\mathcal{T}_{\mathcal{X},\mathcal{Y}}^{(o)}$ on $\DMC_{\mathcal{X},\mathcal{Y}}^{(o)}$ as the quotient topology $\mathcal{T}_{\mathcal{X},\mathcal{Y}}/R_{\mathcal{X},\mathcal{Y}}^{(o)}$.
\end{mydef}

Unless we explicitly state otherwise, we always associate $\DMC_{\mathcal{X},\mathcal{Y}}^{(o)}$ with the quotient topology $\mathcal{T}_{\mathcal{X},\mathcal{Y}}^{(o)}$ .

For every $W\in\DMC_{\mathcal{X},\mathcal{Y}}$, let $\hat{W}\in\DMC_{\mathcal{X},\mathcal{Y}}^{(o)}$ be the $R_{\mathcal{X},\mathcal{Y}}^{(o)}$-equivalence class containing $W$.

\begin{mylem}
The projection mapping $\Proj:\DMC_{\mathcal{X},\mathcal{Y}}\rightarrow\DMC_{\mathcal{X},\mathcal{Y}}^{(o)}$ defined as $\Proj(W)=\hat{W}$ is continuous and closed.
\label{lemProjContClosed}
\end{mylem}
\begin{proof}
See Appendix \ref{appProjContClosed}.
\end{proof}

\begin{mycor}
For every $W\in\DMC_{\mathcal{X},\mathcal{Y}}$, $\hat{W}$ is a compact subset of $\DMC_{\mathcal{X},\mathcal{Y}}$.
\label{corEqClassCompact}
\end{mycor}
\begin{proof}
Since $\DMC_{\mathcal{X},\mathcal{Y}}$ is compact, then $\DMC_{\mathcal{X},\mathcal{Y}}^{(o)}=\DMC_{\mathcal{X},\mathcal{Y}}/R_{\mathcal{X},\mathcal{Y}}^{(o)}$ is compact as well.

Let $\Proj: \DMC_{\mathcal{X},\mathcal{Y}}\rightarrow\DMC_{\mathcal{X},\mathcal{Y}}^{(o)}$ be as in Lemma \ref{lemProjContClosed}. Since $\Proj$ is closed and since $\{W\}$ is closed in $\DMC_{\mathcal{X},\mathcal{Y}}$, $\{\hat{W}\}=\Proj(\{W\})$ is closed in $\DMC_{\mathcal{X},\mathcal{Y}}^{(o)}$. Therefore, $\hat{W}=\Proj^{-1}(\{\hat{W}\})$ is closed in $\DMC_{\mathcal{X},\mathcal{Y}}$ because $\Proj$ is continuous. Now since $\DMC_{\mathcal{X},\mathcal{Y}}$ is compact, $\hat{W}$ is compact as well.
\end{proof}

\begin{mythe}
\label{theDMCXYo}
$\DMC_{\mathcal{X},\mathcal{Y}}^{(o)}$ is a compact, path-connected and metrizable space.
\end{mythe}
\begin{proof}
Since $\DMC_{\mathcal{X},\mathcal{Y}}$ is compact and path-connected, $\DMC_{\mathcal{X},\mathcal{Y}}^{(o)}=\DMC_{\mathcal{X},\mathcal{Y}}/R_{\mathcal{X},\mathcal{Y}}^{(o)}$ is compact and path-connected as well.

Since the projection map $\Proj$ of Lemma \ref{lemProjContClosed} is closed, Lemma \ref{lemUpperSemiCont} implies that the quotient space $\DMC_{\mathcal{X},\mathcal{Y}}^{(o)}=\DMC_{\mathcal{X},\mathcal{Y}}/R_{\mathcal{X},\mathcal{Y}}^{(o)}$ is upper semi-continuous. On the other hand, Corollary \ref{corEqClassCompact} shows that all the members of $\DMC_{\mathcal{X},\mathcal{Y}}^{(o)}$ are compact in $\DMC_{\mathcal{X},\mathcal{Y}}$. Therefore, the conditions of Theorem \ref{theQuotientTop} are satisfied.

Since $\DMC_{\mathcal{X},\mathcal{Y}}$ is a metric space, it is Hausdorff and regular. Moreover, since it can be seen as a subspace of $\mathbb{R}^{|\mathcal{X}|\cdot|\mathcal{Y}|}$, it is also second-countable. By Theorem \ref{theQuotientTop} we get that $\DMC_{\mathcal{X},\mathcal{Y}}^{(o)}=\DMC_{\mathcal{X},\mathcal{Y}}/R_{\mathcal{X},\mathcal{Y}}^{(o)}$ is Hausdorff, regular and second-countable, and from Theorem \ref{theMetrSep} we conclude that $\DMC_{\mathcal{X},\mathcal{Y}}^{(o)}$ is separable and metrizable.
\end{proof}

\subsection{Canonical embedding and canonical identification}

Let $\mathcal{X},\mathcal{Y}_1$ and $\mathcal{Y}_2$ be three finite sets such that $|\mathcal{Y}_1|\leq |\mathcal{Y}_2|$. We will show that there is a canonical embedding from $\DMC_{\mathcal{X},\mathcal{Y}_1}^{(o)}$ to $\DMC_{\mathcal{X},\mathcal{Y}_2}^{(o)}$. In other words, there exists an explicitly constructable compact subset $A$ of $\DMC_{\mathcal{X},\mathcal{Y}_2}^{(o)}$ such that $A$ is homeomorphic to $\DMC_{\mathcal{X},\mathcal{Y}_1}^{(o)}$. $A$ and the homeomorphism depend only on $\mathcal{X},\mathcal{Y}_1$ and $\mathcal{Y}_2$ (this is why we say that they are canonical). Moreover, we can show that $A$ depends only on $|\mathcal{Y}_1|$, $\mathcal{X}$ and $\mathcal{Y}_2$.

\begin{mylem}
\label{lemEquivChannelInj}
For every $W\in\DMC_{\mathcal{X},\mathcal{Y}_1}$ and every injection $f$ from $\mathcal{Y}_1$ to $\mathcal{Y}_2$, $W$ is equivalent to $D_f\circ W$.
\end{mylem}
\begin{proof}
Clearly $D_f\circ W$ is degraded from $W$. Now let $f'$ be any mapping from $\mathcal{Y}_2$ to $\mathcal{Y}_1$ such that $f'(f(y_1))=y_1$ for every $y_1\in\mathcal{Y}_1$. We have $W=(D_{f'}\circ D_f)\circ W=D_{f'}\circ(D_f\circ W)$, and so $W$ is also degraded from $D_f\circ W$.
\end{proof}

\begin{mycor}
\label{corEquivChannelInj}
For every $W,W'\in\DMC_{\mathcal{X},\mathcal{Y}_1}$ and every two injections $f,g$ from $\mathcal{Y}_1$ to $\mathcal{Y}_2$, we have:
$$W R_{\mathcal{X},\mathcal{Y}_1}^{(o)} W'\;\;\Leftrightarrow\;\; (D_f \circ W)R_{\mathcal{X},\mathcal{Y}_2}^{(o)}(D_{g} \circ W').$$
\end{mycor}
\begin{proof}
Since $W$ is equivalent to $D_f\circ W$ and $W'$ is equivalent to $D_g\circ W'$, then $W$ is equivalent to $W'$ if and only if $D_f\circ W$ is equivalent to $D_g\circ W'$.
\end{proof}

\vspace*{3mm}

For every $W\in\DMC_{\mathcal{X},\mathcal{Y}_1}$, we denote the $R_{\mathcal{X},\mathcal{Y}_1}^{(o)}$-equivalence class of $W$ as $\hat{W}$, and for every $W\in\DMC_{\mathcal{X},\mathcal{Y}_2}$, we denote the $R_{\mathcal{X},\mathcal{Y}_2}^{(o)}$-equivalence class of $W$ as $\tilde{W}$.

\begin{myprop}
\label{propEmbedOutEquiv}
Let $f:\mathcal{Y}_1\rightarrow\mathcal{Y}_2$ be any fixed injection between $\mathcal{Y}_1$ and $\mathcal{Y}_2$. Define the mapping $F:\DMC_{\mathcal{X},\mathcal{Y}_1}^{(o)}\rightarrow \DMC_{\mathcal{X},\mathcal{Y}_2}^{(o)}$ as
$F(\hat{W})=\widetilde{D_f\circ W'}=\Proj_2(D_f\circ W'),$ where $W'\in \hat{W}$ and $\Proj_2:\DMC_{\mathcal{X},\mathcal{Y}_2}\rightarrow\DMC_{\mathcal{X},\mathcal{Y}_2}^{(o)}$ is the projection onto the $R_{\mathcal{X},\mathcal{Y}_2}^{(o)}$-equivalence classes. We have:
\begin{itemize}
\item $F$ is well defined, i.e., $\Proj_2(D_f\circ W')$ does not depend on $W'\in\hat{W}$.
\item $F$ is a homeomorphism between $\DMC_{\mathcal{X},\mathcal{Y}_1}^{(o)}$ and $F\big(\DMC_{\mathcal{X},\mathcal{Y}_1}^{(o)}\big)$.
\item $F$ does not depend on $f$, i.e., $F$ depends only on $\mathcal{X},\mathcal{Y}_1$ and $\mathcal{Y}_2$.
\item $F\big(\DMC_{\mathcal{X},\mathcal{Y}_1}^{(o)}\big)$ depends only on $|\mathcal{Y}_1|$, $\mathcal{X}$ and $\mathcal{Y}_2$.
\item For every $W'\in\hat{W}$ and every $W''\in F(\hat{W})$, $W'$ is equivalent to $W''$.
\end{itemize}
\end{myprop}
\begin{proof}
Corollary \ref{corEquivChannelInj} implies that $\Proj_2(D_f \circ W)=\Proj_2(D_f \circ W')$ if and only if $W R_{\mathcal{X},\mathcal{Y}_1}^{(o)}W'$. Therefore, $\Proj_2(D_f\circ W')$ does not depend on $W'\in\hat{W}$, hence $F$ is well defined. Corollary \ref{corEquivChannelInj} also shows that $\Proj_2(D_f\circ W')$ does not depend on the particular choice of the injection $f$, hence it is canonical (i.e., it depends only on $\mathcal{X},\mathcal{Y}_1$ and $\mathcal{Y}_2$).

On the other hand, the mapping $W\rightarrow D_f \circ W$ is a continuous mapping from $\DMC_{\mathcal{X},\mathcal{Y}_1}$ to $\DMC_{\mathcal{X},\mathcal{Y}_2}$, and $\Proj_2$ is continuous. Therefore, the mapping $W\rightarrow \Proj_2(D_f \circ W)$ is a continuous mapping from $\DMC_{\mathcal{X},\mathcal{Y}_1}$ to $\DMC_{\mathcal{X},\mathcal{Y}_2}^{(o)}$. Now since $\Proj_2(D_f \circ W)$ depends only on the $R_{\mathcal{X},\mathcal{Y}_1}^{(o)}$-equivalence class $\hat{W}$ of $W$, Lemma \ref{lemQuotientFunction} implies that $F$ is continuous. Moreover, we can see from Corollary \ref{corEquivChannelInj} that $F$ is an injection.

For every closed subset $B$ of $\DMC_{\mathcal{X},\mathcal{Y}_1}^{(o)}$, $B$ is compact since $\DMC_{\mathcal{X},\mathcal{Y}_1}^{(o)}$ is compact, hence $F(B)$ is compact because $F$ is continuous. This implies that $F(B)$ is closed in $\DMC_{\mathcal{X},\mathcal{Y}_2}^{(o)}$ since $\DMC_{\mathcal{X},\mathcal{Y}_2}^{(o)}$ is Hausdorff (as it is metrizable). Therefore, $F$ is a closed mapping.

Now since $F$ is an injection that is both continuous and closed, we can deduce that $F$ is a homeomorphism between $\DMC_{\mathcal{X},\mathcal{Y}_1}^{(o)}$ and $F\big(\DMC_{\mathcal{X},\mathcal{Y}_1}^{(o)}\big)\subset \DMC_{\mathcal{X},\mathcal{Y}_2}^{(o)}$.

We would like now to show that $F\big(\DMC_{\mathcal{X},\mathcal{Y}_1}^{(o)}\big)$ depends only on $|\mathcal{Y}_1|$, $\mathcal{X}$ and $\mathcal{Y}_2$. Let $\mathcal{Y}_1'$ be a finite set such that $|\mathcal{Y}_1|=|\mathcal{Y}_1'|$. For every $W\in \DMC_{\mathcal{X},\mathcal{Y}_1'}$, let $\overline{W}\in\DMC_{\mathcal{X},\mathcal{Y}_1'}^{(o)}$ be the $R_{\mathcal{X},\mathcal{Y}_1'}^{(o)}$-equivalence class of $W$.

Let $g:\mathcal{Y}_1'\rightarrow \mathcal{Y}_1$ be a fixed bijection from $\mathcal{Y}_1'$ to $\mathcal{Y}_1$ and let $f'=f\circ g$. Define $F': \DMC_{\mathcal{X},\mathcal{Y}_1'}^{(o)}\rightarrow \DMC_{\mathcal{X},\mathcal{Y}_2}^{(o)}$ as $F'(\overline{W})=\widetilde{D_{f'}\circ W'}=\Proj_2(D_{f'}\circ W'),$ where $W'\in \overline{W}$. As above, $F'$ is well defined, and it is a homeomorphism from $\DMC_{\mathcal{X},\mathcal{Y}_1'}^{(o)}$ to $F'\big(\DMC_{\mathcal{X},\mathcal{Y}_1'}^{(o)}\big)$. We want to show that $F'\big(\DMC_{\mathcal{X},\mathcal{Y}_1'}^{(o)}\big)=F\big(\DMC_{\mathcal{X},\mathcal{Y}_1}^{(o)}\big)$. For every $\overline{W}\in \DMC_{\mathcal{X},\mathcal{Y}_1'}^{(o)}$, let $W'\in\overline{W}$. We have $$\textstyle F'(\overline{W})= \Proj_2(D_{f'}\circ W')=\Proj_2(D_f\circ(D_g\circ W'))=F\left(\widehat{D_g\circ W'}\right)\in F\big(\DMC_{\mathcal{X},\mathcal{Y}_1}^{(o)}\big).$$ Since this is true for every $\overline{W}\in \DMC_{\mathcal{X},\mathcal{Y}_1'}^{(o)}$, we deduce that $F'\big(\DMC_{\mathcal{X},\mathcal{Y}_1'}^{(o)}\big)\subset F\big(\DMC_{\mathcal{X},\mathcal{Y}_1}^{(o)}\big)$. By exchanging the roles of $\mathcal{Y}_1$ and $\mathcal{Y}_1'$ and using the fact that $f=f'\circ g^{-1}$, we get $F\big(\DMC_{\mathcal{X},\mathcal{Y}_1}^{(o)}\big)\subset F'\big(\DMC_{\mathcal{X},\mathcal{Y}_1'}^{(o)}\big)$. We conclude that $F\big(\DMC_{\mathcal{X},\mathcal{Y}_1}^{(o)}\big)=F'\big(\DMC_{\mathcal{X},\mathcal{Y}_1'}^{(o)}\big)$, which means that $F\big(\DMC_{\mathcal{X},\mathcal{Y}_1}^{(o)}\big)$ depends only on $|\mathcal{Y}_1|$, $\mathcal{X}$ and $\mathcal{Y}_2$.

Finally, for every $W'\in\hat{W}$ and every $W''\in F(\hat{W})=\widetilde{D_f\circ W'}$, $W''$ is equivalent to $D_f\circ W'$ and $D_f\circ W'$ is equivalent to $W'$ (by Lemma \ref{lemEquivChannelInj}), hence $W''$ is equivalent to $W'$.
\end{proof}

\begin{mycor}
\label{corIdentOutEquiv}
If $|\mathcal{Y}_1|=|\mathcal{Y}_2|$, there exists a canonical homeomorphism from $\DMC_{\mathcal{X},\mathcal{Y}_1}^{(o)}$ to $\DMC_{\mathcal{X},\mathcal{Y}_2}^{(o)}$ depending only on $\mathcal{X},\mathcal{Y}_1$ and $\mathcal{Y}_2$.
\end{mycor}
\begin{proof}
Let $f$ be a bijection from $\mathcal{Y}_1$ to $\mathcal{Y}_2$. Define the mapping $F:\DMC_{\mathcal{X},\mathcal{Y}_1}^{(o)}\rightarrow \DMC_{\mathcal{X},\mathcal{Y}_2}^{(o)}$ as
$F(\hat{W})=\widetilde{D_f\circ W'}=\Proj_2(D_f\circ W'),$ where $W'\in \hat{W}$ and $\Proj_2:\DMC_{\mathcal{X},\mathcal{Y}_2}\rightarrow\DMC_{\mathcal{X},\mathcal{Y}_2}^{(o)}$ is the projection onto the $R_{\mathcal{X},\mathcal{Y}_2}^{(o)}$-equivalence classes.

Also, define the mapping $F':\DMC_{\mathcal{X},\mathcal{Y}_2}^{(o)}\rightarrow \DMC_{\mathcal{X},\mathcal{Y}_1}^{(o)}$ as
$F'(\tilde{V})=\widehat{D_{f^{-1}}\circ V'}=\Proj_1(D_{f^{-1}}\circ V'),$ where $V'\in \tilde{V}$ and $\Proj_1:\DMC_{\mathcal{X},\mathcal{Y}_1}\rightarrow\DMC_{\mathcal{X},\mathcal{Y}_1}^{(o)}$ is the projection onto the $R_{\mathcal{X},\mathcal{Y}_1}^{(o)}$-equivalence classes.

Proposition \ref{propEmbedOutEquiv} shows that $F$ and $F'$ are well defined.

For every $W\in\DMC_{\mathcal{X},\mathcal{Y}_1}$, we have:
\begin{align*}
F'(F(\hat{W}))&\stackrel{(a)}{=}F'(\widetilde{D_f\circ W})\stackrel{(b)}{=}\widehat{D_{f^{-1}}\circ (D_f\circ W)}=\hat{W},
\end{align*}
where (a) follows from the fact that $W\in \hat{W}$ and (b) follows from the fact that $D_f\circ W\in \widetilde{D_f\circ W}$.

We can similarly show that $F(F'(\tilde{V}))=\tilde{V}$ for every $\tilde{V}\in\DMC_{\mathcal{X},\mathcal{Y}_2}^{(o)}$. Therefore, both $F$ and $F'$ are bijections. Proposition \ref{propEmbedOutEquiv} now implies that $F$ is a homeomorphism from $\DMC_{\mathcal{X},\mathcal{Y}_1}^{(o)}$ to $F\big(\DMC_{\mathcal{X},\mathcal{Y}_1}^{(o)}\big)=\DMC_{\mathcal{X},\mathcal{Y}_2}^{(o)}$. Moreover, $F$ depends only on $\mathcal{X},\mathcal{Y}_1$ and $\mathcal{Y}_2$.
\end{proof}

\vspace*{3mm}

Corollary \ref{corIdentOutEquiv} allows us to identify $\DMC_{\mathcal{X},\mathcal{Y}_1}^{(o)}$ with $\DMC_{\mathcal{X},\mathcal{Y}_2}^{(o)}$ whenever $|\mathcal{Y}_1|=|\mathcal{Y}_2|$. In the rest of this paper, we identify $\DMC_{\mathcal{X},\mathcal{Y}}^{(o)}$ with $\DMC_{\mathcal{X},[n]}^{(o)}$ through the canonical identification, where $n=|\mathcal{Y}|$ and $[n]=\{1,\ldots,n\}$.

Moreover, for every $1\leq n\leq m$, Proposition \ref{propEmbedOutEquiv} allows us to identify $\DMC_{\mathcal{X},[n]}^{(o)}$ with the canonical subspace of $\DMC_{\mathcal{X},[m]}^{(o)}$ that is homeomorphic to $\DMC_{\mathcal{X},[n]}^{(o)}$. In the rest of this paper, we consider that $\DMC_{\mathcal{X},[n]}^{(o)}$ is a compact subspace of $\DMC_{\mathcal{X},[m]}^{(o)}$.

\vspace*{3mm}

Intuitively, $\DMC_{\mathcal{X},[n]}^{(o)}$ has a ``lower dimension" compared to $\DMC_{\mathcal{X},[m]}^{(o)}$. So one expects that the interior of $\DMC_{\mathcal{X},[n]}^{(o)}$ in $(\DMC_{\mathcal{X},[m]}^{(o)},\mathcal{T}_{\mathcal{X},[m]}^{(o)})$ is empty if $m>n$. The following proposition shows that this intuition is accurate.

\begin{myprop}
\label{propInteriorEmptyDMCXno}
If $|\mathcal{X}|\geq 2$, then for every $1\leq n<m$, the interior of $\DMC_{\mathcal{X},[n]}^{(o)}$ in $(\DMC_{\mathcal{X},[m]}^{(o)},\mathcal{T}_{\mathcal{X},[m]}^{(o)})$ is empty.
\end{myprop}
\begin{proof}
See Appendix \ref{appInteriorEmptyDMCXno}.
\end{proof}

\section{Space of equivalent channels}

\label{secEquivSpace}

We would like to form the space of all equivalent channels having the same input alphabet $\mathcal{X}$. The previous section showed that if $|\mathcal{Y}_1|=|\mathcal{Y}_2|$, there is a canonical identification between $\DMC_{\mathcal{X},\mathcal{Y}_1}^{(o)}$ and $\DMC_{\mathcal{X},\mathcal{Y}_2}^{(o)}$. This shows that if we are interested in equivalent channels, it is sufficient to study the spaces $\DMC_{\mathcal{X},[n]}$ and $\DMC_{\mathcal{X},[n]}^{(o)}$ for every $n\geq 1$. Define the space $$\textstyle\DMC_{\mathcal{X},\ast}={\displaystyle\coprod_{n\geq 1}} \DMC_{\mathcal{X},[n]}.$$
The subscript $\ast$ indicates that the output alphabets of the considered channels are arbitrary but finite.

\vspace*{3mm}

We define the equivalence relation $R_{\mathcal{X},\ast}^{(o)}$ on $\DMC_{\mathcal{X},\ast}$ as follows:
$$\forall W,W'\in \textstyle\DMC_{\mathcal{X},\ast},\;\; WR_{\mathcal{X},\ast}^{(o)}W'\;\;\Leftrightarrow\;\;W\;\text{is equivalent to}\;W'.$$

\begin{mydef}
The space of equivalent channels with input alphabet $\mathcal{X}$ is the quotient of the space of channels with input alphabet $\mathcal{X}$ by the equivalence relation:
$$\textstyle\DMC_{\mathcal{X},\ast}^{(o)}=\DMC_{\mathcal{X},\ast}/R_{\mathcal{X},\ast}^{(o)}.$$
\end{mydef}

For every $n\geq 1$ and every $W,W'\in \DMC_{\mathcal{X},[n]}$, we have $W R_{\mathcal{X},\ast}^{(o)}W'$ if and only if $W R_{\mathcal{X},[n]}^{(o)}W'$ by definition. Therefore, $\DMC_{\mathcal{X},[n]}/R_{\mathcal{X},\ast}^{(o)}$ can be canonically identified with $\DMC_{\mathcal{X},[n]}/R_{\mathcal{X},[n]}^{(o)}=\DMC_{\mathcal{X},[n]}^{(o)}$. But since we identified $\DMC_{\mathcal{X},[n]}^{(o)}$ to its image through the canonical embedding in $\DMC_{\mathcal{X},[m]}^{(o)}$ for every $m\geq n$, we have to make sure that these identifications are consistent with each other.

Remember that for every $m\geq n\geq 1$ and every $W\in\DMC_{\mathcal{X},[n]}$, we identified $\hat{W}$ with $\widetilde{D_f\circ W}$, where $f$ is any injection from $[n]$ to $[m]$, $\hat{W}$ is the $R_{\mathcal{X},[n]}^{(o)}$-equivalence class of $W$ and $\widetilde{D_f\circ W}$ is the $R_{\mathcal{X},[m]}^{(o)}$-equivalence class of $D_f\circ W$. Since $D_f\circ W$ is equivalent to $W$ (by Lemma \ref{lemEquivChannelInj}), $W$ is $R_{\mathcal{X},\ast}^{(o)}$-equivalent to $D_f\circ W$ for every $W\in\DMC_{\mathcal{X},[n]}^{(o)}$. We conclude that identifying $\DMC_{\mathcal{X},[n]}^{(o)}$ to its image through the canonical embedding in $\DMC_{\mathcal{X},[m]}^{(o)}$ for every $m\geq n\geq 1$ is consistent with identifying $\DMC_{\mathcal{X},[n]}/R_{\mathcal{X},\ast}^{(o)}$ to $\DMC_{\mathcal{X},[n]}^{(o)}$ for every $n\geq 1$. Hence, we can write
$$\textstyle\DMC_{\mathcal{X},\ast}^{(o)}={\displaystyle\bigcup_{n\geq 1}}\DMC_{\mathcal{X},[n]}^{(o)}.$$

For any $W,W'\in \DMC_{\mathcal{X},\ast}$, Proposition \ref{propCharacOutputEquiv} shows that $W R_{\mathcal{X},\ast}^{(o)} W'$ if and only if ${\MP}_{W}={\MP}_{W'}$. Therefore, for every $\hat{W}\in \DMC_{\mathcal{X},\ast}^{(o)}$, we can define the Blackwell measure of $\hat{W}$ as ${\MP}_{\hat{W}}:={\MP}_{W'}$ for any $W'\in\hat{W}$. We also define the rank of $\hat{W}$ as $\rank(\hat{W})=|\supp({\MP}_{\hat{W}})|$. Due to Proposition \ref{propCharacOutputEquiv}, we have
$$\textstyle\DMC_{\mathcal{X},[n]}^{(o)}=\{\hat{W}\in\DMC_{\mathcal{X},\ast}^{(o)}:\;\rank(\hat{W})\leq n\}.$$

A subset $A$ of $\DMC_{\mathcal{X},\ast}^{(o)}$ is said to be \emph{rank-bounded} if there exists $n\geq 1$ such that $A\subset \DMC_{\mathcal{X},[n]}^{(o)}$. $A$ is \emph{rank-unbounded} if it is not rank-bounded.

\subsection{Natural topologies on $\DMC_{\mathcal{X},\ast}^{(o)}$}

\label{subsecNaturalTop}

Since $\DMC_{\mathcal{X},\ast}^{(o)}$ is the quotient of $\DMC_{\mathcal{X},\ast}$ and since $\DMC_{\mathcal{X},\ast}$ was not given any topology, there is no ``standard topology" on $\DMC_{\mathcal{X},\ast}^{(o)}$.

However, there are many properties that one may require from any ``reasonable" topology on $\DMC_{\mathcal{X},\ast}^{(o)}$. For example, one may require the continuity of all mappings that are relevant to information theory such as capacity, mutual information, probability of error of any fixed code, optimal probability of error of a given rate and blocklength, channel sums and products, etc ... The continuity of these mappings under different topologies on $\DMC_{\mathcal{X},\ast}^{(o)}$ is studied in \cite{RajContTop}.

In this paper, we focus on one particular requirement that we consider the most basic property required from any ``acceptable" topology on $\DMC_{\mathcal{X},\ast}^{(o)}$:

\begin{mydef}
A topology $\mathcal{T}$ on $\DMC_{\mathcal{X},\ast}^{(o)}$ is said to be \emph{natural} if it induces the quotient topology $\mathcal{T}_{\mathcal{X},[n]}^{(o)}$ on $\DMC_{\mathcal{X},[n]}^{(o)}$ for every $n\geq 1$.
\end{mydef}

The reason why we consider such topology as natural is because $\DMC_{\mathcal{X},[n]}^{(o)}$ is subset of $\DMC_{\mathcal{X},\ast}^{(o)}$ and the quotient topology $\mathcal{T}_{\mathcal{X},[n]}^{(o)}$ is the ``standard" and ``most natural" topology on $\DMC_{\mathcal{X},[n]}^{(o)}$. Therefore, we do not want to induce any non-standard topology on $\DMC_{\mathcal{X},[n]}^{(o)}$ by relativization.

Before discussing any particular natural topology, we would like to discuss a few properties that are common to all natural topologies.

\begin{myprop}
Every natural topology is $\sigma$-compact, separable and path-connected.
\end{myprop}
\begin{proof}
Since $\DMC_{\mathcal{X},\ast}^{(o)}$ is the countable union of compact and separable subspaces (namely $\{\DMC_{\mathcal{X},[n]}^{(o)}\}_{n\geq 1}$), $\DMC_{\mathcal{X},\ast}^{(o)}$ is $\sigma$-compact and separable.

On the other hand, since $\displaystyle\bigcap_{n\geq 1}\textstyle\DMC_{\mathcal{X},[n]}^{(o)}=\DMC_{\mathcal{X},[1]}^{(o)}\neq\o$ and since $\DMC_{\mathcal{X},[n]}^{(o)}$ is path-connected for every $n\geq1$, the union $\textstyle\DMC_{\mathcal{X},\ast}^{(o)}={\displaystyle\bigcup_{n\geq 1}}\DMC_{\mathcal{X},[n]}^{(o)}$ is path-connected.
\end{proof}

\begin{myprop}
\label{propNaturallyOpenUnbounded}
If $|\mathcal{X}|\geq 2$ and $\mathcal{T}$ is a natural topology, every open set is rank-unbounded.
\end{myprop}
\begin{proof}
Assume to the contrary that there exists a non-empty open set $U\in \mathcal{T}$ such that $U\subset \DMC_{\mathcal{X},[n]}^{(o)}$ for some $n\geq 1$. $U\cap \DMC_{\mathcal{X},[n+1]}^{(o)}$ is open in $\DMC_{\mathcal{X},[n+1]}^{(o)}$ because $\mathcal{T}$ is natural. On the other hand, $U\cap \DMC_{\mathcal{X},[n+1]}^{(o)}\subset U\subset \DMC_{\mathcal{X},[n]}^{(o)}$. Proposition \ref{propInteriorEmptyDMCXno} now implies that $U\cap \DMC_{\mathcal{X},[n+1]}^{(o)}=\o$. Therefore,
$$\textstyle U=U\cap \DMC_{\mathcal{X},[n]}^{(o)}\subset U\cap \DMC_{\mathcal{X},[n+1]}^{(o)}=\o,$$ which is a contradiction.
\end{proof}

\begin{mycor}
\label{corInteriorEmptyDMCXno}
If $|\mathcal{X}|\geq 2$ and $\mathcal{T}$ is a natural topology, then for every $n\geq 1$, the interior of $\DMC_{\mathcal{X},[n]}^{(o)}$ in $(\DMC_{\mathcal{X},\ast}^{(o)},\mathcal{T})$ is empty.
\end{mycor}

\begin{myprop}
\label{propNaturaIsNotBaire}
If $|\mathcal{X}|\geq 2$ and $\mathcal{T}$ is a Hausdorff natural topology, then $(\DMC_{\mathcal{X},\ast}^{(o)},\mathcal{T})$ is not a Baire space.
\end{myprop}
\begin{proof}
Fix $n\geq 1$. Since $\mathcal{T}$ is natural, $\DMC_{\mathcal{X},[n]}^{(o)}$ is a compact subset of $(\DMC_{\mathcal{X},\ast}^{(o)},\mathcal{T})$. But $\mathcal{T}$ is Hausdorff, so $\DMC_{\mathcal{X},[n]}^{(o)}$ is a closed subset of $(\DMC_{\mathcal{X},\ast}^{(o)},\mathcal{T})$. Therefore, $\DMC_{\mathcal{X},\ast}^{(o)}\setminus\DMC_{\mathcal{X},[n]}^{(o)}$ is open.

On the other hand, Corollary \ref{corInteriorEmptyDMCXno} shows that the interior of $\DMC_{\mathcal{X},[n]}^{(o)}$ in $(\DMC_{\mathcal{X},\ast}^{(o)},\mathcal{T})$ is empty. Therefore, $\DMC_{\mathcal{X},\ast}^{(o)}\setminus\DMC_{\mathcal{X},[n]}^{(o)}$ is dense in $(\DMC_{\mathcal{X},\ast}^{(o)},\mathcal{T})$.

Now since
$$\bigcap_{n\geq 1} {\textstyle\left(\DMC_{\mathcal{X},\ast}^{(o)}\setminus\DMC_{\mathcal{X},[n]}^{(o)}\right)}={\DMC}_{\mathcal{X},\ast}^{(o)}\setminus {\textstyle\left({\displaystyle\bigcup_{n\geq 1}}\DMC_{\mathcal{X},[n]}^{(o)}\right)}=\o,$$
and since $\DMC_{\mathcal{X},\ast}^{(o)}\setminus\DMC_{\mathcal{X},[n]}^{(o)}$ is open and dense in $(\DMC_{\mathcal{X},\ast}^{(o)},\mathcal{T})$ for every $n\geq 1$, we conclude that $(\DMC_{\mathcal{X},\ast}^{(o)},\mathcal{T})$ is not a Baire space.
\end{proof}

\begin{mycor}
\label{corNaturalNotComplete}
If $|\mathcal{X}|\geq 2$, no natural topology on $\DMC_{\mathcal{X},\ast}^{(o)}$ can be completely metrizable.
\end{mycor}
\begin{proof}
The corollary follows from Proposition \ref{propNaturaIsNotBaire} and the fact that every completely metrizable topology is both Hausdorff and Baire.
\end{proof}

\begin{myprop}
\label{propNaturalNotLocallyCompact}
If $|\mathcal{X}|\geq 2$ and $\mathcal{T}$ is a Hausdorff natural topology, then $(\DMC_{\mathcal{X},\ast}^{(o)},\mathcal{T})$ is not locally compact anywhere, i.e., for every $\hat{W}\in \DMC_{\mathcal{X},\ast}^{(o)}$, there is no compact neighborhood of $\hat{W}$ in $(\DMC_{\mathcal{X},\ast}^{(o)},\mathcal{T})$.
\end{myprop}
\begin{proof}
Assume to the contrary that there exists a compact neighborhood $K$ of $\hat{W}$. There exists an open set $U$ such that $\hat{W}\in U\subset K$.

Since $K$ is compact and Hausdorff, it is a Baire space. Moreover, since $U$ is an open subset of $K$, $U$ is also a Baire space.

Fix $n\geq 1$. Since the interior of $\DMC_{\mathcal{X},[n]}^{(o)}$ in $(\DMC_{\mathcal{X},\ast}^{(o)},\mathcal{T})$ is empty, the interior of $U\cap\DMC_{\mathcal{X},[n]}^{(o)}$ in $U$ is also empty. Therefore, $U\setminus\DMC_{\mathcal{X},[n]}^{(o)}$ is dense in $U$. On the other hand, since $\mathcal{T}$ is natural, $\DMC_{\mathcal{X},[n]}^{(o)}$ is compact which implies that it is closed because $\mathcal{T}$ is Hausdorff. Therefore, $U\setminus\DMC_{\mathcal{X},[n]}^{(o)}$ is open in $U$. Now since
$$\bigcap_{n\geq 1} {\textstyle\left(U\setminus\DMC_{\mathcal{X},[n]}^{(o)}\right)}=U\setminus {\textstyle\left({\displaystyle\bigcup_{n\geq 1}}\DMC_{\mathcal{X},[n]}^{(o)}\right)}=\o,$$
and since $U\setminus\DMC_{\mathcal{X},[n]}^{(o)}$ is open and dense in $U$ for every $n\geq 1$, $U$ is not Baire, which is a contradiction. Therefore, there is no compact neighborhood of $\hat{W}$ in $(\DMC_{\mathcal{X},\ast}^{(o)},\mathcal{T})$.
\end{proof}

\section{Strong topology on $\DMC_{\mathcal{X},\ast}^{(o)}$}

\label{subsecStrongTop}

The first natural topology that we study is the \emph{strong topology} $\mathcal{T}_{s,\mathcal{X},\ast}^{(o)}$ on $\DMC_{\mathcal{X},\ast}^{(o)}$, which is the finest natural topology.

Since the spaces $\{\DMC_{\mathcal{X},[n]}\}_{n\geq 1}$ are disjoint and since there is no a priori way to (topologically) compare channels in $\DMC_{\mathcal{X},[n]}$ with channels in $\DMC_{\mathcal{X},[n']}$ for $n\neq n'$, the ``most natural" topology that we can define on $\DMC_{\mathcal{X},\ast}$ is the disjoint union topology $\mathcal{T}_{s,\mathcal{X},\ast}:=\displaystyle\bigoplus_{n\geq 1}\mathcal{T}_{\mathcal{X},[n]}$. Clearly, the space $(\DMC_{\mathcal{X},\ast},\mathcal{T}_{s,\mathcal{X},\ast})$ is disconnected. Moreover, $\mathcal{T}_{s,\mathcal{X},\ast}$ is metrizable because it is the disjoint union of metrizable spaces. It is also $\sigma$-compact because it is the union of countably many compact spaces.

We added the subscript $s$ to emphasize the fact that $\mathcal{T}_{s,\mathcal{X},\ast}$ is a strong topology (remember that the disjoint union topology is the \emph{finest} topology that makes the canonical injections continuous).

\begin{mydef}
We define the strong topology $\mathcal{T}_{s,\mathcal{X},\ast}^{(o)}$ on $\DMC_{\mathcal{X},\ast}^{(o)}$ as the quotient topology $\mathcal{T}_{s,\mathcal{X},\ast}/R_{\mathcal{X},\ast}^{(o)}$.

We call open and closed sets in $(\DMC_{\mathcal{X},\ast}^{(o)},\mathcal{T}_{s,\mathcal{X},\ast}^{(o)})$ as strongly open and strongly closed sets respectively.
\end{mydef}

Let $\Proj:\DMC_{\mathcal{X},\ast}\rightarrow\DMC_{\mathcal{X},\ast}^{(o)}$ be the projection onto the $R_{\mathcal{X},\ast}^{(o)}$-equivalence classes, and for every $n\geq 1$ let $\Proj_n:\DMC_{\mathcal{X},[n]}\rightarrow\DMC_{\mathcal{X},[n]}^{(o)}$ be the projection onto the $R_{\mathcal{X},[n]}^{(o)}$-equivalence classes. Due to the identifications that we made in Section \ref{secEquivSpace}, we have $\Proj(W)=\Proj_n(W)$ for every $W\in\DMC_{\mathcal{X},[n]}$. Therefore, for every $U\subset \DMC_{\mathcal{X},\ast}^{(o)}$, we have
$$\textstyle\Proj^{-1}(U)={\displaystyle\coprod_{n\geq 1}}\Proj_n^{-1}(U\cap\DMC_{\mathcal{X},[n]}^{(o)}).$$
Hence,
\begin{align*}
\textstyle U\in\mathcal{T}_{s,\mathcal{X},\ast}^{(o)}\;\;&\stackrel{(a)}{\Leftrightarrow}\;\;\textstyle\Proj^{-1}(U)\in\mathcal{T}_{s,\mathcal{X},\ast}\\
&\textstyle\stackrel{(b)}{\Leftrightarrow}\;\;\Proj^{-1}(U)\cap\DMC_{\mathcal{X},[n]} \in\mathcal{T}_{\mathcal{X},[n]},\;\;\forall n\geq 1\\
&\textstyle\Leftrightarrow\;\;\left({\displaystyle\coprod_{n'\geq 1}}\Proj_{n'}^{-1}(U\cap\DMC_{\mathcal{X},[n']}^{(o)})\right)\cap\DMC_{\mathcal{X},[n]} \in\mathcal{T}_{\mathcal{X},[n]},\;\;\forall n\geq 1\\
&\textstyle\Leftrightarrow\;\;\Proj_n^{-1}(U\cap\DMC_{\mathcal{X},[n]}^{(o)}) \in\mathcal{T}_{\mathcal{X},[n]},\;\;\forall n\geq 1\\
&\textstyle\stackrel{(c)}{\Leftrightarrow}\;\;U\cap\DMC_{\mathcal{X},[n]}^{(o)} \in\mathcal{T}_{\mathcal{X},[n]}^{(o)},\;\;\forall n\geq 1,
\end{align*}
where (a) and (c) follow from the properties of the quotient topology, and (b) follows from the properties of the disjoint union topology.

We conclude that $U\subset \DMC_{\mathcal{X},\ast}^{(o)}$ is strongly open in $\DMC_{\mathcal{X},\ast}^{(o)}$ if and only if $U\cap \DMC_{\mathcal{X},[n]}^{(o)}$ is open in $\DMC_{\mathcal{X},[n]}^{(o)}$ for every $n\geq 1$. This shows that the topology on $\DMC_{\mathcal{X},[n]}^{(o)}$ that is inherited from $(\DMC_{\mathcal{X},\ast}^{(o)},\mathcal{T}_{s,\mathcal{X},\ast}^{(o)})$ is exactly $\mathcal{T}_{\mathcal{X},[n]}^{(o)}$. Therefore, $\mathcal{T}_{s,\mathcal{X},\ast}^{(o)}$ is a natural topology. On the other hand, if $\mathcal{T}$ is an arbitrary natural topology and $U\in\mathcal{T}$, then $U\cap\DMC_{\mathcal{X},[n]}^{(o)}$ is open in $\DMC_{\mathcal{X},[n]}^{(o)}$ for every $n\geq 1$, so $U\in\mathcal{T}_{s,\mathcal{X},\ast}^{(o)}$. We conclude that $\mathcal{T}_{s,\mathcal{X},\ast}^{(o)}$ is the finest natural topology.

\vspace*{3mm}

We can also characterize the strongly closed subsets of $\DMC_{\mathcal{X},\ast}^{(o)}$ in terms of the closed sets of the $\DMC_{\mathcal{X},[n]}^{(o)}$ spaces:
\begin{align*}
\textstyle F\;\text{is strongly closed in}\;\DMC_{\mathcal{X},\ast}^{(o)}\;\;&\Leftrightarrow\;\;\textstyle\DMC_{\mathcal{X},\ast}^{(o)}\setminus F\;\text{is strongly open in}\;\textstyle\DMC_{\mathcal{X},\ast}^{(o)}\\
&\textstyle\Leftrightarrow\;\;\left(\DMC_{\mathcal{X},\ast}^{(o)}\setminus F\right)\cap\DMC_{\mathcal{X},[n]}^{(o)}\;\text{is open in}\;\DMC_{\mathcal{X},[n]}^{(o)},\;\;\forall n\geq 1\\
&\textstyle\Leftrightarrow\;\;\DMC_{\mathcal{X},[n]}^{(o)}\setminus \left(F\cap\DMC_{\mathcal{X},[n]}^{(o)}\right)\;\text{is open in}\;\DMC_{\mathcal{X},[n]}^{(o)},\;\;\forall n\geq 1\\
&\textstyle\Leftrightarrow\;\;F\cap\DMC_{\mathcal{X},[n]}^{(o)}\;\text{is closed in}\;\DMC_{\mathcal{X},[n]}^{(o)},\;\;\forall n\geq 1.
\end{align*}

Since $\DMC_{\mathcal{X},[n]}^{(o)}$ is metrizable for every $n\geq 1$, it is also normal. We can use this fact to prove that the strong topology on $\DMC_{\mathcal{X},\ast}^{(o)}$ is normal:

\begin{mylem}
\label{lemDMCXoNorm}
$(\DMC_{\mathcal{X},\ast}^{(o)},\mathcal{T}_{s,\mathcal{X},\ast}^{(o)})$ is normal.
\end{mylem}
\begin{proof}
See Appendix \ref{appDMCXoNorm}.
\end{proof}

\vspace*{3mm}

The following theorem shows that the strong topology satisfies many desirable properties.

\begin{mythe}
\label{theDMCXo}
$(\DMC_{\mathcal{X},\ast}^{(o)},\mathcal{T}_{s,\mathcal{X},\ast}^{(o)})$ is a compactly generated, sequential and $T_4$ space.
\end{mythe}
\begin{proof}
Since $(\DMC_{\mathcal{X},\ast},\mathcal{T}_{s,\mathcal{X},\ast})$ is metrizable, it is sequential. Therefore, $(\DMC_{\mathcal{X},\ast}^{(o)},\mathcal{T}_{s,\mathcal{X},\ast}^{(o)})$, which is the quotient of a sequential space, is sequential.

Let us now show that $\DMC_{\mathcal{X},\ast}^{(o)}$ is $T_4$. Fix $\hat{W}\in\DMC_{\mathcal{X},\ast}^{(o)}$. For every $n\geq 1$, $\{\hat{W}\}\cap \DMC_{\mathcal{X},[n]}^{(o)}$ is either $\{\hat{W}\}$ or $\o$ depending on whether $\hat{W}\in \DMC_{\mathcal{X},[n]}^{(o)}$ or not. Since $\DMC_{\mathcal{X},[n]}^{(o)}$ is metrizable, it is $T_1$ and so singletons are closed in $\DMC_{\mathcal{X},[n]}^{(o)}$. We conclude that in all cases, $\{\hat{W}\}\cap \DMC_{\mathcal{X},[n]}^{(o)}$ is closed in $\DMC_{\mathcal{X},[n]}^{(o)}$ for every $n\geq 1$. Therefore, $\{\hat{W}\}$ is strongly closed in $\DMC_{\mathcal{X},\ast}^{(o)}$. This shows that $(\DMC_{\mathcal{X},\ast}^{(o)},\mathcal{T}_{s,\mathcal{X},\ast}^{(o)})$ is $T_1$. On the other hand, Lemma \ref{lemDMCXoNorm} shows that $(\DMC_{\mathcal{X},\ast}^{(o)},\mathcal{T}_{s,\mathcal{X},\ast}^{(o)})$ is normal. This means that $(\DMC_{\mathcal{X},\ast}^{(o)},\mathcal{T}_{s,\mathcal{X},\ast}^{(o)})$ is $T_4$, which implies that it is Hausdorff.

Now since $(\DMC_{\mathcal{X},\ast},\mathcal{T}_{s,\mathcal{X},\ast})$ is metrizable, it is compactly generated. On the other hand, the quotient space $(\DMC_{\mathcal{X},\ast}^{(o)},\mathcal{T}_{s,\mathcal{X},\ast}^{(o)})$ was shown to be Hausdorff. We conclude that $(\DMC_{\mathcal{X},\ast}^{(o)},\mathcal{T}_{s,\mathcal{X},\ast}^{(o)})$ is compactly generated.
\end{proof}

\begin{mycor}
If $|\mathcal{X}|\geq 2$, $(\DMC_{\mathcal{X},\ast}^{(o)},\mathcal{T}_{s,\mathcal{X},\ast}^{(o)})$ is not locally compact anywhere.
\end{mycor}
\begin{proof}
Since $\mathcal{T}_{s,\mathcal{X},\ast}^{(o)}$ is a natural Hausdorff topology, Proposition \ref{propNaturalNotLocallyCompact} implies that $\mathcal{T}_{s,\mathcal{X},\ast}^{(o)}$ is not locally compact anywhere.
\end{proof}

\vspace*{3mm}

Although $(\DMC_{\mathcal{X},\ast},\mathcal{T}_{s,\mathcal{X},\ast})$ is second-countable (because it is a $\sigma$-compact metrizable space), the quotient space $(\DMC_{\mathcal{X},\ast}^{(o)},\mathcal{T}_{s,\mathcal{X},\ast}^{(o)})$ is not second-countable. In fact, we will show later that $(\DMC_{\mathcal{X},\ast}^{(o)},\mathcal{T}_{s,\mathcal{X},\ast}^{(o)})$ fails to be first-countable (and hence it is not metrizable). This is one manifestation of the strength of the topology $\mathcal{T}_{s,\mathcal{X},\ast}^{(o)}$. In order to show that $(\DMC_{\mathcal{X},\ast}^{(o)},\mathcal{T}_{s,\mathcal{X},\ast}^{(o)})$ is not first-countable, we need to characterize the converging sequences in $(\DMC_{\mathcal{X},\ast}^{(o)},\mathcal{T}_{s,\mathcal{X},\ast}^{(o)})$.

A sequence $(\hat{W}_n)_{n\geq 1}$ in $\DMC_{\mathcal{X},\ast}^{(o)}$ is said to be \emph{rank-bounded} if $\rank(\hat{W}_n)$ is bounded. $(\hat{W}_n)_{n\geq 1}$ is \emph{rank-unbounded} if it is not bounded.

The following proposition shows that every rank-unbounded sequence does not converge in $(\DMC_{\mathcal{X},\ast}^{(o)},\mathcal{T}_{s,\mathcal{X},\ast}^{(o)})$.
\begin{myprop}
\label{propCharacConvSeq}
A sequence $(\hat{W}_n)_{n\geq 0}$ converges in $(\DMC_{\mathcal{X},\ast}^{(o)},\mathcal{T}_{s,\mathcal{X},\ast}^{(o)})$ if and only if there exists $m\geq 1$ such that $\hat{W}_n\in \DMC_{\mathcal{X},[m]}^{(o)}$ for every $n\geq 0$, and $(\hat{W}_n)_{n\geq 0}$ converges in $(\DMC_{\mathcal{X},[m]}^{(o)},\mathcal{T}_{\mathcal{X},[m]}^{(o)})$.
\end{myprop}
\begin{proof}
Assume that a sequence $(\hat{W}_n)_{n\geq 0}$ in $\DMC_{\mathcal{X},\ast}^{(o)}$ is rank-unbounded. This cannot happen unless $|\mathcal{X}|\geq 2$. In order to show that $(\hat{W}_n)_{n\geq 0}$ does not converge, it is sufficient to show that there exists a subsequence of $(\hat{W}_n)_{n\geq 0}$ which does not converge.

Let $(\hat{W}_{n_k})_{k\geq 0}$ be any subsequence of $(\hat{W}_n)_{n\geq 0}$ where the rank strictly increases, i.e., $\rank(W_{n_k})<\rank(W_{n_{k'}})$ for every $0\leq k<k'$. We will show that $(\hat{W}_{n_k})_{k\geq 0}$ does not converge.

Assume to the contrary that $(\hat{W}_{n_k})_{k\geq 0}$ converges to $\hat{W}\in\DMC_{\mathcal{X},\ast}^{(o)}$. Define the set $$A=\{\hat{W}_{n_k}:\;k\geq 0\}\setminus\{\hat{W}\}.$$ For every $m\geq 1$, the set $A\cap \DMC_{\mathcal{X},[m]}^{(o)}$ contains finitely many points. This means that $A\cap \DMC_{\mathcal{X},[m]}^{(o)}$ is a finite union of singletons (which are closed in $\DMC_{\mathcal{X},[m]}^{(o)}$), hence $A\cap \DMC_{\mathcal{X},[m]}^{(o)}$ is closed in $\DMC_{\mathcal{X},[m]}^{(o)}$ for every $m\geq 1$. Therefore $A$ is closed in $(\DMC_{\mathcal{X},\ast}^{(o)},\mathcal{T}_{s,\mathcal{X},\ast}^{(o)})$.

Now define $U=\DMC_{\mathcal{X},\ast}^{(o)}\setminus A$. Since $A$ is strongly closed, $U$ is strongly open. Moreover, $U$ contains $\hat{W}$, so $U$ is a neighborhood of $\hat{W}$. Therefore, there exists $k_0\geq 0$ such that  $\hat{W}_{n_k}\in U$ for every $k\geq k_0$. Now since the rank of $(\hat{W}_{n_k})_{k\geq 0}$  strictly increases, we can find $k\geq k_0$ such that $\rank(\hat{W}_{n_k})>\rank(\hat{W})$. This means that $\hat{W}_{n_k}\neq\hat{W}$ and so $\hat{W}_{n_k}\in A$. Therefore, $\hat{W}_{n_k}\notin U$ which is a contradiction.

We conclude that every converging sequence in $(\DMC_{\mathcal{X},\ast}^{(o)},\mathcal{T}_{s,\mathcal{X},\ast}^{(o)})$ must be rank-bounded.

Now let $(\hat{W}_n)_{n\geq 0}$ be a rank-bounded sequence in $\DMC_{\mathcal{X},\ast}^{(o)}$, i.e., there exists $m\geq 1$ such that  $\hat{W}_n\in \DMC_{\mathcal{X},[m]}^{(o)}$ for every $n\geq 0$. If $(\hat{W}_n)_{n\geq 0}$ converges in $(\DMC_{\mathcal{X},\ast}^{(o)},\mathcal{T}_{s,\mathcal{X},\ast}^{(o)})$ then it converges in $\DMC_{\mathcal{X},[m]}^{(o)}$ since $\DMC_{\mathcal{X},[m]}^{(o)}$ is strongly closed.

Conversely, assume that $(\hat{W}_n)_{n\geq 0}$ converges in $(\DMC_{\mathcal{X},[m]}^{(o)},\mathcal{T}_{\mathcal{X},[m]}^{(o)})$ to $\hat{W}\in \DMC_{\mathcal{X},[m]}^{(o)}$. Let $O$ be any neighborhood of $\hat{W}$ in $(\DMC_{\mathcal{X},\ast}^{(o)},\mathcal{T}_{s,\mathcal{X},\ast}^{(o)})$. There exists a strongly open set $U$ such that $\hat{W}\in U\subset O$. Since $U\cap \DMC_{\mathcal{X},[m]}^{(o)}$ is open in $(\DMC_{\mathcal{X},[m]}^{(o)},\mathcal{T}_{\mathcal{X},[m]}^{(o)})$, there exists $n_0>0$ such that $\hat{W}_n\in  U\cap \DMC_{\mathcal{X},[m]}^{(o)}$ for every $n\geq n_0$. This implies that $\hat{W}_n\in  O$ for every $n\geq n_0$. Therefore $(\hat{W}_n)_{n\geq 0}$ converges to $\hat{W}$ in $(\DMC_{\mathcal{X},\ast}^{(o)},\mathcal{T}_{s,\mathcal{X},\ast}^{(o)})$.
\end{proof}

\begin{mycor}
\label{corNotFirstCountable}
If $|\mathcal{X}|\geq 2$, $(\DMC_{\mathcal{X},\ast}^{(o)},\mathcal{T}_{s,\mathcal{X},\ast}^{(o)})$ is not first-countable anywhere, i.e., for every $\hat{W}\in\DMC_{\mathcal{X},\ast}^{(o)}$, there is no countable neighborhood basis of $\hat{W}$.
\end{mycor}
\begin{proof}
Fix $\hat{W}\in \DMC_{\mathcal{X},\ast}^{(o)}$ and assume to the contrary that $\hat{W}$ admits a countable neighborhood basis $\{O_n\}_{n\geq 1}$ in $(\DMC_{\mathcal{X},\ast}^{(o)},\mathcal{T}_{s,\mathcal{X},\ast}^{(o)})$. For every $n\geq 1$, let $U_n'$ be a strongly open set such that $\hat{W}\in U_n'\subset O_n$. Define $\displaystyle U_n=\bigcap_{i=1}^n U_n'$. $U_n$ is strongly open because it is the intersection of finitely many strongly open sets. Moreover, $U_n\subset O_m$ for every $n\geq m$.

For every $n\geq 1$, Proposition \ref{propNaturallyOpenUnbounded} implies that $U_n$ (which is non-empty and strongly open) is rank-unbounded, so it cannot be contained in $\DMC_{\mathcal{X},[n]}^{(o)}$. Hence there exists $\hat{W}_n\in U_n$ such that $\hat{W}_n\notin \DMC_{\mathcal{X},[n]}^{(o)}$.

Since $\hat{W}_n\notin \DMC_{\mathcal{X},[n]}^{(o)}$, we have $\rank(\hat{W}_n)>n$ for every $n\geq 1$. Therefore, $(\hat{W}_n)_{n\geq 1}$ is rank-unbounded. Proposition \ref{propCharacConvSeq} implies that $(\hat{W}_n)_{n\geq 1}$ does not converge in $(\DMC_{\mathcal{X},\ast}^{(o)},\mathcal{T}_{s,\mathcal{X},\ast}^{(o)})$.

Now let $O$ be a neighborhood of $\hat{W}$ in $(\DMC_{\mathcal{X},\ast}^{(o)},\mathcal{T}_{s,\mathcal{X},\ast}^{(o)})$. Since $\{O_n\}_{n\geq 1}$ is a neighborhood basis for $\hat{W}$, there exists $n_0\geq 1$ such that $O_{n_0}\subset O$. For every $n\geq n_0$, we have $\hat{W}_n\in U_n\subset O_{n_0}\subset O$. This means that $(\hat{W}_n)_{n\geq 1}$ converges to $\hat{W}$ in $(\DMC_{\mathcal{X},\ast}^{(o)},\mathcal{T}_{s,\mathcal{X},\ast}^{(o)})$ which is a contradiction. Therefore, $\hat{W}$ does not admit a countable neighborhood basis in $(\DMC_{\mathcal{X},\ast}^{(o)},\mathcal{T}_{s,\mathcal{X},\ast}^{(o)})$.
\end{proof}

\subsection{Compact subspaces of $(\DMC_{\mathcal{X},\ast}^{(o)},\mathcal{T}_{s,\mathcal{X},\ast}^{(o)})$}

It is well known that a compact subset of $\mathbb{R}$ is compact if and only if it is closed and bounded. The following proposition shows that a similar statement holds for $(\DMC_{\mathcal{X},\ast}^{(o)},\mathcal{T}_{s,\mathcal{X},\ast}^{(o)})$.

\begin{myprop}
\label{propCharacCompactDMCXos}
A subspace of $(\DMC_{\mathcal{X},\ast}^{(o)},\mathcal{T}_{s,\mathcal{X},\ast}^{(o)})$ is compact if and only if it is rank-bounded and strongly closed.
\end{myprop}
\begin{proof}
If $|\mathcal{X}|=1$, all channels are equivalent to each other and so $\DMC_{\mathcal{X},\ast}^{(o)}=\DMC_{\mathcal{X},[1]}^{(o)}$ consists of a single point. Therefore, all subsets of $\DMC_{\mathcal{X},\ast}^{(o)}$ are rank-bounded, compact and strongly closed.

Assume now that $|\mathcal{X}|\geq 2$.
Let $A$ be a subspace of $(\DMC_{\mathcal{X},\ast}^{(o)},\mathcal{T}_{s,\mathcal{X},\ast}^{(o)})$. If $A$ is rank-bounded and strongly closed, then there exists $n\geq 1$ such that $A\subset \DMC_{\mathcal{X},[n]}^{(o)}$. Since $A$ is strongly closed, then $A=A\cap \DMC_{\mathcal{X},[n]}^{(o)}$ is closed in $\DMC_{\mathcal{X},[n]}^{(o)}$ which is compact. Therefore, $A$ is compact.

Now let $A$ be a compact subspace of $(\DMC_{\mathcal{X},\ast}^{(o)},\mathcal{T}_{s,\mathcal{X},\ast}^{(o)})$. Since $(\DMC_{\mathcal{X},\ast}^{(o)},\mathcal{T}_{s,\mathcal{X},\ast}^{(o)})$ is Hausdorff, $A$ is strongly closed. It remains to show that $A$ is rank-bounded.

Assume to the contrary that $A$ is rank-unbounded. We can construct a sequence $(\hat{W}_n)_{n\geq 0}$ in $A$ where the rank is strictly increasing, i.e., $\rank(\hat{W}_n)<\rank(\hat{W}_{n'})$ for every $0\leq n<n'$. Since the rank of $(\hat{W}_n)_{n\geq 0}$ is strictly increasing, every subsequence of $(\hat{W}_n)_{n\geq 0}$ is rank-unbounded. Proposition \ref{propCharacConvSeq} implies that every subsequence of $(\hat{W}_n)_{n\geq 0}$ does not converge in $(\DMC_{\mathcal{X},\ast}^{(o)},\mathcal{T}_{s,\mathcal{X},\ast}^{(o)})$. On the other hand, we have:
\begin{itemize}
\item $A$ is countably compact because it is compact.
\item Since $A$ is strongly closed and since $(\DMC_{\mathcal{X},\ast}^{(o)},\mathcal{T}_{s,\mathcal{X},\ast}^{(o)})$ is a sequential space, $A$ is sequential.
\item $A$ is Hausdorff because $(\DMC_{\mathcal{X},\ast}^{(o)},\mathcal{T}_{s,\mathcal{X},\ast}^{(o)})$ is Hausdorff.
\end{itemize}
Now since every countably compact sequential Hausdorff space is sequentially compact \cite{SequentialSpace}, $A$ must be sequentially compact. Therefore, $(\hat{W}_n)_{n\geq 0}$ has a converging subsequence which is a contradiction. We conclude that $A$ must be rank-bounded.
\end{proof}

\section{The noisiness metric on $\DMC$ spaces}

Theorem \ref{theDMCXYo} implies that $\DMC_{\mathcal{X},[n]}^{(o)}$ is metrizable for every $n\geq 1$. One might ask whether the spaces $\DMC_{\mathcal{X},[n]}^{(o)}$ are ``simultaneously metrizable" in the sense that we can define a metric $d_n$ on $\DMC_{\mathcal{X},[n]}^{(o)}$ for every $n\geq 1$ in such a way that $d_n$ is the restriction of $d_{n+1}$ for every $n\geq 1$. If this is the case, we can then define a metric on $\DMC_{\mathcal{X},\ast}^{(o)}={\displaystyle\bigcup_{n\geq 1}}\DMC_{\mathcal{X},[n]}^{(o)}$ as $d(\hat{W},\hat{W}')=d_n(\hat{W},\hat{W}')$ for any $n\geq 1$ satisfying $\hat{W},\hat{W}'\in \DMC_{\mathcal{X},[n]}^{(o)}$. In this section we will show that such metrics can be constructed.

\subsection{Noisiness metric on $\DMC_{\mathcal{X},\mathcal{Y}}^{(o)}$}

For every $m\geq 1$, let $\Delta_{[m]\times\mathcal{X}}$ be the space of probability distributions on $[m]\times\mathcal{X}$.

Let $\mathcal{Y}$ be a finite set and let $W\in\DMC_{\mathcal{X},\mathcal{Y}}$. For every $p\in \Delta_{[m]\times\mathcal{X}}$, define $P_c(p,W)$ as follows:
\begin{equation}
\label{eqProbCorrectGuess}
P_c(p,W)=\sup_{D\in \DMC_{\mathcal{Y},[m]}}\sum_{\substack{u\in[m],\\x\in\mathcal{X},\\y\in\mathcal{Y}}}p(u,x)W(y|x)D(u|y).
\end{equation}
$P_c(p,W)$ can be interpreted as follows: let $(U,X)$ be a pair of random variables distributed according to $p$, send $X$ through the channel $W$, and let $Y$ be the output of $W$ in such a way that $U-X-Y$ is a Markov chain. Let $\hat{U}$ be the estimate of $U$ obtained by applying a random decoder $D\in\DMC_{\mathcal{Y},[m]}$. In this interpretation, $p$ can be seen as a random encoder. The probability of correctly guessing $U$ by using the decoder $D$ is given by $$\displaystyle \sum_{\substack{u\in[m],\\x\in\mathcal{X},\\y\in\mathcal{Y}}}p(u,x)W(y|x)D(u|y).$$ Therefore, $P_c(p,W)$ is the optimal probability of correctly guessing $U$ from $Y$. Note that we can take the supremum in \eqref{eqProbCorrectGuess} over only deterministic channels $D\in\DMC_{\mathcal{Y},[m]}$ because we can always choose an optimal decoder that is deterministic.

It is well known that if $W$ is degraded from $W'$, then $P_c(p,W)\leq P_c(p,W')$ for every $p\in\Delta_{[m]\times\mathcal{X}}$ and every $m\geq 1$. It was shown in \cite{Buscemi} that the converse is also true. Therefore, $W$ is equivalent to $W'$ if and only if $P_c(p,W)=P_c(p,W')$ for every $p\in\Delta_{[m]\times\mathcal{X}}$ and every $m\geq 1$. This shows that the quantity $P_c(p,W)$ depends only on the $R_{\mathcal{X},\mathcal{Y}}^{(o)}$-equivalence class of $W$. Therefore, if $\hat{W}\in\DMC_{\mathcal{X},\mathcal{Y}}^{(o)}$, we can define $P_c(p,\hat{W}):=P_c(p,W')$ for any $W'\in \hat{W}$.

Define the \emph{noisiness distance} $d_{\mathcal{X},\mathcal{Y}}^{(o)}: \DMC_{\mathcal{X},\mathcal{Y}}^{(o)}\times \DMC_{\mathcal{X},\mathcal{Y}}^{(o)}\rightarrow\mathbb{R}^+$ as follows:
$$d_{\mathcal{X},\mathcal{Y}}^{(o)}(\hat{W}_1,\hat{W}_2)=\sup_{\substack{m\geq 1,\\p\in\Delta_{[m]\times\mathcal{X}}}}|P_c(p,\hat{W}_1)-P_c(p,\hat{W}_2)|.$$
It is easy to see that $0\leq d_{\mathcal{X},\mathcal{Y}}^{(o)}(\hat{W}_1,\hat{W}_2)\leq 1 $ for every $\hat{W}_1,\hat{W}_2\in \DMC_{\mathcal{X},\mathcal{Y}}^{(o)}$. Moreover, we have:
\begin{itemize}
\item $d_{\mathcal{X},\mathcal{Y}}^{(o)}(\hat{W},\hat{W})=0$ for every $\hat{W}\in\DMC_{\mathcal{X},\mathcal{Y}}^{(o)}$.
\item For every $\hat{W}_1,\hat{W}_2\in \DMC_{\mathcal{X},\mathcal{Y}}^{(o)}$, if $d_{\mathcal{X},\mathcal{Y}}^{(o)}(\hat{W}_1,\hat{W}_2)=0$, then $P_c(p,\hat{W}_1)=P_c(p,\hat{W}_2)$ for every $p\in\Delta_{[m]\times\mathcal{X}}$ and every $m\geq 1$, which implies that the channels in $\hat{W}_1$ are equivalent to the channels in $\hat{W}_2$, hence $\hat{W}_1=\hat{W}_2$.
\item $d_{\mathcal{X},\mathcal{Y}}^{(o)}(\hat{W}_1,\hat{W}_2)=d_{\mathcal{X},\mathcal{Y}}^{(o)}(\hat{W}_2,\hat{W}_1)$ for every $\hat{W}_1,\hat{W}_2\in \DMC_{\mathcal{X},\mathcal{Y}}^{(o)}$.
\item $d_{\mathcal{X},\mathcal{Y}}^{(o)}(\hat{W}_1,\hat{W}_3)\leq d_{\mathcal{X},\mathcal{Y}}^{(o)}(\hat{W}_1,\hat{W}_2)+d_{\mathcal{X},\mathcal{Y}}^{(o)}(\hat{W}_2,\hat{W}_3)$ for every $\hat{W}_1,\hat{W}_2,\hat{W}_3\in \DMC_{\mathcal{X},\mathcal{Y}}^{(o)}$.
\end{itemize}
This shows that $d_{\mathcal{X},\mathcal{Y}}^{(o)}$ is a metric on $\DMC_{\mathcal{X},\mathcal{Y}}^{(o)}$. $d_{\mathcal{X},\mathcal{Y}}^{(o)}$ is called the noisiness metric because it compares the ``noisiness" of $\hat{W}_1$ with that of $\hat{W}_2$: if $P_c(p,\hat{W}_1)$ is close to $P_c(p,\hat{W}_2)$ for every random encoder $p$, then $\hat{W}_1$ and $\hat{W}_2$ have close ``noisiness levels".

A natural question to ask is whether the metric topology on $\DMC_{\mathcal{X},\mathcal{Y}}^{(o)}$ that is induced by $d_{\mathcal{X},\mathcal{Y}}^{(o)}$ is the same as the quotient topology $\mathcal{T}_{\mathcal{X},\mathcal{Y}}^{(o)}$ that we defined in Section \ref{subsecDMCXYo}. To answer this question, we need the following lemma.

\begin{mylem}
\label{lemRelDistance}
For every $W_1,W_2\in\DMC_{\mathcal{X},\mathcal{Y}}$, we have:
$$d_{\mathcal{X},\mathcal{Y}}^{(o)}(\hat{W}_1,\hat{W}_2)\leq d_{\mathcal{X},\mathcal{Y}}(W_1,W_2),$$
where $\hat{W}_1$ and $\hat{W}_2$ are the $R_{\mathcal{X},\mathcal{Y}}^{(o)}$-equivalence classes of $W_1$ and $W_2$ respectively.
\end{mylem}
\begin{proof}
See Appendix \ref{appRelDistance}.
\end{proof}

\begin{myprop}
\label{propEQuivTXYodXYo}
$(\DMC_{\mathcal{X},\mathcal{Y}}^{(o)},d_{\mathcal{X},\mathcal{Y}}^{(o)})$ and $(\DMC_{\mathcal{X},\mathcal{Y}}^{(o)},\mathcal{T}_{\mathcal{X},\mathcal{Y}}^{(o)})$ are topologically equivalent.
\end{myprop}
\begin{proof}
Consider the projection mapping $\Proj: \DMC_{\mathcal{X},\mathcal{Y}}\rightarrow \DMC_{\mathcal{X},\mathcal{Y}}^{(o)}$ defined as $\Proj(W)=\hat{W}$, where $\hat{W}$ is the $R_{\mathcal{X},\mathcal{Y}}^{(o)}$-equivalence class of $W$.

Lemma \ref{lemRelDistance} implies that $\Proj$ is a continuous mapping from $(\DMC_{\mathcal{X},\mathcal{Y}},d_{\mathcal{X},\mathcal{Y}})$ to $(\DMC_{\mathcal{X},\mathcal{Y}}^{(o)},d_{\mathcal{X},\mathcal{Y}}^{(o)})$. Now since $\Proj(W)=\Proj(W')$ whenever $W R_{\mathcal{X},\mathcal{Y}}^{(o)} W'$, Lemma \ref{lemQuotientFunction} implies that the identity mapping $id: \DMC_{\mathcal{X},\mathcal{Y}}^{(o)}\rightarrow \DMC_{\mathcal{X},\mathcal{Y}}^{(o)}$ is continuous from $(\DMC_{\mathcal{X},\mathcal{Y}}^{(o)},\mathcal{T}_{\mathcal{X},\mathcal{Y}}^{(o)})$ to $(\DMC_{\mathcal{X},\mathcal{Y}}^{(o)},d_{\mathcal{X},\mathcal{Y}}^{(o)})$. We have:
\begin{itemize}
\item For every $U\subset \DMC_{\mathcal{X},\mathcal{Y}}^{(o)}$ that is open in $(\DMC_{\mathcal{X},\mathcal{Y}}^{(o)},d_{\mathcal{X},\mathcal{Y}}^{(o)})$, $U=id^{-1}(U)\in \mathcal{T}_{\mathcal{X},\mathcal{Y}}^{(o)}$ because $id$ is a continuous mapping from $(\DMC_{\mathcal{X},\mathcal{Y}}^{(o)},\mathcal{T}_{\mathcal{X},\mathcal{Y}}^{(o)})$ to $(\DMC_{\mathcal{X},\mathcal{Y}}^{(o)},d_{\mathcal{X},\mathcal{Y}}^{(o)})$.
\item For every $U\in \mathcal{T}_{\mathcal{X},\mathcal{Y}}^{(o)}$, the set $\DMC_{\mathcal{X},\mathcal{Y}}^{(o)}\setminus U$ is closed in $(\DMC_{\mathcal{X},\mathcal{Y}}^{(o)},\mathcal{T}_{\mathcal{X},\mathcal{Y}}^{(o)})$ which is compact. Therefore, $\DMC_{\mathcal{X},\mathcal{Y}}^{(o)}\setminus U$ is a compact subset of $(\DMC_{\mathcal{X},\mathcal{Y}}^{(o)}, \mathcal{T}_{\mathcal{X},\mathcal{Y}}^{(o)})$. Now since $id$ is continuous from $(\DMC_{\mathcal{X},\mathcal{Y}}^{(o)},\mathcal{T}_{\mathcal{X},\mathcal{Y}}^{(o)})$ to $(\DMC_{\mathcal{X},\mathcal{Y}}^{(o)},d_{\mathcal{X},\mathcal{Y}}^{(o)})$, $\DMC_{\mathcal{X},\mathcal{Y}}^{(o)}\setminus U=id(\DMC_{\mathcal{X},\mathcal{Y}}^{(o)}\setminus U)$ is a compact subset of $(\DMC_{\mathcal{X},\mathcal{Y}}^{(o)},d_{\mathcal{X},\mathcal{Y}}^{(o)})$ which is Hausdorff (because it is metric). This shows that $\DMC_{\mathcal{X},\mathcal{Y}}^{(o)}\setminus U$ is closed in $(\DMC_{\mathcal{X},\mathcal{Y}}^{(o)},d_{\mathcal{X},\mathcal{Y}}^{(o)})$, which implies that $U$ is open in $(\DMC_{\mathcal{X},\mathcal{Y}}^{(o)},d_{\mathcal{X},\mathcal{Y}}^{(o)})$.
\end{itemize}
We conclude that $U\subset \DMC_{\mathcal{X},\mathcal{Y}}^{(o)}$ is open in $(\DMC_{\mathcal{X},\mathcal{Y}}^{(o)},d_{\mathcal{X},\mathcal{Y}}^{(o)})$ if and only if it is open in $(\DMC_{\mathcal{X},\mathcal{Y}}^{(o)},\mathcal{T}_{\mathcal{X},\mathcal{Y}}^{(o)})$.
\end{proof}

\begin{mycor}
$(\DMC_{\mathcal{X},\mathcal{Y}}^{(o)},d_{\mathcal{X},\mathcal{Y}}^{(o)})$ is a compact path-connected metric space.
\end{mycor}

\vspace*{3mm}

The reader might be wondering why we considered and studied the quotient topology $\mathcal{T}_{\mathcal{X},\mathcal{Y}}^{(o)}$ while it is possible to explicitly define a metric on the space $\DMC_{\mathcal{X},\mathcal{Y}}^{(o)}$. There are two reasons:
\begin{itemize}
\item The definition of $d_{\mathcal{X},\mathcal{Y}}^{(o)}$ does not seem to be intuitive at the first sight and it is not clear why one would adopt it as a standard metric on $\DMC_{\mathcal{X},\mathcal{Y}}^{(o)}$. Just being a metric is not convincing enough. On the other hand, the existence of a natural standard topology on $\DMC_{\mathcal{X},\mathcal{Y}}$ makes the quotient topology the most natural starting point.
\item If one wants to show that a mapping $f:\DMC_{\mathcal{X},\mathcal{Y}}^{(o)}\rightarrow S$ is continuous from $(\DMC_{\mathcal{X},\mathcal{Y}}^{(o)},d_{\mathcal{X},\mathcal{Y}}^{(o)})$ to a topological space $(S,\mathcal{V})$, it is much easier to prove it through the quotient topology $\mathcal{T}_{\mathcal{X},\mathcal{Y}}^{(o)}$ rather than proving it directly using the metric $d_{\mathcal{X},\mathcal{Y}}^{(o)}$. Therefore, it is important to show the topological equivalence between $(\DMC_{\mathcal{X},\mathcal{Y}}^{(o)},d_{\mathcal{X},\mathcal{Y}}^{(o)})$ and $(\DMC_{\mathcal{X},\mathcal{Y}}^{(o)},\mathcal{T}_{\mathcal{X},\mathcal{Y}}^{(o)})$.
\end{itemize}

It is worth mentioning that in the proof of Proposition \ref{propEQuivTXYodXYo}, the only topological property of $(\DMC_{\mathcal{X},\mathcal{Y}}^{(o)},\mathcal{T}_{\mathcal{X},\mathcal{Y}}^{(o)})$ that we used is its compactness. This means that we do not need Lemma \ref{lemProjContClosed} to prove Theorem \ref{theDMCXYo}. An alternative proof of Theorem \ref{theDMCXYo} would be to show the compactness and path-connectedness by inheriting those properties from $\DMC_{\mathcal{X},\mathcal{Y}}$, and then show that $(\DMC_{\mathcal{X},\mathcal{Y}}^{(o)},\mathcal{T}_{\mathcal{X},\mathcal{Y}}^{(o)})$ is topologically equivalent to $(\DMC_{\mathcal{X},\mathcal{Y}}^{(o)},d_{\mathcal{X},\mathcal{Y}}^{(o)})$ as in Proposition \ref{propEQuivTXYodXYo}.

The main reason why we restricted ourselves to topological methods in Section \ref{subsecDMCXYo} is because they might be useful if one wants to generalize our results to spaces of non-discrete channels. It might not be easy to find an explicit metric for those spaces, or even worse, those spaces might fail to be metrizable. Therefore, one might want to prove weaker topological properties such as being Hausdorff and/or regular. In such cases, the methods of Section \ref{subsecDMCXYo} might be useful.

\subsection{Noisiness metric on $\DMC_{\mathcal{X},\ast}^{(o)}$}

\label{subsecMetricDMCXo}

For every $\hat{W}_1,\hat{W}_2\in\DMC_{\mathcal{X},\ast}^{(o)}$, define the \emph{noisiness metric} on $\DMC_{\mathcal{X},\ast}^{(o)}$ as follows: $$d_{\mathcal{X},\ast}^{(o)}(\hat{W},\hat{W}'):=d_{\mathcal{X},[n]}^{(o)}(\hat{W},\hat{W}')\;\text{where}\;n\geq 1\;\text{satisfies}\;\hat{W},\hat{W}'\in\textstyle\DMC_{\mathcal{X},[n]}^{(o)}.$$
$d_{\mathcal{X},\ast}^{(o)}(\hat{W},\hat{W}')$ is well defined because $d_{\mathcal{X},[n]}^{(o)}(\hat{W},\hat{W}')$ does not depend on $n\geq 1$ as long as $\hat{W},\hat{W}'\in\textstyle\DMC_{\mathcal{X},[n]}^{(o)}$. We can also express $d_{\mathcal{X},\ast}^{(o)}$ as follows:
$$d_{\mathcal{X},\ast}^{(o)}(\hat{W}_1,\hat{W}_2)=\sup_{\substack{m\geq 1,\\p\in\Delta_{[m]\times\mathcal{X}}}}|P_c(p,\hat{W}_1)-P_c(p,\hat{W}_2)|.$$

 It is easy to see that $d_{\mathcal{X},\ast}^{(o)}$ is a metric on $\DMC_{\mathcal{X},\ast}^{(o)}$. Let $\mathcal{T}_{\mathcal{X},\ast}^{(o)}$ be the metric topology on $\DMC_{\mathcal{X},\ast}^{(o)}$ that is induced by $d_{\mathcal{X},\ast}^{(o)}$.  We call $\mathcal{T}_{\mathcal{X},\ast}^{(o)}$ the \emph{noisiness topology} on $\DMC_{\mathcal{X},\ast}^{(o)}$. 

Clearly, $\mathcal{T}_{\mathcal{X},\ast}^{(o)}$ is natural because the restriction of $d_{\mathcal{X},\ast}^{(o)}$ on $\DMC_{\mathcal{X},[n]}^{(o)}$ is exactly $d_{\mathcal{X},[n]}^{(o)}$, and the topology induced by $d_{\mathcal{X},[n]}^{(o)}$ is $\mathcal{T}_{\mathcal{X},[n]}^{(o)}$. If $|\mathcal{X}|\geq 2$, Proposition \ref{propNaturalNotLocallyCompact} and Corollary \ref{corNaturalNotComplete} imply that $(\DMC_{\mathcal{X},\ast}^{(o)},d_{\mathcal{X},\ast}^{(o)})$ is not complete nor locally compact.

Since $\mathcal{T}_{s,\mathcal{X},\ast}^{(o)}$ is the finest natural topology, $\mathcal{T}_{s,\mathcal{X},\ast}^{(o)}$ is finer than $\mathcal{T}_{\mathcal{X},\ast}^{(o)}$. On the other hand, if $|\mathcal{X}|\geq 2$, $\mathcal{T}_{\mathcal{X},\ast}^{(o)}$ is metrizable and $\mathcal{T}_{s,\mathcal{X},\ast}^{(o)}$ is not (because it is not first-countable). Therefore, if $|\mathcal{X}|\geq 2$, the strong topology $\mathcal{T}_{s,\mathcal{X},\ast}^{(o)}$ is strictly finer than the noisiness topology $\mathcal{T}_{\mathcal{X},\ast}^{(o)}$.

It is worth mentioning that Propositions \ref{propCharacConvSeq} and \ref{propCharacCompactDMCXos} do not hold for $(\DMC_{\mathcal{X},\ast}^{(o)}, \mathcal{T}_{\mathcal{X},\ast}^{(o)})$. It is easy to find a rank-unbounded sequence $\{\hat{W}_n\}_{n\geq 0}$ which converges in $(\DMC_{\mathcal{X},\ast}^{(o)}, \mathcal{T}_{\mathcal{X},\ast}^{(o)})$ to a point $\hat{W}\in \DMC_{\mathcal{X},\ast}^{(o)}$. The set $\{\hat{W}_n:\;n\geq 0\}\cup\{\hat{W}\}$ is clearly compact and rank-unbounded.

\section{Topologies from Blackwell measures}

We saw in Section \ref{subsecStrongTop} that for every $\hat{W}\in\DMC_{\mathcal{X},\ast}^{(o)}$, a Blackwell measure ${\MP}_{\hat{W}}$ on $\Delta_{\mathcal{X}}$ is defined. Moreover, Proposition \ref{propCharacOutputEquiv} implies that $\hat{W}$ is uniquely determined by ${\MP}_{\hat{W}}$. Therefore, each $R_{\mathcal{X},\ast}^{(o)}$-equivalence class in $\DMC_{\mathcal{X},\ast}^{(o)}$ can be identified with its Blackwell measure. On the other hand, Proposition \ref{propCharacPostMetaProb} shows that the collection of Blackwell measures of the channels with input alphabet $\mathcal{X}$ is the same as the collection of balanced and finitely supported meta-probability measures on $\mathcal{X}$.

Therefore, the mapping $\hat{W}\rightarrow{\MP}_{\hat{W}}$ is a bijection from $\DMC_{\mathcal{X},\ast}^{(o)}$ to $\mathcal{MP}_{bf}(\mathcal{X})$. We call this mapping \emph{the canonical bijection} from $\DMC_{\mathcal{X},\ast}^{(o)}$ to $\mathcal{MP}_{bf}(\mathcal{X})$. Similarly, the inverse mapping is called the \emph{canonical bijection} from $\mathcal{MP}_{bf}(\mathcal{X})$ to $\DMC_{\mathcal{X},\ast}^{(o)}$.

Since $\Delta_{\mathcal{X}}$ is a metric space, there are many standard ways to construct topologies on $\mathcal{MP}(\mathcal{X})$. If we choose any of these standard topologies on $\mathcal{MP}(\mathcal{X})$ and then relativize it to the subspace $\mathcal{MP}_{bf}(\mathcal{X})$, we can construct topologies on $\DMC_{\mathcal{X},\ast}^{(o)}$ through the canonical bijection.

We saw in Section \ref{subsecConvMeasuresTopology} that there are three topologies that can be constructed on $\mathcal{MP}(\mathcal{X})$: the total variation topology, the strong convergence topology, and the weak-$\ast$ topology. But since every measure in $\mathcal{MP}_{bf}(\mathcal{X})$ is a finitely supported measure, strong convergence and total variation convergence are equivalent in $\mathcal{MP}_{bf}(\mathcal{X})$ (see Section \ref{subsecConvMeasuresTopology}). Therefore, it is sufficient to study the total-variation topology and the weak-$\ast$ topology. We will start by studying the weak-$\ast$ topology.

\subsection{Weak-$\ast$ topology}

\label{subsecWeakStar}

We first note that in the case of binary input channels, the weak-$\ast$ topology is equivalent to the topology induced by the convergence in distribution of $D$-densities (or $L$-densities, or $G$-densities) that was defined in \cite{RichardsonUrbanke}. Note also that the weak-$\ast$ topology is equivalent to the topology that is induced by the Le Cam deficiency distance \cite{LeCam}.

Consider the topology on $\DMC_{\mathcal{X},\ast}^{(o)}$ that is obtained by transporting the weak-$\ast$ topology from $\mathcal{MP}_{bf}(\mathcal{X})$ to $\DMC_{\mathcal{X},\ast}^{(o)}$ through the canonical bijection $F_{\can}$, i.e., we let $U\subset \DMC_{\mathcal{X},\ast}^{(o)}$ be open if and only if $F_{\can}^{-1}(U)$ is weakly-$\ast$ open. We will call this topology \emph{the weak-$\ast$ topology on $\DMC_{\mathcal{X},\ast}^{(o)}$}.

In this section, we show that the weak-$\ast$ topology is the same as the noisiness topology $\mathcal{T}_{\mathcal{X},\ast}^{(o)}$. We will show this using the Wasserstein metric.

Since $\Delta_{\mathcal{X}}$ is complete and separable, the $1^{st}$-Wasserstein distance metrizes the weak-$\ast$ topology \cite{WassersteinMetric}. Therefore, in order to show that the weak-$\ast$ topology and the noisiness topology $\mathcal{T}_{\mathcal{X},\ast}^{(o)}$ are the same, it is sufficient to show that the canonical bijection $F_{\can}$ from $(\mathcal{MP}_{bf}(\mathcal{X}),W_1)$ to $(\DMC_{\mathcal{X},\ast}^{(o)},d_{\mathcal{X},\ast}^{(o)})$ is a homeomorphism.

Note that since $\Delta_{\mathcal{X}}$ is compact, the metric space $(\mathcal{MP}(\mathcal{X}),W_1)$ is compact as well \cite{WassersteinMetric}.

\begin{mylem}
\label{lemDMCXoWasserstein}
For every $\hat{W},\hat{W}'\in\DMC_{\mathcal{X},\ast}^{(o)}$, we have $d_{\mathcal{X},\ast}^{(o)}(\hat{W},\hat{W}')\leq |\mathcal{X}|\cdot W_1({\MP}_{\hat{W}},{\MP}_{\hat{W}'})$.
\end{mylem}
\begin{proof}
See Appendix \ref{appDMCXoWasserstein}.
\end{proof}

\vspace*{3mm}

Lemma \ref{lemDMCXoWasserstein} can also be expressed as follows: for every ${\MP},{\MP}'\in\mathcal{MP}_{bf}(\mathcal{X})$, we have $d_{\mathcal{X},\ast}^{(o)}(F_{\can}({\MP}),F_{\can}({\MP}'))\leq |\mathcal{X}|\cdot W_1({\MP},{\MP}')$. This shows that the canonical bijection $F_{\can}$ is continuous. Therefore, the weak-$\ast$ topology is at least as strong as $\mathcal{T}_{\mathcal{X},\ast}^{(o)}$. It remains to show that $F_{\can}^{-1}$ is continuous. One approach to prove the continuity of $F_{\can}^{-1}$ is to find a lower bound of $d_{\mathcal{X},\ast}^{(o)}(\hat{W},\hat{W}')$ in terms of the Wasserstein metric, but this is tedious. We will follow another approach in order to show that the canonical bijection $F_{\can}$ is a homeomorphism. We need the following proposition:

\begin{myprop}
\label{propClosureFinSupported}
The  weak-$\ast$ closure of $\mathcal{MP}_{bf}(\mathcal{X})$ is $\mathcal{MP}_{b}(\mathcal{X})$.
\end{myprop}
\begin{proof}
See appendix \ref{appClosureFinSupported}.
\end{proof}

\begin{mythe}
\label{theTopEquivalenceNoisinesWeakStar}
The weak-$\ast$ topology on $\DMC_{\mathcal{X},\ast}^{(o)}$ is the same as the noisiness topology $\mathcal{T}_{\mathcal{X},\ast}^{(o)}$.
\end{mythe}
\begin{proof}
Let $(\overline{\DMC}_{\mathcal{X},\ast}^{(o)}, \overline{d}_{\mathcal{X},\ast}^{(o)})$ be a completion of $(\DMC_{\mathcal{X},\ast}^{(o)},d_{\mathcal{X},\ast}^{(o)})$. Since $\mathcal{MP}_{b}(\mathcal{X})$ is the weak-$\ast$ closure of $\mathcal{MP}_{bf}(\mathcal{X})$ (Proposition \ref{propClosureFinSupported}), we can extend the canonical bijection $F_{\can}: \mathcal{MP}_{bf}(\mathcal{X})\rightarrow \DMC_{\mathcal{X},\ast}^{(o)}$ to a mapping $\overline{F}: \mathcal{MP}_{b}(\mathcal{X})\rightarrow \overline{\DMC}_{\mathcal{X},\ast}^{(o)}$ as follows:
\begin{equation}
\label{eqExtensionFunction}
\overline{F}({\MP})=\lim_{n\to\infty}F_{\can}({\MP}_n),
\end{equation}
where $({\MP}_n)_{n\geq 0}$ is any sequence in $\mathcal{MP}_{bf}(\mathcal{X})$ that converges to ${\MP}\in \mathcal{MP}_{b}(\mathcal{X})$, and where the limit in \eqref{eqExtensionFunction} is taken inside $\overline{\DMC}_{\mathcal{X},\ast}^{(o)}$. In order to show that $\overline{F}$ is well defined, we have to make sure that the limit in \eqref{eqExtensionFunction} exists and that it does not depend on the sequence $({\MP}_n)_{n\geq 0}$.

Since the sequence $({\MP}_n)_{n\geq 0}$ converges, it is a Cauchy sequence. Therefore, for every $\epsilon>0$ there exists $n_0>0$ such that for every $n_1,n_2\geq 1$ we have $\displaystyle W_1({\MP}_{n_1},{\MP}_{n_2})<\frac{\epsilon}{|\mathcal{X}|}$. By Lemma \ref{lemDMCXoWasserstein}, we have $$ \overline{d}_{\mathcal{X},\ast}^{(o)}(F_{\can}({\MP}_{n_1}),F_{\can}({\MP}_{n_2}))=d_{\mathcal{X},\ast}^{(o)}(F_{\can}({\MP}_{n_1}),F_{\can}({\MP}_{n_2}))\leq |\mathcal{X}|\cdot W_1({\MP}_{n_1},{\MP}_{n_2})<\epsilon.$$
Therefore, $(F_{\can}({\MP}_n))_{n\geq 0}$ is a Cauchy sequence in $(\overline{\DMC}_{\mathcal{X},\ast}^{(o)}, \overline{d}_{\mathcal{X},\ast}^{(o)})$ which is complete, hence the limit in \eqref{eqExtensionFunction} exists. Now assume that $({\MP}_n')_{n\geq 0}$ is another sequence in $\mathcal{MP}_{bf}(\mathcal{X})$ which converges to ${\MP}$. We have:
\begin{align*}
\lim_{n\to\infty} \overline{d}_{\mathcal{X},\ast}^{(o)}\big(F_{\can}({\MP}_n),F_{\can}({\MP}_n')\big)&=\lim_{n\to\infty} d_{\mathcal{X},\ast}^{(o)}\big(F_{\can}({\MP}_n),F_{\can}({\MP}_n')\big)\\
&\stackrel{(a)}{\leq} \lim_{n\to\infty} |\mathcal{X}|\cdot W_1({\MP}_n,{\MP}_n')\stackrel{(b)}{=}0,
\end{align*}
where (a) follows from Lemma \ref{lemDMCXoWasserstein} and (b) follows from the fact that $({\MP}_n)_{n\geq 0}$ and $({\MP}_n')_{n\geq 0}$ converge to the same point. Therefore, $(F_{\can}({\MP}_n))_{n\geq 0}$ and $(F_{\can}({\MP}_n'))_{n\geq 0}$ converge to the same point in $\overline{\DMC}_{\mathcal{X},\ast}^{(o)}$. We conclude that $\overline{F}$ is well defined.

Now fix ${\MP},{\MP}'\in\mathcal{MP}_b(\mathcal{X})$ and let $({\MP}_n)_{n\geq 0}$ and $({\MP}_n')_{n\geq 0}$ be two sequences in $\mathcal{MP}_{bf}(\mathcal{X})$ that converge to ${\MP}$ and ${\MP}'$ respectively. We have:
\begin{align*}
\overline{d}_{\mathcal{X},\ast}^{(o)}\left(\overline{F}({\MP}),\overline{F}({\MP}')\right)&=\overline{d}_{\mathcal{X},\ast}^{(o)}\left(\lim_{n\to\infty}F_{\can}({\MP}_n),\lim_{n\to\infty}F_{\can}({\MP}'_n)\right)\\
&\stackrel{(a)}{=}\lim_{n\to\infty}\overline{d}_{\mathcal{X},\ast}^{(o)}(F_{\can}({\MP}_n),F_{\can}({\MP}_n'))\\
&=\lim_{n\to\infty}d_{\mathcal{X},\ast}^{(o)}(F_{\can}({\MP}_n),F_{\can}({\MP}_n'))\\
&\stackrel{(b)}{\leq}\lim_{n\to\infty}|\mathcal{X}|\cdot W_1({\MP}_n,{\MP}_n')\stackrel{(c)}{=}|\mathcal{X}|\cdot W_1({\MP},{\MP}'),
\end{align*}
where (a) and (c) follow from the fact that metric distances are continuous, and (b) follows from Lemma \ref{lemDMCXoWasserstein}. Therefore, $\overline{F}$ is continuous from $(\mathcal{MP}_b(\mathcal{X}),W_1)$ to $(\overline{\DMC}_{\mathcal{X},\ast}^{(o)}, \overline{d}_{\mathcal{X},\ast}^{(o)})$. Moreover, since $\mathcal{MP}_b(\mathcal{X})$ is weakly-$\ast$ closed in $\mathcal{MP}(\mathcal{X})$ which is compact, $\mathcal{MP}_b(\mathcal{X})$ is compact under the weak-$\ast$ topology. Therefore for every weakly-$\ast$ closed subset $A$ of $\mathcal{MP}_b(\mathcal{X})$, $A$ is compact and so $\overline{F}(A)$ is compact in $(\overline{\DMC}_{\mathcal{X},\ast}^{(o)}, \overline{d}_{\mathcal{X},\ast}^{(o)})$ which is Hausdorff. This implies that $\overline{F}(A)$ is closed in $(\overline{\DMC}_{\mathcal{X},\ast}^{(o)}, \overline{d}_{\mathcal{X},\ast}^{(o)})$ for every weakly-$\ast$ closed subset $A$ of $\mathcal{MP}_b(\mathcal{X})$. Therefore, $\overline{F}$ is both continuous and closed. In particular, $\overline{F}(\mathcal{MP}_b(\mathcal{X}))$ is closed in $(\overline{\DMC}_{\mathcal{X},\ast}^{(o)}, \overline{d}_{\mathcal{X},\ast}^{(o)})$. But $\overline{F}(\mathcal{MP}_b(\mathcal{X}))\supset \overline{F}(\mathcal{MP}_{bf}(\mathcal{X}))=F_{\can}(\mathcal{MP}_{bf}(\mathcal{X}))=\DMC_{\mathcal{X},\ast}^{(o)}$, and $\DMC_{\mathcal{X},\ast}^{(o)}$ is dense in $(\overline{\DMC}_{\mathcal{X},\ast}^{(o)}, \overline{d}_{\mathcal{X},\ast}^{(o)})$. Therefore, we must have $\overline{F}(\mathcal{MP}_b(\mathcal{X}))=\overline{\DMC}_{\mathcal{X},\ast}^{(o)}$. We conclude that $\overline{F}$ is a homeomorphism from $(\mathcal{MP}_b(\mathcal{X}),W_1)$ to $(\overline{\DMC}_{\mathcal{X},\ast}^{(o)}, \overline{d}_{\mathcal{X},\ast}^{(o)})$.

Now since $\overline{F}\big(\mathcal{MP}_{bf}(\mathcal{X})\big)=\DMC_{\mathcal{X},\ast}^{(o)}$, the restriction of $\overline{F}$ to $\mathcal{MP}_{bf}(\mathcal{X})$ is a homeomorphism from $(\mathcal{MP}_{bf}(\mathcal{X}),W_1)$ to $(\DMC_{\mathcal{X},\ast}^{(o)},d_{\mathcal{X},\ast}^{(o)})$. But the restriction of $\overline{F}$ to $\mathcal{MP}_{bf}(\mathcal{X})$ is nothing but $F_{\can}$. We conclude that the canonical bijection is a homeomorphism from $(\mathcal{MP}_{bf}(\mathcal{X}),W_1)$ to $(\DMC_{\mathcal{X},\ast}^{(o)},d_{\mathcal{X},\ast}^{(o)})$. Therefore, the weak-$\ast$ topology on $\DMC_{\mathcal{X},\ast}^{(o)}$ is the same as the noisiness topology $\mathcal{T}_{\mathcal{X},\ast}^{(o)}$.
\end{proof}

\vspace*{3mm}

Since $(\mathcal{MP}_b(\mathcal{X}),W_1)$ is homeomorphic to $(\overline{\DMC}_{\mathcal{X},\ast}^{(o)}, \overline{d}_{\mathcal{X},\ast}^{(o)})$, we can interpret this by saying that $\overline{\DMC}_{\mathcal{X},\ast}^{(o)}$ is the space of all equivalent channels with input alphabet $\mathcal{X}$ and arbitrary output alphabet (with arbitrary cardinality). Moreover, since $\DMC_{\mathcal{X},\ast}^{(o)}$ is dense in $(\overline{\DMC}_{\mathcal{X},\ast}^{(o)}, \overline{d}_{\mathcal{X},\ast}^{(o)})$, we can say that any channel with input alphabet $\mathcal{X}$ can be approximated in the noisiness/weak-$\ast$ sense by a channel having a finite output alphabet.

\subsection{Total variation topology}
The \emph{total-variation metric distance} $d_{TV,\mathcal{X},\ast}^{(o)}$ on $\DMC_{\mathcal{X},\ast}^{(o)}$ is defined as
$$d_{TV,\mathcal{X},\ast}^{(o)}(\hat{W},\hat{W}')=\|{\MP}_{\hat{W}}-{\MP}_{\hat{W}'}\|_{TV}.$$

The \emph{total-variation topology} $\mathcal{T}_{TV,\mathcal{X},\ast}^{(o)}$ is the metric topology that is induced by $d_{TV,\mathcal{X},\ast}^{(o)}$ on $\DMC_{\mathcal{X},\ast}^{(o)}$. We will refer to the open sets (respectively, closed sets, compact sets, \ldots) of $\mathcal{T}_{TV,\mathcal{X},\ast}^{(o)}$ as TV-open (respectively, TV-closed, TV-compact, \ldots). The same notation is also used for open sets of $\mathcal{MP}_{bf}(\mathcal{X})$, $\mathcal{MP}_{b}(\mathcal{X})$ and $\mathcal{MP}(\mathcal{X})$ in the total variation topology.

\begin{myprop}
\label{propTotalVarNotNatural}
If $|\mathcal{X}|\geq 2$ and $n\geq 2$, then $\DMC_{\mathcal{X},[n]}^{(o)}$ is not TV-compact in $\DMC_{\mathcal{X},\ast}^{(o)}$.
\end{myprop}
\begin{proof}
Let $p,p'\in\Delta_{\mathcal{X}}$ be such that $p\neq p'$ and $\displaystyle\frac{1}{2}p+\frac{1}{2}p'=\pi_{\mathcal{X}}$, where $\pi_{\mathcal{X}}$ is the uniform distribution on $\mathcal{X}$. For every $n\geq 1$, define $p_n,p_n'\in\Delta_{\mathcal{X}}$ as
$$p_n=\frac{1}{n}p+\left(1-\frac{1}{n}\right)\pi_{\mathcal{X}},$$
and
$$p_n'=\frac{1}{n}p'+\left(1-\frac{1}{n}\right)\pi_{\mathcal{X}}.$$
Clearly, $\displaystyle\frac{1}{2}p_n+\frac{1}{2}p'_n=\pi_{\mathcal{X}}$ for every $n\geq 1$.

Now let $\MP_n=\frac{1}{2}\delta_{p_n}+\frac{1}{2}\delta_{p_n'}$, where $\delta_{p_n}$ and $\delta_{p_n'}$ are Dirac measures centered at $p_n$ and $p_n'$ respectively. Clearly, $\MP_n$ is balanced and finitely supported for every $n\geq 1$. Let $\hat{W}_n=F_{\can}(\MP_n)$. We have $$|\supp({\MP}_{\hat{W}_n})|=|\supp({\MP}_n)|=|\{p_n,p_n'\}|=2.$$ Therefore, $\hat{W}_n\in\DMC_{\mathcal{X},[2]}^{(o)}\subset \DMC_{\mathcal{X},[m]}^{(o)}$ for every $n\geq 1$ and every $m\geq 2$. It is easy to see that $d_{TV,\mathcal{X},\ast}^{(o)}(\hat{W}_{n_1},\hat{W}_{n_2})=\|\MP_{n_1}-\MP_{n_2}\|_{TV}=1$ for every $n_2>n_1\geq 1$. Therefore, no subsequence of $(\MP_n)_{n\geq 1}$ can converge. This means that $\DMC_{\mathcal{X},[m]}^{(o)}$ is not sequentially compact for any $m\geq 2$. Now since $\mathcal{T}_{TV,\mathcal{X},\ast}^{(o)}$ is metrizable, we conclude that $\DMC_{\mathcal{X},[n]}^{(o)}$ is not compact for any $n\geq 2$.
\end{proof}

\begin{mycor}
If $|\mathcal{X}|\geq 2$, then $\mathcal{T}_{TV,\mathcal{X},\ast}^{(o)}$ is not a natural topology.
\end{mycor}
\begin{proof}
If $\mathcal{T}_{TV,\mathcal{X},\ast}^{(o)}$ were natural, $\DMC_{\mathcal{X},[2]}^{(o)}$ would be compact, and this is not the case.
\end{proof}

\vspace*{3mm}

Since the noisiness topology is the same as the weak-$\ast$ topology, $\mathcal{T}_{\mathcal{X},\ast}^{(o)}$ is coarser than $\mathcal{T}_{TV,\mathcal{X},\ast}^{(o)}$. On the other hand, since $\mathcal{T}_{\mathcal{X},\ast}^{(o)}$ is natural and $\mathcal{T}_{TV,\mathcal{X},\ast}^{(o)}$ is not, $\mathcal{T}_{\mathcal{X},\ast}^{(o)}$ is strictly coarser than $\mathcal{T}_{TV,\mathcal{X},\ast}^{(o)}$ when $|\mathcal{X}|\geq 2$.

Note that the sequence $(\MP_n)_{n\geq 1}$ in the proof of Proposition \ref{propTotalVarNotNatural} converges in the strong topology because of Proposition \ref{propCharacConvSeq}. Therefore, $\mathcal{T}_{s,\mathcal{X},\ast}^{(o)}$ is not finer than $\mathcal{T}_{TV,\mathcal{X},\ast}^{(o)}$.

Although $\mathcal{T}_{TV,\mathcal{X},\ast}^{(o)}$ is not a natural topology itself, it has many properties of natural topologies.

\begin{myprop}
\label{propTVOpenRankUnbounded}
If $|\mathcal{X}|\geq 2$, every non-empty TV-open subset of $\DMC_{\mathcal{X},\ast}^{(o)}$ is rank-unbounded.
\end{myprop}
\begin{proof}
Let $U$ be a non-empty TV-open set of $\DMC_{\mathcal{X},\ast}^{(o)}$. Let $\hat{W}\in U$ and let $\epsilon>0$ be such that $\hat{W}'\in U$ whenever $d_{TV,\mathcal{X},\ast}^{(o)}(\hat{W},\hat{W}')<\epsilon$.

Let $p$, $p'$, $(p_n)_{n\geq 1}$ and $(p_n')_{n\geq 1}$ be as in Proposition \ref{propTotalVarNotNatural}. For every $n\geq 1$, define $\MP_n\in\mathcal{MP}(\mathcal{X})$ as follows:
$${\MP}_n=\left(1-\frac{\epsilon}{4n}\right){\MP}_{\hat{W}} +\frac{\epsilon}{8n^2}\cdot\sum_{i=1}^n (\delta_{p_i}+\delta_{p_i'}).$$
Clearly, $\MP_n$ is balanced and finitely supported, so $\MP_n\in \mathcal{MP}_{bf}(\mathcal{X})$. Moreover, $$d_{TV,\mathcal{X},\ast}^{(o)}(F_{\can}({\MP}_n),\hat{W}) =\|{\MP}_n-{\MP}_{\hat{W}}\|_{TV}\leq\frac{\epsilon}{2n}<\epsilon.$$ Therefore, $F_{\can}(\MP_n)\in U$ for every $n\geq 1$. On the other hand, $\supp(\MP_n)\supset\{p_i,p_i':\; 1\leq i\leq n\}$, which means that $|\supp(\MP_n)|\geq 2n$ and so $F_{\can}(\MP_n)\notin\DMC_{\mathcal{X},[n]}^{(o)}$ for every $n\geq 1$. We conclude that $U$ is rank-unbounded.
\end{proof}

\begin{mycor}
If $|\mathcal{X}|\geq 2$, the TV-interior of $\DMC_{\mathcal{X},[n]}^{(o)}$ in $\DMC_{\mathcal{X},\ast}^{(o)}$ is empty.
\end{mycor}

\vspace*{3mm}

Note that the sequence $(F_{\can}(\MP_n))_{n\geq 1}$ in the proof of Proposition \ref{propTVOpenRankUnbounded} is rank-unbounded and converges in total variation to $\hat{W}$. On the other hand, Proposition \ref{propCharacConvSeq} implies that $(F_{\can}(\MP_n))_{n\geq 1}$ does not converge in $(\DMC_{\mathcal{X},\ast}^{(o)},\mathcal{T}_{s,\mathcal{X},\ast}^{(o)})$. We conclude that $\mathcal{T}_{TV,\mathcal{X},\ast}^{(o)}$ is not finer than $\mathcal{T}_{s,\mathcal{X},\ast}^{(o)}$.

Although $\DMC_{\mathcal{X},[n]}^{(o)}$ is not TV-compact if $|\mathcal{X}|\geq 2$ and $n\geq 2$, it is TV-complete:
\begin{myprop}
For every $n\geq 1$, $\DMC_{\mathcal{X},[n]}^{(o)}$ is TV-complete in $\DMC_{\mathcal{X},\ast}^{(o)}$.
\end{myprop}
\begin{proof}
Let $\mathcal{MP}_{b,n}(\mathcal{X})$ be the set of balanced meta-probability measures whose support is of size at most $n$:
$$\mathcal{MP}_{b,n}(\mathcal{X})=\{\MP\in\mathcal{MP}_b(\mathcal{X}):\;|\supp(\MP)|\leq n\}.$$
Since $(\DMC_{\mathcal{X},[n]}^{(o)},d_{TV,\mathcal{X},\ast}^{(o)})$ is isometric to $(\mathcal{MP}_{b,n}(\mathcal{X}),\|\cdot\|_{TV})$, and since $(\mathcal{MP}(\mathcal{X}),\|\cdot\|_{TV})$ is complete, it is sufficient to show that $\mathcal{MP}_{b,n}(\mathcal{X})$ is TV-closed in $\mathcal{MP}(\mathcal{X})$.

Let $\MP$ be in the TV-closure of $\mathcal{MP}_{b,n}(\mathcal{X})$. Since we are working in a metric space, there exists a sequence $(\MP_m)_{m\geq 0}$ in $\mathcal{MP}_{b,n}(\mathcal{X})$ that TV-converges to $\MP$. Assume that $\MP\notin\mathcal{MP}_{b,n}(\mathcal{X})$. There exist $p_1,\ldots,p_{n+1}\in\Delta_{\mathcal{X}}$ that are pairwise different and which satisfy $\MP(p_i)>0$ for every $1\leq i\leq n+1$. Since $(\MP_m)_{m\geq 0}$ TV-converges to $\MP$, there exists $m_0\geq 0$ such that $\MP_{m_0}(p_i)>0$ for every $1\leq i\leq n+1$. This contradicts the fact $\MP_{m_0}\in\mathcal{MP}_{b,n}(\mathcal{X})$. Therefore, $\MP\in\mathcal{MP}_{b,n}(\mathcal{X})$ for every $\MP$ in the TV-closure of $\mathcal{MP}_{b,n}(\mathcal{X})$. This shows that $\mathcal{MP}_{b,n}(\mathcal{X})$ is TV-closed. Therefore, $\DMC_{\mathcal{X},[n]}^{(o)}$ is TV-complete in $\DMC_{\mathcal{X},\ast}^{(o)}$.
\end{proof}

\begin{myprop}
If $|\mathcal{X}|\geq 2$, $(\DMC_{\mathcal{X},\ast}^{(o)},\mathcal{T}_{TV,\mathcal{X},\ast}^{(o)})$ is neither Baire nor locally compact anywhere.
\end{myprop}
\begin{proof}
Since $\DMC_{\mathcal{X},[n]}^{(o)}$ is TV-complete, it is TV-closed. Since it also has empty TV-interior, the same techniques that were used for natural topologies in Section \ref{subsecNaturalTop} can be applied for $\mathcal{T}_{TV,\mathcal{X},\ast}^{(o)}$.
\end{proof}
\vspace*{3mm}

The above proposition shows that $(\DMC_{\mathcal{X},\ast}^{(o)},\mathcal{T}_{TV,\mathcal{X},\ast}^{(o)})$ cannot be completely metrized. Note that since $(\DMC_{\mathcal{X},\ast}^{(o)},d_{TV,\mathcal{X},\ast}^{(o)})$ is isometric to $(\mathcal{MP}_{bf}(\mathcal{X}),\|\cdot\|_{TV})$, and since $(\mathcal{MP}(\mathcal{X}),\|\cdot\|_{TV})$ is complete, the completion of $(\DMC_{\mathcal{X},\ast}^{(o)},d_{TV,\mathcal{X},\ast}^{(o)})$ is isometric to the closure of $\mathcal{MP}_{bf}(\mathcal{X})$ in $(\mathcal{MP}(\mathcal{X}),\|\cdot\|_{TV})$. It can be shown that the TV-closure of $\mathcal{MP}_{bf}(\mathcal{X})$ in $\mathcal{MP}(\mathcal{X})$ is the set of all balanced and countably supported meta-probability measures on $\mathcal{X}$. Therefore, the completion of $(\DMC_{\mathcal{X},\ast}^{(o)},d_{TV,\mathcal{X},\ast}^{(o)})$ can be thought of as the space of equivalent channels from $\mathcal{X}$ to a countably infinite output alphabet. This allows us to say that any channel with input alphabet $\mathcal{X}$ and a countable output alphabet can be approximated in the total variation sense by a channel having a finite output alphabet.

\section{The natural Borel $\sigma$-algebra on $\DMC_{\mathcal{X},\ast}^{(o)}$}

Let $\mathcal{T}$ be a Hausdorff natural topology on $\DMC_{\mathcal{X},\ast}^{(o)}$. Since $\mathcal{T}_{s,\mathcal{X},\ast}^{(o)}$ is the finest natural topology, we have $\mathcal{T}\subset \mathcal{T}_{s,\mathcal{X},\ast}^{(o)}$. Therefore, $\mathcal{B}(\mathcal{T})\subset \mathcal{B}(\mathcal{T}_{s,\mathcal{X},\ast}^{(o)})$.

On the other hand, for every $U\in \mathcal{T}_{s,\mathcal{X},\ast}^{(o)}$ and every $n\geq 1$, we have $U\cap\DMC_{\mathcal{X},[n]}^{(o)}\in\mathcal{T}_{\mathcal{X},[n]}^{(o)}$. But $\mathcal{T}$ is a natural topology, so there must exist $U_n\in \mathcal{T}$ such that $U_n\cap\DMC_{\mathcal{X},[n]}^{(o)}=U\cap\DMC_{\mathcal{X},[n]}^{(o)}$. Since $U_n\in\mathcal{T}$, we have $U_n\in\mathcal{B}(\mathcal{T})$. Moreover, $\DMC_{\mathcal{X},[n]}^{(o)}$ is $\mathcal{T}$-closed (because it is compact and $\mathcal{T}$ is Hausdorff). Therefore, $\DMC_{\mathcal{X},[n]}^{(o)}\in\mathcal{B}(\mathcal{T})$. This implies that $U\cap\DMC_{\mathcal{X},[n]}^{(o)}=U_n\cap\DMC_{\mathcal{X},[n]}^{(o)}\in\mathcal{B}(T)$, hence
$$U=\bigcup_{n\geq 1}(U\cap {\DMC}_{\mathcal{X},[n]}^{(o)})\in\mathcal{B}(T).$$
Since this is true for every $U\in \mathcal{T}_{s,\mathcal{X},\ast}^{(o)}$, we have $\mathcal{T}_{s,\mathcal{X},\ast}^{(o)}\subset\mathcal{B}(\mathcal{T})$ which implies that $\mathcal{B}(\mathcal{T}_{s,\mathcal{X},\ast}^{(o)})\subset\mathcal{B}(\mathcal{T})$. We conclude that all Hausdorff natural topologies on $\DMC_{\mathcal{X},\ast}^{(o)}$ have the same $\sigma$-algebra.  This $\sigma$-algebra deserves to be called the \emph{natural Borel $\sigma$-algebra} on $\DMC_{\mathcal{X},\ast}^{(o)}$.

Note that for every $n\geq 1$, the inclusion mapping $i_n: \DMC_{\mathcal{X},[n]}^{(o)}\rightarrow\DMC_{\mathcal{X},\ast}^{(o)}$ is continuous from $(\DMC_{\mathcal{X},[n]}^{(o)},\mathcal{T}_{\mathcal{X},[n]}^{(o)})$ to $(\DMC_{\mathcal{X},\ast}^{(o)},\mathcal{T}_{s,\mathcal{X},\ast}^{(o)})$, hence it is measurable. Therefore, for every $B\in \mathcal{B}(\mathcal{T}_{s,\mathcal{X},\ast}^{(o)})$, we have $i_n^{-1}(B)=B\cap\DMC_{\mathcal{X},[n]}^{(o)}\in\mathcal{B}(\mathcal{T}_{\mathcal{X},[n]}^{(o)})$. In the following, we show a converse for this statement.

Fix $n\geq 1$ and let $U\in\mathcal{T}_{\mathcal{X},[n]}^{(o)}$. There exists $U'\in\mathcal{T}_{s,\mathcal{X},\ast}^{(o)}$ such that $U=U'\cap\DMC_{\mathcal{X},[n]}^{(o)}$. Since $U'$ and $\DMC_{\mathcal{X},[n]}^{(o)}$ are respectively open and closed in the topology $\mathcal{T}_{s,\mathcal{X},\ast}^{(o)}$, they are both in its Borel $\sigma$-algebra. Therefore, $U=U'\cap\DMC_{\mathcal{X},[n]}^{(o)}\in\mathcal{B}(\mathcal{T}_{s,\mathcal{X},\ast}^{(o)})$ for every $U\in\mathcal{T}_{\mathcal{X},[n]}^{(o)}$. This means that $\mathcal{T}_{\mathcal{X},[n]}^{(o)}\subset \mathcal{B}(\mathcal{T}_{s,\mathcal{X},\ast}^{(o)})$ and $\mathcal{B}(\mathcal{T}_{\mathcal{X},[n]}^{(o)})\subset \mathcal{B}(\mathcal{T}_{s,\mathcal{X},\ast}^{(o)})$ for every $n\geq 1$.

Assume now that $A\subset\DMC_{\mathcal{X},\ast}^{(o)}$ satisfies $A\cap\DMC_{\mathcal{X},[n]}^{(o)} \in\mathcal{B}(\mathcal{T}_{\mathcal{X},[n]}^{(o)})$ for every $n\geq 1$. This implies that $A\cap\DMC_{\mathcal{X},[n]}^{(o)} \in\mathcal{B}(\mathcal{T}_{s,\mathcal{X},\ast}^{(o)})$ for every $n\geq 1$, hence
$$A=\bigcup_{n\geq 1}(A\cap{\DMC}_{\mathcal{X},[n]}^{(o)}) \in\mathcal{B}(\mathcal{T}_{s,\mathcal{X},\ast}^{(o)}).$$
We conclude that a subset $A$ of $\DMC_{\mathcal{X},\ast}^{(o)}$ is in the natural Borel $\sigma$-algebra if and only if $A\cap\DMC_{\mathcal{X},[n]}^{(o)} \in\mathcal{B}(\mathcal{T}_{\mathcal{X},[n]}^{(o)})$ for every $n\geq 1$.

\section{Conclusion}

The fact that the noisiness and weak-$\ast$ topologies are the same gives us more freedom in proving theorems. Statements that can be hard to prove using the weak-$\ast$ formulation might be easier to prove using the noisiness formulation. For example, the convergence of the polarization process is slightly easier to prove in the noisiness formulation \cite{RajPolarConvTop}.

The strong topology is too strong to be adopted as the ``standard natural topology". However, it can still be useful because it is relatively easy to work with as it has a quotient formulation. Moreover, since it is finer than the noisiness/weak-$\ast$ topology, many statements that are true for the strong topology are also true for coarser topologies, e.g., any sequence that converges in the strong topology also converges in the noisiness/weak-$\ast$ topology.

Although the total variation topology is not natural, it can still be useful because it is finer than the noisiness/weak-$\ast$ topology.

Many interesting questions remain open: Are all natural topologies Hausdorff? Can we find more topological properties that are common for all natural topologies? Is there a coarsest natural topology? Is there a natural topology that is coarser than the noisiness/weak-$\ast$ one?

Finding meaningful measures on $\DMC_{\mathcal{X},\ast}^{(o)}$ might be challenging. One might be tempted to require that the measure of $\DMC_{\mathcal{X},[n]}^{(o)}$ should be zero because it is ``finite dimensional" whereas $\DMC_{\mathcal{X},\ast}^{(o)}$ is ``infinite dimensional". On the other hand, if $\DMC_{\mathcal{X},[n]}^{(o)}$ has a zero measure for every $n\geq 1$, the whole space $\DMC_{\mathcal{X},\ast}^{(o)}$ will have a zero measure because it is a countable union of these subspaces. Nevertheless, statements such as ``the property X is true for almost all channels" can still make sense. One possible definition of null-sets is as follows: for every set $A$ in the natural Borel $\sigma$-algebra, we say that $A$ is a null-set if and only if there exists $n_0\geq 1$ such that $P_n\left(\Proj_n^{-1}(A\cap\DMC_{\mathcal{X},[n]}^{(o)})\right)=0$ for every $n\geq n_0$, where $\Proj_n$ is the projection onto the $R_{\mathcal{X},[n]}^{(o)}$-equivalence classes and $P_n$ is the uniform probability measure on $\DMC_{\mathcal{X},[n]}\equiv (\Delta_{[n]})^{\mathcal{X}}$. Another possible definition, which is weaker, is to say that $A$ is a null-set if and only if $\displaystyle \lim_{n\to\infty}P_n\left({\Proj}_n^{-1}(A\cap{\DMC}_{\mathcal{X},[n]}^{(o)})\right)=0$.

It is worth mentioning that the standard weak-$\ast$ (and in particular the $L$-density) approach is not possible in the quantum setting because there is no quantum analogue for the conditional probability of the input given the output. On the other hand, since the strong topology and the noisiness metric are defined in terms of the forward transition probabilities, they can be generalized to the quantum setting.

Another notion of equivalence is the Shannon-equivalence that allows randomization at both the input and the output, as well as shared randomness between the transmitter and the receiver \cite{ShannonDegrad}. The Shannon deficiency that was introduced in \cite{RaginskyShannon} compares a particular channel with the Shannon-equivalence-class of another channel, but it is not a metric distance between Shannon-equivalence-classes. In \cite{RajShan}, we provide a characterization of the Shannon ordering and we prove that some of the results of this paper holds for the space of Shannon-equivalent channels.

In \cite{RajInputDegrad}, we introduce the notions of input-degradedness and input-equivalence. A channel $W$ is said to be input-degraded from another channel $W'$ if $W$ can be simulated from $W'$ by local operations at the input. In \cite{RajInputDegrad}, we provide a characterization of input-degradedness and and we prove that many of the results of this paper hold for the space of input-equivalent channels.

\section*{Acknowledgment}

I would like to thank Emre Telatar and Mohammad Bazzi for helpful discussions. I am also grateful to Maxim Raginsky for his comments.

\appendices

\section{Proof of Proposition \ref{propCharacOutputEquiv}}
\label{appCharacOutputEquiv}

For every $A\subset \Delta_{\mathcal{X}}$, let $\conv(A)$ be the convex hull of $A$. We say that $p\in A$ is \emph{convex-extreme} if it is an extreme point of $\conv(A)$, i.e., for every $p_1,\ldots,p_n\in\conv(A)$ and every $\lambda_1,\ldots,\lambda_n>0$ satisfying $\displaystyle\sum_{i=1}^n \lambda_i=1$ and $\displaystyle\sum_{i=1}^n \lambda_ip_i=p$, we have $p_1=\ldots=p_n=p$. It is easy to see that if $A$ is finite, then the convex-extreme points of $A$ coincide with the extreme points of $\conv(A)$. We denote the set of convex-extreme points of $A$ as $\CE(A)$.

Let $W\in\DMC_{\mathcal{X},\mathcal{Y}}$ and $W'\in\DMC_{\mathcal{X},\mathcal{Z}}$ be such that $W'$ is degraded from $W$. There exists $V\in\DMC_{\mathcal{Y},\mathcal{Z}}$ such that $W'=V\circ W$. Let $X$ be a random variable uniformly distributed in $\mathcal{X}$, let $Y$ be the output of $W$ when $X$ is the input, and let $Z$ be the output of $V$ when $Y$ is the input in such a way that $X-Y-Z$ is a Markov chain. Clearly, $P_{Z|X}(z|x)=W'(z|x)$ for every $(x,z)\in\mathcal{X}\times\mathcal{Z}$.

For every $z\in\mathcal{Z}$, we have:
\begin{equation}
\label{eqPWprimeoz}
P_{W'}^o(z)=P_Z(z)=\sum_{y\in\mathcal{Y}}P_Y(y)P_{Z|Y}(z|y)=\sum_{y\in\Imag(W)} V(z|y) P_W^o(y).
\end{equation}

Define $V^{-1}\in\DMC_{\Imag(W'),\Imag(W)}$ as $$V^{-1}(y|z)=P_{Y|Z}(y|z)=\frac{P_{Y}(y)P_{Z|Y}(z|y)}{P_Z(z)}=\frac{V(z|y) P_W^o(y)}{\displaystyle\sum_{y'\in\Imag(W)}V(z|y') P_W^o(y')}.$$
Note that for every $(y,z)\in\Imag(W)\times\Imag(W')$, we have $V^{-1}(y|z)=0$ if and only if $V(z|y)=0$.

For every $(x,z)\in\mathcal{X}\times\Imag(W')$, we have:
\begin{equation}
\begin{aligned}
W'^{-1}_z(x)&=P_{X|Z}(x|z)=\sum_{y\in\mathcal{Y}}P_{X,Y|Z}(x,y|z)=\sum_{\substack{y\in\mathcal{Y},\\P_Y(y)>0}}P_{X,Y|Z}(x,y|z)\\
&=\sum_{y\in\Imag(W)}P_{Y|Z}(y|z)P_{X|Y,Z}(x|y,z)\stackrel{(a)}{=}\sum_{y\in\Imag(W)}V^{-1}(y|z)P_{X|Y}(x|y)\\
&=\sum_{y\in\Imag(W)}V^{-1}(y|z)W_y^{-1}(x),
\end{aligned}
\label{eqConvexHullWzm1}
\end{equation}
where (a) follows from the fact that $X-Y-Z$ is a Markov chain.

Equation \eqref{eqConvexHullWzm1} shows that for every $z\in\Imag(W')$, we have $$W'^{-1}_z\in\conv(\{W^{-1}_y:\;y\in\Imag(W)\})=\conv(\supp({\MP}_W)).$$ Therefore,
\begin{equation}
\label{eqConvexHullDegrad}
\conv(\supp({\MP}_{W'}))=\conv(\{W'^{-1}_z:\;z\in\Imag(W')\})\subset \conv(\supp({\MP}_W)).
\end{equation}

\vspace*{3mm}

Now for every $p\in\Delta_{\mathcal{X}}$, define $$\mathcal{Y}_p:=\{y\in\Imag(W):\; W^{-1}_y=p\}.$$
Similarly, $$\mathcal{Z}_p:=\{z\in\Imag(W'):\; W'^{-1}_z=p\}.$$

Let $p_{ext}\in\CE(\supp({\MP}_W))$ and let $z\in\Imag(W')$. Equation \eqref{eqConvexHullWzm1} shows that if $z\in\mathcal{Z}_{p_{ext}}$, then $V^{-1}(y|z)=0$ for every $y\in \Imag(W)\setminus\mathcal{Y}_{p_{ext}}$. Now since $V^{-1}(y|z)=0 \Leftrightarrow V(z|y)=0$ for every $(y,z)\in\Imag(W)\times\Imag(W')$, we deduce that if $z\in\mathcal{Z}_{p_{ext}}$ then $V(z|y)=0$ for every $y\in \Imag(W)\setminus\mathcal{Y}_{p_{ext}}$. Therefore,
\begin{equation}
\label{eqProbConvexExtreme}
\begin{aligned}
{\MP}_{W'}(p_{ext})&=\sum_{z\in\mathcal{Z}_{p_{ext}}} P_{W'}^o(z)\stackrel{(a)}{=}\sum_{z\in\mathcal{Z}_{p_{ext}}} \sum_{y\in\Imag(W)} V(z|y) P_W^o(y)\\
&\stackrel{(b)}{=}\sum_{z\in\mathcal{Z}_{p_{ext}}} \sum_{y\in\mathcal{Y}_{p_{ext}}} V(z|y) P_W^o(y)\leq \sum_{z\in\Imag(W')} \sum_{y\in\mathcal{Y}_{p_{ext}}} V(z|y) P_W^o(y)\\
&=\sum_{y\in\mathcal{Y}_{p_{ext}}} P_W^o(y)={\MP}_{W}(p_{ext}),
\end{aligned}
\end{equation}
where (a) follows from Equation \eqref{eqPWprimeoz}, and (b) follows from the fact that for every $y\in \Imag(W)\setminus\mathcal{Y}_{p_{ext}}$, we have $V(z|y)=0$.

Now assume that $W$ and $W'$ are equivalent. Equation \eqref{eqConvexHullDegrad} (applied twice) implies that we must have $\conv(\supp({\MP}_{W'}))= \conv(\supp({\MP}_W))$ which implies that $\supp({\MP}_{W'})$ and $\supp({\MP}_{W})$ have the same convex-extreme points. Now fix a convex-extreme point $p_{ext}\in\CE(\supp({\MP}_{W'}))=\CE(\supp({\MP}_{W}))$. Equation \eqref{eqProbConvexExtreme} (applied twice) implies that ${\MP}_W(p_{ext})={\MP}_{W'}(p_{ext})$. By using Equation \eqref{eqProbConvexExtreme} again we obtain:
$$\sum_{z\in\mathcal{Z}_{p_{ext}}} \sum_{y\in\mathcal{Y}_{p_{ext}}} V(z|y) P_W^o(y)= \sum_{z\in\Imag(W')} \sum_{y\in\mathcal{Y}_{p_{ext}}} V(z|y) P_W^o(y),$$
hence
$$\sum_{z\in \Imag(W')\setminus\mathcal{Z}_{p_{ext}}} \sum_{y\in\mathcal{Y}_{p_{ext}}} V(z|y) P_W^o(y)= 0.$$
But $P_W^o(y)>0$ for every $y\in\mathcal{Y}_{p_{ext}}$. Therefore, for every $z\in \Imag(W')\setminus\mathcal{Z}_{p_{ext}}$ and every $y\in\mathcal{Y}_{p_{ext}}$, we must have $V(z|y)=0$ (which implies that $V^{-1}(y|z)=0$). We conclude that for every $z\in \Imag(W')\setminus\mathcal{Z}_{p_{ext}}$, we can rewrite Equations \eqref{eqPWprimeoz} and \eqref{eqConvexHullWzm1} as:
$$P_{W'}^o(z)=\sum_{y\in\Imag(W)\setminus\mathcal{Y}_{p_{ext}}} V(z|y) P_W^o(y),$$
and
$$W'^{-1}_z=\sum_{y\in\Imag(W)\setminus\mathcal{Y}_{p_{ext}}}V^{-1}(y|z)W_y^{-1}.$$
We can now repeat the above argument but on $\supp({\MP}_W)\setminus\{p_{ext}\}$ and $\supp({\MP}_{W'})\setminus\{p_{ext}\}$ instead of $\supp({\MP}_W)$ and $\supp({\MP}_{W'})$. We deduce that $\conv(\supp({\MP}_W)\setminus\{p_{ext}\})=\conv(\supp({\MP}_{W'})\setminus\{p_{ext}\})$ so $\supp({\MP}_W)\setminus\{p_{ext}\}$ and $\supp({\MP}_{W'})\setminus\{p_{ext}\}$ have the same convex-extreme points. We can also prove that ${\MP}_W(p_{ext}')={\MP}_{W'}(p_{ext}')$ for every $p_{ext}'\in\CE(\supp({\MP}_{W'})\setminus\{p_{ext}\})=\CE(\supp({\MP}_{W})\setminus\{p_{ext}\})$.

Notice that any point of $\supp({\MP}_W)$ (respectively $\supp({\MP}_{W'})$) becomes convex-extreme after removing a finite number of elements from $\supp({\MP}_W)$ (respectively $\supp({\MP}_{W'})$). Therefore, after inductively applying the above argument a finite number of times, we can deduce that $\supp({\MP}_W)=\supp({\MP}_{W'})$ and ${\MP}_W(p)={\MP}_{W'}(p)$ for every $p\in \supp({\MP}_W)=\supp({\MP}_{W'})$, hence ${\MP}_W={\MP}_{W'}$.

\vspace*{3mm}

Now let $W\in\DMC_{\mathcal{X},\mathcal{Y}}$ and $W'\in\DMC_{\mathcal{X},\mathcal{Z}}$ be any two channels satisfying ${\MP}_W={\MP}_{W'}$. We have $\supp({\MP}_W)=\supp({\MP}_{W'})$, and for every $p\in\supp({\MP}_W)=\supp({\MP}_{W'})$, we have
$$\sum_{y\in\mathcal{Y}_p}P_W^o(y)={\MP}_W(p)={\MP}_{W'}(p)=\sum_{z\in\mathcal{Z}_p}P_{W'}^o(z).$$

Define the channel $V\in\DMC_{\mathcal{Y},\mathcal{Z}}$ as
$$V(z|y)=\begin{cases}\displaystyle\frac{1}{|\mathcal{Z}|}\quad&\text{if}\;y\notin\Imag(W),\\\displaystyle\frac{P^o_{W'}(z)}{{\MP}_{W'}(W^{-1}_y)}\quad&\text{if}\;y\in\Imag(W)\;\text{and}\;z\in\mathcal{Z}_{W^{-1}_y},\\0\quad&\text{otherwise}.\end{cases}$$
A simple calculation shows that $\displaystyle\sum_{z\in\mathcal{Z}}V(z|y)=1$ for every $y\in\mathcal{Y}$, so $V$ is a valid channel.

Notice that for every $(y,z)\in\Imag(W)\times\Imag(W')$, we have:
$$z\in\mathcal{Z}_{W^{-1}_y}\;\;\Leftrightarrow \;\; W'^{-1}_z=W^{-1}_y\;\;\Leftrightarrow \;\; y\in\mathcal{Y}_{W'^{-1}_z}.$$
Moreover, if $z\in\Imag(W')\;\text{and}\; y\in\mathcal{Y}_{W'^{-1}_z}$, we have ${\MP}_{W'}(W^{-1}_y)={\MP}_{W}(W'^{-1}_z)$. Therefore, we can rewrite $V$ as:
$$V(z|y)=\begin{cases}\displaystyle\frac{P^o_{W'}(z)}{{\MP}_{W}(W'^{-1}_z)}\quad&\text{if}\;z\in\Imag(W')\;\text{and}\; y\in\mathcal{Y}_{W'^{-1}_z},\\\displaystyle\frac{1}{|\mathcal{Z}|}\quad&\text{if}\;y\notin\Imag(W),\\0\quad&\text{otherwise}.\end{cases}$$

Let $W''=V\circ W\in\DMC_{\mathcal{X},\mathcal{Z}}$. For every $z\in\mathcal{Z}\setminus\Imag(W')$, Equation \eqref{eqPWprimeoz} implies that:
\begin{align*}
P_{W''}^o(z)&=\sum_{y\in\Imag(W)} V(z|y) P_W^o(y)\stackrel{(a)}{=}0=P^o_{W'}(z),
\end{align*}
where (a) follows from the fact that $V(z|y)=0$ if $y\in\Imag(W)$ and $z\notin \Imag(W')$.

On the other hand, for every $z\in\Imag(W')$, Equation \eqref{eqPWprimeoz} implies that:
\begin{align*}
P_{W''}^o(z)&=\sum_{y\in\Imag(W)} V(z|y) P_W^o(y)=
\sum_{y\in\mathcal{Y}_{W'^{-1}_z}} \frac{P^o_{W'}(z)}{{\MP}_{W}(W'^{-1}_z)}P_W^o(y)\\
&=\frac{P^o_{W'}(z)}{{\MP}_{W}(W'^{-1}_z)}\sum_{y\in\mathcal{Y}_{W'^{-1}_z}}P_W^o(y)=\frac{P^o_{W'}(z)}{{\MP}_{W}(W'^{-1}_z)} {\MP}_{W}(W'^{-1}_z)=P^o_{W'}(z).
\end{align*}
Therefore, $P_{W''}^o(z)=P_{W'}^o(z)$ for every $z\in\mathcal{Z}$, which implies that $\Imag(W'')=\Imag(W')$.

Now define $V^{-1}\in\DMC_{\Imag(W''),\Imag(W)}$ as $$V^{-1}(y|z)=\frac{V(z|y) P_W^o(y)}{\displaystyle\sum_{y'\in\Imag(W)}V(z|y') P_W^o(y')}.$$

Equation \eqref{eqConvexHullWzm1} implies that for every $z\in\Imag(W'')=\Imag(W')$, we have: 
\begin{align*}
W''^{-1}_z&=\sum_{y\in\Imag(W)}V^{-1}(y|z)W_y^{-1}\stackrel{(a)}{=}\sum_{y\in \mathcal{Y}_{W'^{-1}_z}}V^{-1}(y|z)W_y^{-1}\\
&=\sum_{y\in \mathcal{Y}_{W'^{-1}_z}}V^{-1}(y|z)W'^{-1}_z\stackrel{(b)}{=}\sum_{y\in \Imag(W)}V^{-1}(y|z)W'^{-1}_z=W'^{-1}_z,
\end{align*}
where (a) and (b) follow from the fact that for every $(y,z)\in\Imag(W)\times\Imag(W'')$, we have $V^{-1}(y|z)=0$ if and only if $V(z|y)=0$.

We conclude that $P_{W''}^o=P_{W'}^o$, and for every $z\in\Imag(W'')=\Imag(W')$, we have $W''^{-1}_z=W'^{-1}_z$. Therefore, $W'=W''=V\circ W$ and so $W'$ is degraded from $W$. By exchanging the roles of $W$ and $W'$ we get that $W$ is also degraded from $W'$, hence $W$ and $W'$ are equivalent.

\section{Proof of Lemma \ref{lemProjContClosed}}
\label{appProjContClosed}

We need the following lemma:

\begin{mylem}
The relation $R_{\mathcal{X},\mathcal{Y}}^{(o)}$ is closed in $\DMC_{\mathcal{X},\mathcal{Y}}\times \DMC_{\mathcal{X},\mathcal{Y}}$.
\label{lemRClosed}
\end{mylem}
\begin{proof}
Define the mapping $f:(\DMC_{\mathcal{X},\mathcal{Y}})^2\times(\DMC_{\mathcal{Y},\mathcal{Y}})^2\rightarrow (\DMC_{\mathcal{X},\mathcal{Y}})^4$ as:
$$f(W,W',V,V')=(W,V'\circ W',W',V\circ W).$$
$f$ is continuous because channel composition is continuous.

Define the set $A\subset (\DMC_{\mathcal{X},\mathcal{Y}})^4$ as:
$$\textstyle A:=\{(W,W,W',W'):\; W,W'\in\DMC_{\mathcal{X},\mathcal{Y}}\}.$$
It is easy to see that $A$ is a closed subset of $(\DMC_{\mathcal{X},\mathcal{Y}})^4$. We have:
$$\textstyle f^{-1}(A)=\{(W,W',V,V')\in (\DMC_{\mathcal{X},\mathcal{Y}})^2\times (\DMC_{\mathcal{Y},\mathcal{Y}})^2:\; V'\circ W'=W,\;V\circ W=W'\}.$$

Since $f$ is continuous and since $A$ is a closed subset of $(\DMC_{\mathcal{X},\mathcal{Y}})^4$, $f^{-1}(A)$ is a closed subset of $(\DMC_{\mathcal{X},\mathcal{Y}})^2\times (\DMC_{\mathcal{Y},\mathcal{Y}})^2$ which is compact. Therefore, $f^{-1}(A)$ is compact.

Now define the mapping $g: (\DMC_{\mathcal{X},\mathcal{Y}})^2\times (\DMC_{\mathcal{Y},\mathcal{Y}})^2\rightarrow (\DMC_{\mathcal{X},\mathcal{Y}})^2$ as follows:
$$g(W,W',V,V')=(W,W').$$
Since $g$ is continuous and since $f^{-1}(A)$ is compact, $g(f^{-1}(A))$ is a compact subset of $\DMC_{\mathcal{X},\mathcal{Y}}^2$. Now notice that
\begin{align*}
g(f^{-1}(A))&=\textstyle \big\{(W,W')\in (\DMC_{\mathcal{X},\mathcal{Y}})^2:\; \exists V,V'\in\DMC_{\mathcal{Y},\mathcal{Y}},\; (W,W',V,V')\in f^{-1}(A)\big\}\\
&=\textstyle \big\{(W,W')\in (\DMC_{\mathcal{X},\mathcal{Y}})^2:\; \exists V,V'\in\DMC_{\mathcal{Y},\mathcal{Y}},\; V'\circ W'=W,\; V\circ W=W'\big\}\\
&=\textstyle \big\{(W,W')\in (\DMC_{\mathcal{X},\mathcal{Y}})^2:\; W\;\text{is equivalent to}\; W'\big\}\\
&=\textstyle \big\{(W,W')\in (\DMC_{\mathcal{X},\mathcal{Y}})^2:\; W R_{\mathcal{X},\mathcal{Y}}^{(o)} W'\big\}=R_{\mathcal{X},\mathcal{Y}}^{(o)}.
\end{align*}
We conclude that $R_{\mathcal{X},\mathcal{Y}}^{(o)}$ is compact, hence it is also closed because $ (\DMC_{\mathcal{X},\mathcal{Y}})^2$ is a metric space.
\end{proof}

\vspace*{3mm}

Now we are ready to prove Lemma \ref{lemProjContClosed}:

Let $\Proj:\DMC_{\mathcal{X},\mathcal{Y}}\rightarrow\DMC_{\mathcal{X},\mathcal{Y}}^{(o)}$ be defined as $\Proj(W)=\hat{W}$. The continuity of $\Proj$ follows from the definition of the quotient topology.

Now let $A$ be a closed subset of $\DMC_{\mathcal{X},\mathcal{Y}}$. We want to show that $\Proj(A)$ is closed.

Since $A$ is closed in $\DMC_{\mathcal{X},\mathcal{Y}}$, the set $\DMC_{\mathcal{X},\mathcal{Y}}\times A$ is closed in $(\DMC_{\mathcal{X},\mathcal{Y}})^2$. On the other hand, $R_{\mathcal{X},\mathcal{Y}}^{(o)}$ is closed in $(\DMC_{\mathcal{X},\mathcal{Y}})^2$ by Lemma \ref{lemRClosed}. Therefore, $(\DMC_{\mathcal{X},\mathcal{Y}}\times A)\cap R_{\mathcal{X},\mathcal{Y}}^{(o)}$ is closed in $(\DMC_{\mathcal{X},\mathcal{Y}})^2$ which is compact, hence $(\DMC_{\mathcal{X},\mathcal{Y}}\times A)\cap R_{\mathcal{X},\mathcal{Y}}^{(o)}$ is compact. We have:
\begin{align*}
\textstyle(\DMC_{\mathcal{X},\mathcal{Y}}\times A)\cap R_{\mathcal{X},\mathcal{Y}}^{(o)}&=\textstyle\big\{(W,W')\in(\DMC_{\mathcal{X},\mathcal{Y}})^2:\; W R_{\mathcal{X},\mathcal{Y}}^{(o)} W'\;\text{and}\;W'\in A\big\}.
\end{align*}

Now define the mapping $g:(\DMC_{\mathcal{X},\mathcal{Y}})^2\rightarrow\DMC_{\mathcal{X},\mathcal{Y}}$ as
$$g(W,W')=W.$$
Let $A_R:=g\big((\DMC_{\mathcal{X},\mathcal{Y}}\times A)\cap R_{\mathcal{X},\mathcal{Y}}^{(o)}\big)$. Since $g$ is continuous and since $(\DMC_{\mathcal{X},\mathcal{Y}}\times A)\cap R_{\mathcal{X},\mathcal{Y}}^{(o)}$ is compact, $A_R$ is also compact. We have:
\begin{align*}
\textstyle A_R&=\big\{W\in \textstyle\DMC_{\mathcal{X},\mathcal{Y}}:\;\exists W'\in A,\; W R_{\mathcal{X},\mathcal{Y}}^{(o)} W'\big\}=\Proj^{-1}(\Proj(A)).
\end{align*}

Since $\DMC_{\mathcal{X},\mathcal{Y}}$ is a metric space and since $A_R$ is compact, $\Proj^{-1}(\Proj(A))=A_R$ is closed in $\DMC_{\mathcal{X},\mathcal{Y}}$. On the other hand, we have $\Proj^{-1}\big(\DMC_{\mathcal{X},\mathcal{Y}}^{(o)}\setminus \Proj(A)\big)=\DMC_{\mathcal{X},\mathcal{Y}}\setminus \Proj^{-1}(\Proj(A))$, hence $\Proj^{-1}\big(\DMC_{\mathcal{X},\mathcal{Y}}^{(o)}\setminus \Proj(A)\big)$ is open in $\DMC_{\mathcal{X},\mathcal{Y}}$, which implies that $\DMC_{\mathcal{X},\mathcal{Y}}^{(o)}\setminus \Proj(A)$ is open in $\DMC_{\mathcal{X},\mathcal{Y}}^{(o)}$. Therefore, $\Proj(A)$ is closed in $\DMC_{\mathcal{X},\mathcal{Y}}^{(o)}$.

\section{Proof of Proposition \ref{propInteriorEmptyDMCXno}}

\label{appInteriorEmptyDMCXno}

let $\hat{U}$ be an arbitrary non-empty open subset of $(\DMC_{\mathcal{X},[m]}^{(o)},\mathcal{T}_{\mathcal{X},[m]}^{(o)})$ and let $\Proj$ be the projection onto the $R_{\mathcal{X},[m]}^{(o)}$-equivalence classes. $\Proj^{-1}(\hat{U})$ is open in the metric space $(\DMC_{\mathcal{X},[m]},d_{\mathcal{X},[m]})$. Therefore, there exists $W\in \DMC_{\mathcal{X},[m]}$ and $\epsilon>0$ such that $\Proj^{-1}(\hat{U})$ contains the open ball of center $W$ and radius $\epsilon$.

We will show that there exists $W'\in \DMC_{\mathcal{X},[m]}$ such that $\rank(W')=m>n$ and $d_{\mathcal{X},[m]}(W,W')<\epsilon$. If $\rank(W)=m$, take $W'=W$.

Assume that $\rank(W)<m$. Let $P_W^o\in\Delta_{[m]}$, $\Imag(W)$ and $\{W_y^{-1}:\;y\in\Imag(W)\}$ be as in Section \ref{secMetaProb}.

Let $\{v_y\}_{y\in[m]}$ be a collection of $m$ vectors in $\mathbb{R}^{\mathcal{X}}$ such that:
\begin{itemize}
\item $\displaystyle\sum_{y\in\Imag(W)}P_{W}^o(y)\cdot v_y=0$.
\item $\displaystyle\sum_{y\in[m]\setminus\Imag(W)}v_y=0$.
\item For every $y\in[m]$, $\displaystyle\sum_{x\in\mathcal{X}}v_y(x)=0$.
\item The vectors $\{v_y\}_{y\in[m]}$ are pairwise different.
\end{itemize}
Such collection can always be found.

Let $0<\delta,\delta'<1$ and define $P_{W'}^o\in\mathbb{R}^{[m]}$ as follows:
$$P_{W'}^o(y)=\begin{cases}\displaystyle (1-\delta)P_W^o(y)\quad&\text{if}\;y\in\Imag(W),\\\displaystyle\frac{\delta}{m-|\Imag(W)|}\quad&\text{otherwise}. \end{cases}$$
Clearly, $P_{W'}^o\in\Delta_{[m]}$ and $P_{W'}^o(y)>0$ for every $y\in[m]$. Now for every $y\in [m]$, define $W'^{-1}_y$ as follows:
\begin{align*}
W'^{-1}_y=\begin{cases}\displaystyle (1-\delta)W^{-1}_y + \delta\pi_{\mathcal{X}}+\delta' v_y \quad&\text{if}\;y\in\Imag(W),\\\pi_{\mathcal{X}}+\delta' v_y\quad&\text{otherwise}, \end{cases}
\end{align*}
where $\pi_{\mathcal{X}}\in\Delta_{\mathcal{X}}$ is the uniform probability distribution on $\mathcal{X}$. A simple calculation shows that $\displaystyle\sum_{y\in [m]} P_{W'}^o(y) W'^{-1}_y=\pi_{\mathcal{X}}$, and for every $y\in[m]$ we have $\displaystyle\sum_{x\in\mathcal{X}} W'^{-1}_y(x)=1$.

Notice that for $y\in\Imag(W)$, since $0<\delta<1$, $(1-\delta)W^{-1}_y + \delta\pi_{\mathcal{X}}$ lies inside the interior of the probability distribution simplex $\Delta_{\mathcal{X}}$. This means that for $\delta'$ small enough, $(1-\delta)W^{-1}_y + \delta\pi_{\mathcal{X}}+\delta'v_y\in \Delta_{\mathcal{X}}$ for every $y\in\Imag(W)$, and $\pi_{\mathcal{X}}+\delta' v_y\in\Delta_{\mathcal{X}}$ for every $y\notin\Imag(W)$. For every $0<\delta<1$, choose $\delta':=\delta'(\delta)$ so that $0<\delta'<\delta$ and $W'^{-1}_y\in\Delta_{\mathcal{X}}$ for every $y\in[m]$.

It is easy to see that for $\delta$ small enough, $W'^{-1}_{y_1}\neq W'^{-1}_{y_2}$ for every $y_1,y_2\in[m]$ satisfying $y_1\neq y_2$. Define the channel $W'\in\DMC_{\mathcal{X},[m]}$ as follows:
$$W'(y|x)=|\mathcal{X}|P_{W'}^o(y)W'^{-1}_y(x).$$
Since $P_{W'}^o(y)>0$ for every $y\in[m]$, we have $\supp({\MP}_{W'})=\{W'^{-1}_y:\;y\in[m]\}$. Therefore, there exists $\delta_0>0$ such for every $0<\delta<\delta_0$, we have $\rank(W')=m$. On the other hand, we have $\displaystyle \lim_{\delta\to 0} P_{W'}^o=P_{W}^o$ and $\displaystyle \lim_{\delta\to 0} W'^{-1}_y=W^{-1}_y$ for every $y\in\Imag(W)$. Therefore, $\displaystyle \lim_{\delta\to 0} W'=W$ (where the limit is taken in $(\DMC_{\mathcal{X},[m]},d_{\mathcal{X},[m]})$). This shows that there exists $W'\in\DMC_{\mathcal{X},[m]}$ such that $\rank(W')=m>n$ and $d_{\mathcal{X},[m]}(W,W')<\epsilon$, which means that $W'\in\Proj^{-1}(\hat{U})$ and $W'$ is not equivalent to any channel in $\DMC_{\mathcal{X},[n]}$ (see Corollary \ref{corNonEquiv}). Therefore, $\Proj(W')\in \hat{U}$ and $\Proj(W')\notin\DMC_{\mathcal{X},[n]}^{(o)}$ because $W'$ is not equivalent to any channel in $\DMC_{\mathcal{X},[n]}$. This shows that every non-empty open subset of $\DMC_{\mathcal{X},[m]}^{(o)}$ is not contained in $\DMC_{\mathcal{X},[n]}^{(o)}$. We conclude that the interior of $\DMC_{\mathcal{X},[n]}^{(o)}$ in $\DMC_{\mathcal{X},[m]}^{(o)}$ is empty.

\section{Proof of Lemma \ref{lemDMCXoNorm}}

\label{appDMCXoNorm}

Define $\DMC_{\mathcal{X},[0]}^{(o)}=\o$, which is strongly closed in $\DMC_{\mathcal{X},\ast}^{(o)}$.

Let $A$ and $B$ be two disjoint strongly closed subsets of $\DMC_{\mathcal{X},\ast}^{(o)}$. For every $n\geq 0$, let $A_n=A\cap \DMC_{\mathcal{X},[n]}^{(o)}$ and $B_n=B\cap \DMC_{\mathcal{X},[n]}^{(o)}$. Since $A$ and $B$ are strongly closed in $\DMC_{\mathcal{X},\ast}^{(o)}$, $A_n$ and $B_n$ are closed in $\DMC_{\mathcal{X},[n]}^{(o)}$. Moreover, $A_n\cap B_n\subset A\cap B=\o$.

Construct the sequences $(U_n)_{n\geq 0},(U_n')_{n\geq 0},(K_n)_{n\geq 0}$ and $(K_n')_{n\geq 0}$ recursively as follows:

$U_0=U_0'=K_0=K_0'=\o\subset\DMC_{\mathcal{X},[0]}^{(o)}$. Since $A_0=B_0=\o$, we have $A_0\subset U_0\subset K_0$ and $B_0\subset U_0'\subset K_0'$. Moreover, $U_0$ and $U_0'$ are open in $\DMC_{\mathcal{X},[0]}^{(o)}$, $K_0$ and $K_0'$ are closed in $\DMC_{\mathcal{X},[0]}^{(o)}$, and $K_0\cap K_0'=\o$.

Now let $n\geq 1$ and assume that we constructed $(U_i)_{0\leq i< n},(U_i')_{0\leq i< n},(K_i)_{0\leq i< n}$ and $(K_i')_{0\leq i< n}$ such that for every $0\leq i< n$, we have $A_i\subset U_i\subset K_i\subset\DMC_{\mathcal{X},[i]}^{(o)}$, $B_i\subset U_i'\subset K_i'\subset \DMC_{\mathcal{X},[i]}^{(o)}$, $U_i$ and $U_i'$ are open in $\DMC_{\mathcal{X},[i]}^{(o)}$, $K_i$ and $K_i'$ are closed in $\DMC_{\mathcal{X},[i]}^{(o)}$, and $K_i\cap K_i'=\o$. Moreover, assume that $K_i\subset U_{i+1}$ and $K_i'\subset U_{i+1}'$ for every $0\leq i<n-1$.

Let $C_n=A_n\cup K_{n-1}$ and $D_n=B_n\cup K_{n-1}'$. Since $K_{n-1}$ and $K_{n-1}'$ are closed in $\DMC_{\mathcal{X},[n-1]}^{(o)}$ and since $\DMC_{\mathcal{X},[n-1]}^{(o)}$ is closed in $\DMC_{\mathcal{X},[n]}^{(o)}$, we can see that $K_{n-1}$ and $K_{n-1}'$ are closed in $\DMC_{\mathcal{X},[n]}^{(o)}$. Therefore, $C_n$ and $D_n$ are closed in $\DMC_{\mathcal{X},[n]}^{(o)}$. Moreover, we have
\begin{align*}
C_n\cap D_n&=(A_n\cup K_{n-1})\cap(B_n\cup K_{n-1}')\\
&=(A_n\cap B_n)\cup(A_n\cap K_{n-1}')\cup (K_{n-1}\cap B_n)\cup(K_{n-1}\cap K_{n-1}')\\
&\stackrel{(a)}{=}\textstyle \left(A_n\cap K_{n-1}'\cap \DMC_{\mathcal{X},[n-1]}^{(o)}\right)\cup \left(K_{n-1}\cap \DMC_{\mathcal{X},[n-1]}^{(o)}\cap B_n\right)\\
&=(A_{n-1}\cap K_{n-1}')\cup (K_{n-1}\cap B_{n-1})\subset (K_{n-1}\cap K_{n-1}')\cup (K_{n-1}\cap K_{n-1}')=\o,
\end{align*}
where (a) follows from the fact that $A_n\cap B_n=K_{n-1}\cap K_{n-1}'=\o$ and the fact that $K_{n-1}\subset \DMC_{\mathcal{X},[n-1]}^{(o)}$ and $K_{n-1}'\subset \DMC_{\mathcal{X},[n-1]}^{(o)}$.

Since $\DMC_{\mathcal{X},[n]}^{(o)}$ is normal (because it is metrizable), and since $C_n$ and $D_n$ are closed disjoint subsets of $\DMC_{\mathcal{X},[n]}^{(o)}$, there exist two sets $U_n,U_n'\subset \DMC_{\mathcal{X},[n]}^{(o)}$ that are open in $\DMC_{\mathcal{X},[n]}^{(o)}$ and two sets $K_n,K_n'\subset \DMC_{\mathcal{X},[n]}^{(o)}$ that are closed in $\DMC_{\mathcal{X},[n]}^{(o)}$ such that $C_n\subset U_n\subset K_n$, $D_n\subset U_n'\subset K_n'$ and $K_n\cap K_n'=\o$. Clearly, $A_n\subset U_n\subset K_n\subset \DMC_{\mathcal{X},[n]}^{(o)}$, $B_n\subset U_n'\subset K_n'\subset \DMC_{\mathcal{X},[n]}^{(o)}$, $K_{n-1}\subset U_n$ and $K_{n-1}'\subset U_n'$. This concludes the recursive construction.

Now define $\displaystyle U=\bigcup_{n\geq 0}U_n=\bigcup_{n\geq 1}U_n$ and $\displaystyle U'=\bigcup_{n\geq 0}U_n'=\bigcup_{n\geq 1}U_n'$. Since $A_n\subset U_n$ for every $n\geq 1$, we have 
\begin{align*}
\textstyle A=A\cap\DMC_{\mathcal{X},\ast}^{(o)}=A\cap\left({\displaystyle\bigcup_{n\geq 1}}\DMC_{\mathcal{X},[n]}^{(o)}\right)={\displaystyle\bigcup_{n\geq 1}}\left(A\cap \DMC_{\mathcal{X},[n]}^{(o)}\right)={\displaystyle\bigcup_{n\geq 1}} A_n\subset {\displaystyle\bigcup_{n\geq 1}} U_n =U.
\end{align*}
Moreover, for every $n\geq 1$ we have
\begin{align*}
\textstyle U\cap \DMC_{\mathcal{X},[n]}^{(o)}=\left({\displaystyle\bigcup_{i\geq 1} U_i}\right)\cap \DMC_{\mathcal{X},[n]}^{(o)}\stackrel{(a)}{=}\left({\displaystyle\bigcup_{i\geq n} U_i}\right)\cap \DMC_{\mathcal{X},[n]}^{(o)}={\displaystyle\bigcup_{i\geq n} \left(U_i\cap \textstyle\DMC_{\mathcal{X},[n]}^{(o)}\right)},
\end{align*}
where (a) follows from the fact that $U_i\subset K_i\subset U_{i+1}$ for every $i\geq 0$, which means that the sequence $(U_i)_{i\geq 1}$ is increasing.

For every $i\geq n$, we have $\DMC_{\mathcal{X},[n]}^{(o)}\subset \DMC_{\mathcal{X},[i]}^{(o)}$ and $U_i$ is open in $\DMC_{\mathcal{X},[i]}^{(o)}$, hence $U_i\cap \DMC_{\mathcal{X},[n]}^{(o)}$ is open in $\DMC_{\mathcal{X},[n]}^{(o)}$. Therefore, $U\cap \DMC_{\mathcal{X},[n]}^{(o)}=\displaystyle\bigcup_{i\geq n} \left(U_i\cap \textstyle\DMC_{\mathcal{X},[n]}^{(o)}\right)$ is open in $\DMC_{\mathcal{X},[n]}^{(o)}$. Since this is true for every $n\geq 1$, we conclude that $U$ is strongly open in $\DMC_{\mathcal{X},\ast}^{(o)}$.

We can similarly show that $B\subset U'$ and that $U'$ is strongly open in $\DMC_{\mathcal{X},\ast}^{(o)}$. Finally, we have
\begin{align*}
U\cap U'=\left(\bigcup_{n\geq 1} U_n\right)\cap \left(\bigcup_{n'\geq 1} U_{n'}'\right)=\bigcup_{n\geq 1, n'\geq 1}(U_n\cap U_{n'}')\stackrel{(a)}{=}\bigcup_{n\geq 1}(U_n\cap U_n')
&\subset\bigcup_{n\geq 1}(K_n\cap K_n')=\o,
\end{align*}
where (a) follows from the fact that for every $n\geq 1$ and every $n'\geq 1$, we have $$U_n\cap U_{n'}'\subset U_{\max\{n,n'\}}\cap U_{\max\{n,n'\}}'$$ because $(U_n)_{n\geq 1}$ and $(U_n')_{n\geq 1}$ are increasing. We conclude that $(\DMC_{\mathcal{X},\ast}^{(o)},\mathcal{T}_{s,\mathcal{X},\ast}^{(o)})$ is normal.

\section{Proof of Lemma \ref{lemRelDistance}}
\label{appRelDistance}

Let $W_1,W_2\in\DMC_{\mathcal{X},\mathcal{Y}}$, and let $\hat{W}_1$ and $\hat{W}_2$ be the $R_{\mathcal{X},\mathcal{Y}}^{(o)}$-equivalence classes of $W_1$ and $W_2$ respectively.

Fix $m\geq 1$, $p\in\Delta_{[m]\times\mathcal{X}}$ and $D\in \DMC_{\mathcal{Y},[m]}$. We have:

\begin{align*}
\sum_{\substack{u\in[m],\\x\in\mathcal{X},\\y\in\mathcal{Y}}}&p(u,x) W_1(y|x)D(u|y)\\
&=\Bigg(\sum_{\substack{u\in[m],\\x\in\mathcal{X},\\y\in\mathcal{Y}}}p(u,x)W_2(y|x)D(u|y)\Bigg)+\sum_{\substack{u\in[m],\\x\in\mathcal{X},\\y\in\mathcal{Y}}}p(u,x)\cdot\big(W_1(y|x)-W_2(y|x)\big)\cdot D(u|y)\\
&\leq \Bigg(\sup_{D'\in\DMC_{\mathcal{Y},[m]}} \sum_{\substack{u\in[m],\\x\in\mathcal{X},\\y\in\mathcal{Y}}}p(u,x)W_2(y|x)D'(u|y)\Bigg) + \sum_{\substack{u\in[m],\\x\in\mathcal{X},\\y\in\mathcal{Y}}}p(u,x)\cdot\big(W_1(y|x)-W_2(y|x)\big)\cdot D(u|y)\\
&\leq P_c(p,W_2) + \sum_{\substack{u\in[m],\\x\in\mathcal{X}}}p(u,x)\cdot \sum_{\substack{y\in\mathcal{Y}:\\W_1(y|x)> W_2(y|x)}} \big(W_1(y|x)-W_2(y|x)\big)\cdot \left(\sum_{u'\in[m]} D(u'|y)\right)\\
&= P_c(p,W_2) + \sum_{\substack{u\in[m],\\x\in\mathcal{X}}}p(u,x)\cdot\Bigg( \sum_{\substack{y\in\mathcal{Y}:\\W_1(y|x)> W_2(y|x)}}\big(W_1(y|x)-W_2(y|x)\big)\Bigg)\\
&\stackrel{(a)}{\leq} P_c(p,W_2) + \sum_{\substack{u\in[m],\\x\in\mathcal{X}}}p(u,x)\cdot d_{\mathcal{X},\mathcal{Y}}(W_1,W_2)=P_c(p,W_2) + d_{\mathcal{X},\mathcal{Y}}(W_1,W_2),
\end{align*}
where (a) follows from the fact that
\begin{align*}
\sum_{\substack{y\in\mathcal{Y}:\\W_1(y|x)> W_2(y|x)}}\big(W_1(y|x)-W_2(y|x)\big)&=\frac{1}{2}\sum_{y\in\mathcal{Y}}\big|W_1(y|x)-W_2(y|x)\big|\\
&\leq \frac{1}{2}\sup_{x\in\mathcal{X}}\sum_{y\in\mathcal{Y}}\big|W_1(y|x)-W_2(y|x)\big|=d_{\mathcal{X},\mathcal{Y}}(W_1,W_2).
\end{align*}

Therefore,
$$P_c(p,W_1)= \sup_{D\in\DMC_{\mathcal{Y},[m]}}\sum_{\substack{u\in[m],\\x\in\mathcal{X},\\y\in\mathcal{Y}}}p(u,x) W_1(y|x)D(u|y)\leq P_c(p,W_2) + d_{\mathcal{X},\mathcal{Y}}(W_1,W_2).$$
Similarly, we can show that $P_c(p,W_2)\leq P_c(p,W_1)+ d_{\mathcal{X},\mathcal{Y}}(W_1,W_2)$, hence $$|P_c(p,W_1)-P_c(p,W_2)|\leq d_{\mathcal{X},\mathcal{Y}}(W_1,W_2).$$

We conclude that
\begin{align*}
d_{\mathcal{X},\mathcal{Y}}^{(o)}(\hat{W}_1,\hat{W}_2)&=\sup_{\substack{m\geq 1,\\p\in\Delta_{[m]\times\mathcal{X}}}}|P_c(p,\hat{W}_1)-P_c(p,\hat{W}_2)|\\
&=\sup_{\substack{m\geq 1,\\p\in\Delta_{[m]\times\mathcal{X}}}}|P_c(p,W_1)-P_c(p,W_2)|\\
&\leq d_{\mathcal{X},\mathcal{Y}}(W_1,W_2).
\end{align*}

\section{Proof of Lemma \ref{lemDMCXoWasserstein}}

\label{appDMCXoWasserstein}
Let $\gamma\in\Gamma({\MP}_{\hat{W}},{\MP}_{\hat{W}'})$ be a measure on $\Delta_{\mathcal{X}}\times \Delta_{\mathcal{X}}$ that couples ${\MP}_{\hat{W}}$ and ${\MP}_{\hat{W}'}$.

Let $S=\supp({\MP}_{\hat{W}})$ and $S'=\supp({\MP}_{\hat{W}'})$ be the supports of $\hat{W}$ and $\hat{W}'$ respectively. Since ${\MP}_{\hat{W}}$ and ${\MP}_{\hat{W}'}$ are finitely supported, $\gamma$ is also finitely supported and its support is a subset of $S\times S'$. Therefore, there exists a collection of coefficients $\alpha_{p,p'}\in[0,1]$ such that
$$\gamma=\sum_{\substack{p\in S,\\p'\in S'}}\alpha_{p,p'}\delta_{(p,p')},$$
where $\delta_{(p,p')}$ is a Dirac measure centered at $(p,p')\in\Delta_{\mathcal{X}}\times \Delta_{\mathcal{X}}$. Since ${\MP}_{\hat{W}}$ and ${\MP}_{\hat{W}'}$ are the marginals of $\gamma$ on the first and the second factors respectively, we have $\displaystyle{\MP}_{\hat{W}}(p)=\sum_{p'\in S'}\alpha_{p,p'}$ for every $p\in S$. Similarly, $\displaystyle{\MP}_{\hat{W}'}(p')=\sum_{p\in S}\alpha_{p,p'}$ for every $p'\in S'$.

Let $\mathcal{Y}=S\times S'$ and define the channels $W,W'\in\DMC_{\mathcal{X},\mathcal{Y}}$ as:
$$W(p,p'|x)=|\mathcal{X}|\alpha_{p,p'} \cdot p(x),$$
and
$$W'(p,p'|x)=|\mathcal{X}|\alpha_{p,p'} \cdot p'(x).$$

For every $x\in\mathcal{X}$, we have
\begin{align*}
\sum_{(p,p')\in\mathcal{Y}}W(p,p'|x)&=|\mathcal{X}|\sum_{(p,p')\in S\times S'} \alpha_{p,p'} \cdot p(x)= |\mathcal{X}|\sum_{p\in S} {\MP}_{\hat{W}}(p) \cdot p(x)\\
&=|\mathcal{X}|\int_{\Delta_{\mathcal{X}}}p(x)\cdot d{\MP}_{\hat{W}}(p)=|\mathcal{X}|\frac{1}{|\mathcal{X}|}=1.
\end{align*}
Similarly, $\displaystyle\sum_{(p,p')\in\mathcal{Y}}W'(p,p'|x)=1$. Therefore, $W$ and $W'$ are valid channels.

For every $(p,p')\in\mathcal{Y}$, we have
$$P_W^o(p,p')=\sum_{x\in\mathcal{X}}\frac{1}{|\mathcal{X}|}W(p,p'|x)=\sum_{x\in \mathcal{X}}\alpha_{p,p'}\cdot p(x)=\alpha_{p,p'}.$$
Therefore, $\Imag(W)=\{(p,p')\in\mathcal{Y}:\;\alpha_{p,p'}>0\}$. For every $(p,p')\in\Imag(W)$ and every $x\in\mathcal{X}$, we have:
$$W^{-1}_{p,p'}(x)=\frac{W(p,p'|x)}{|\mathcal{X}|P_W^o(p,p')}=\frac{|\mathcal{X}|\alpha_{p,p'} \cdot p(x)}{|\mathcal{X}|\alpha_{p,p'}}=p(x),$$
hence $W^{-1}_{p,p'}=p$ for every $(p,p')\in\Imag(W)$, which shows that $\supp({\MP}_W)\subset S$. Similarly, we can show that $$\Imag(W')=\{(p,p')\in\mathcal{Y}:\;\alpha_{p,p'}>0\},$$ $\supp({\MP}_{W'})\subset S'$, and for every $(p,p')\in\mathcal{Y}$, $P_{W'}^o(p,p')=\alpha_{p,p'}$ and $W'^{-1}_{p,p'}=p'$.

For every $p\in S$, we have:
$${\MP}_W(p)=\sum_{\substack{y\in\Imag(W),\\W^{-1}_y=p}}P_W^o(y)=\sum_{\substack{p'\in S',\\\alpha_{p,p'}>0}}\alpha_{p,p'}=\sum_{p'\in S'}\alpha_{p,p'}={\MP}_{\hat{W}}(p)>0.$$
This shows that $\supp({\MP}_W)=S=\supp({\MP}_{\hat{W}})$ and ${\MP}_W(p)={\MP}_{\hat{W}}(p)$ for every $p\in S$. Therefore, ${\MP}_W={\MP}_{\hat{W}}$ and so $W$ is equivalent to every channel in $\hat{W}$. Similarly, we can show that ${\MP}_{W'}={\MP}_{\hat{W}'}$ and $W'$ is equivalent to every channel in $\hat{W}'$.

Let $\tilde{W}$ and $\tilde{W}'$ be the $R_{\mathcal{X},\mathcal{Y}}^{(o)}$-equivalence classes of $W$ and $W'$ respectively. We can write $\hat{W}=\tilde{W}$ and $\hat{W}'=\tilde{W}'$ because of the canonical identification of $\DMC_{\mathcal{X},\mathcal{Y}}^{(o)}$ with $\DMC_{\mathcal{X},[n]}^{(o)}$, where $n=|\mathcal{Y}|$. We have:
\begin{align*}
d_{\mathcal{X},\ast}^{(o)}(\hat{W},\hat{W}')&=d_{\mathcal{X},\mathcal{Y}}^{(o)}(\tilde{W},\tilde{W}')\stackrel{(a)}{\leq} d_{\mathcal{X},\mathcal{Y}}(W,W')=\frac{1}{2}\max_{x\in\mathcal{X}}\sum_{(p,p')\in\mathcal{Y}}|W(p,p'|x)-W'(p,p'|x)|\\
&=\frac{1}{2}\max_{x\in\mathcal{X}}\sum_{\substack{p\in S,\\p'\in S'}}\Big||\mathcal{X}|\alpha_{p,p'}\cdot p(x) - |\mathcal{X}|\alpha_{p,p'}\cdot p'(x) \Big|=\frac{1}{2}|\mathcal{X}|\max_{x\in\mathcal{X}} \sum_{\substack{p\in S,\\p'\in S'}} \alpha_{p,p'}\cdot |p(x) - p'(x)|\\
&\leq \frac{1}{2}|\mathcal{X}|\sum_{x\in\mathcal{X}}\sum_{\substack{p\in S,\\p'\in S'}} \alpha_{p,p'}\cdot |p(x) - p'(x)|=\frac{1}{2}|\mathcal{X}|\sum_{\substack{p\in S,\\p'\in S'}} \alpha_{p,p'} \sum_{x\in\mathcal{X}} |p(x) - p'(x)|\\
&=\frac{1}{2}|\mathcal{X}|\sum_{\substack{p\in S,\\p'\in S'}} \alpha_{p,p'} \|p-p'\|_1=|\mathcal{X}|\sum_{\substack{p\in S,\\p'\in S'}} \alpha_{p,p'}d(p,p')=|\mathcal{X}|\int_{\Delta_{\mathcal{X}}\times\Delta_{\mathcal{X}}} d(p,p')\cdot d\gamma(p,p'),
\end{align*}
where (a) follows from Lemma \ref{lemRelDistance}, and $d(p,p')=\frac{1}{2}\|p-p'\|_1$ is the total variation distance between $p$ and $p'$. Therefore,

\begin{align*}
d_{\mathcal{X},\ast}^{(o)}(\hat{W},\hat{W}')\leq |\mathcal{X}|\inf_{\gamma\in\Gamma({\MP}_{\hat{W}},{\MP}_{\hat{W}'})} \int_{\Delta_{\mathcal{X}}\times\Delta_{\mathcal{X}}} d(p,p')\cdot d\gamma(p,p')=|\mathcal{X}|\cdot W_1({\MP}_{\hat{W}},{\MP}_{\hat{W}'}).
\end{align*}

\section{Proof of Proposition \ref{propClosureFinSupported}}
\label{appClosureFinSupported}

If $|\mathcal{X}|=1$, $\Delta_{\mathcal{X}}$ consists of a single probability distribution and $\mathcal{MP}(\mathcal{X})$ consists of a single meta-probability measure which is balanced and finitely supported, so  $\mathcal{MP}(\mathcal{X})=\mathcal{MP}_{b}(\mathcal{X})=\mathcal{MP}_{bf}(\mathcal{X})$.

Now assume that $|\mathcal{X}|\geq 2$. We start by showing that $\mathcal{MP}_{b}(\mathcal{X})$ is weakly-$\ast$ closed.

For every $x\in\mathcal{X}$. Consider the mapping $f_x:\Delta_{\mathcal{X}}\rightarrow\mathbb{R}$ defined as $f_x(p)=p(x)$. Clearly, $f_x$ is bounded and continuous. Therefore, the mapping $$F_x:\mathcal{MP}(\mathcal{X})\rightarrow \mathbb{R}$$ defined as
$$F_x({\MP})= \int_{\Delta_{\mathcal{X}}} f_x d{\MP}=\int_{\Delta_{\mathcal{X}}} p(x)\cdot d{\MP}(p)$$
is continuous in the weak-$\ast$ topology. Therefore, $\displaystyle F_x^{-1}\left(\left\{\frac{1}{|\mathcal{X}|}\right\}\right)$ is weakly-$\ast$ closed. It is easy to see that $\mathcal{MP}_{b}(\mathcal{X})=\displaystyle\bigcap_{x\in\mathcal{X}} F_x^{-1}\left(\left\{\frac{1}{|\mathcal{X}|}\right\}\right)$. This proves that $\mathcal{MP}_{b}(\mathcal{X})$, which is the finite intersection of weakly-$\ast$ closed sets, is weakly-$\ast$ closed.

It remains to show that $\mathcal{MP}_{bf}(\mathcal{X})$ is weakly-$\ast$ dense in $\mathcal{MP}_{b}(\mathcal{X})$. We will show that for every $\epsilon>0$ and every ${\MP}\in \mathcal{MP}_{b}(\mathcal{X})$, there exists ${\MP}'\in \mathcal{MP}_{bf}(\mathcal{X})$ such that $W_1({\MP},{\MP}')<\epsilon$.

Fix $0<\epsilon<1$ and let ${\MP}\in\mathcal{MP}_b(\mathcal{X})$ be any balanced meta-probability measure on $\mathcal{X}$, i.e., for every $x\in\mathcal{X}$ we have
$$\int_{\Delta_{\mathcal{X}}}p(x)d{\MP}(p)=\frac{1}{|\mathcal{X}|}.$$

Now fix $x\in\mathcal{X}$. By the definition of the Lebesgue integral, there exists a finite partition $\{B_{x,i}\}_{1\leq i\leq k_x}$ of $\Delta_{\mathcal{X}}$ and a sequence of positive numbers $(b_{x,i})_{1\leq i\leq k_x}$ such that for every $1\leq i\leq k_x$, $B_{x,i}$ is a Borel set of $\Delta_{\mathcal{X}}$, $b_{x,i}\leq p(x)$ for every $p\in B_{x,i}$, and
$$\sum_{i=1}^{k_x} b_{x,i}{\MP}(B_{x,i}) \geq \left(\int_{\Delta_{\mathcal{X}}}p(x)\cdot d{\MP}(p)\right)-\frac{\epsilon}{12|\mathcal{X}|}=\frac{1}{|\mathcal{X}|}-\frac{\epsilon}{12|\mathcal{X}|}.$$
By applying the same reasoning on the function $1-p(x)\geq 0$, we can find a finite partition $\{C_{x,i}\}_{1\leq i\leq m_x}$ of $\Delta_{\mathcal{X}}$ and a sequence of positive numbers $(c_{x,i})_{1\leq i\leq m_x}$ such that for every $1\leq i\leq m_x$, $C_{x,i}$ is a Borel set of $\Delta_{\mathcal{X}}$, $c_{x,i}\geq p(x)$ for every $p\in C_{x,i}$ and
$$\sum_{i=1}^{m_x} c_{x,i}{\MP}(C_{x,i}) \leq \left(\int_{\Delta_{\mathcal{X}}}p(x)\cdot d{\MP}(p)\right)+\frac{\epsilon}{12|\mathcal{X}|}=\frac{1}{|\mathcal{X}|}+\frac{\epsilon}{12|\mathcal{X}|}.$$

Let $d$ be the total variation distance on $\Delta_{\mathcal{X}}$, i.e., $d(p,p')=\frac{1}{2}\|p-p'\|_1$. Since $\Delta_{\mathcal{X}}$ is compact, it can be covered by a finite number of open balls of radius $\frac{\epsilon}{4}$, i.e., there exist $h$ points $p_1',\ldots,p_h'$ such that $\displaystyle \Delta_{\mathcal{X}}=\bigcup_{i=1}^h B_{\frac{\epsilon}{4}}(p_i')=\bigcup_{i=1}^h\left\{p\in\Delta_{\mathcal{X}}:\;d(p,p_i')<\frac{\epsilon}{4}\right\}$. For every $1\leq i\leq h$, define the set $$D_i=B_{\frac{\epsilon}{4}}(p_i')\setminus\left(\bigcup_{1\leq j<i} B_{\frac{\epsilon}{4}}(p_j')\right).$$ Clearly, the sets $\{D_i\}_{1\leq i\leq h}$ are disjoint Borel sets that cover $\Delta_{\mathcal{X}}$. Let $\displaystyle n=h\times \prod_{x\in\mathcal{X}} (k_x \cdot m_x)$, and let $A_1,\ldots,A_n$ be the Borel sets obtained by intersecting the sets in the collections $\{D_1,\ldots,D_h\}$, $\{B_{x,i}\}_{1\leq i\leq k_x}$ and $\{C_{x,i}\}_{1\leq i\leq m_x}$ for every $x\in\mathcal{X}$. In other words,
\begin{align*}
\{A_i:&\;1\leq i\leq n\}\\
&=\left\{D_i \cap \bigcap_{x\in\mathcal{X}} (B_{x,i_x}\cap C_{x,j_x}) :\; 1\leq i\leq h,\;\text{and}\;\forall x\in\mathcal{X},\;1\leq i_x\leq k_x\;\text{and}\;1\leq j_x\leq m_x\right\}.
\end{align*}

For every $1\leq i\leq n$, let $l_{x,i}=b_{x,i'}$ where $i'$ is the unique integer satisfying $1\leq i'\leq k_x$ and $A_i\subset B_{x,i'}$. Similarly, let $u_{x,i}=c_{x,i''}$ where $i''$ is the unique integer satisfying $1\leq i''\leq k_x$ and $A_i\subset C_{x,i''}$. Clearly, $l_{x,i} \leq p(x)\leq u_{x,i}$ for every $x\in A_i$. Moreover,
$$\sum_{i=1}^n l_{x,i}{\MP}(A_i) =\sum_{i=1}^{k_x} b_{x,i}{\MP}(B_{x,i}) \geq \frac{1}{|\mathcal{X}|}-\frac{\epsilon}{12|\mathcal{X}|},$$
and
$$\sum_{i=1}^n u_{x,i}{\MP}(A_i) =\sum_{i=1}^{m_x} c_{x,i}{\MP}(C_{x,i}) \leq \frac{1}{|\mathcal{X}|}+\frac{\epsilon}{12|\mathcal{X}|}.$$

For every $1\leq i\leq n$, choose $p_i\in A_i$ arbitrarily. Let $j_i$ be the unique integer such that $A_i\subset D_{j_i}$. Since $D_{j_i}\subset B_{\frac{\epsilon}{4}}(p_{j_i}')$, we have $\displaystyle d(p,p_{j_i}')<\frac{\epsilon}{4}$ for every $p\in A_i$. Therefore, $\displaystyle d(p,p_i)\leq d(p,p_{j_i}')+d(p_{j_i}',p_i)<\frac{\epsilon}{2}$ for every $p\in A_i$.

Define the mapping $f:\Delta_{\mathcal{X}}\rightarrow\Delta_{\mathcal{X}}$ as $f(p)=p_i$ for every $p\in A_i$. Clearly, $d(p,f(p))<\frac{\epsilon}{2}$ for every $p\in\Delta_{\mathcal{X}}$.

Now  let ${\MP}_f=f_{\#}({\MP})$, where $f_{\#}({\MP})$ is the push-forward measure of ${\MP}$ by the mapping $f$, i.e., ${\MP}_f(B)=(f_{\#}({\MP}))(B)={\MP}\big(f^{-1}(B)\big)$ for every Borel set $B$ of $\Delta_{\mathcal{X}}$. We have:
$${\MP}_f(B)=\sum_{p_i\in B} {\MP}\big(f^{-1}(\{p_i\})\big)=\sum_{p_i\in B} {\MP}(A_i)=\sum_{p_i\in B} \alpha_i,$$
where $\alpha_i={\MP}(A_i)$ for every $1\leq i\leq n$. Therefore, ${\MP}_f$ is finitely supported and $$\supp({\MP}_f)\subset \{p_i:\; 1\leq i\leq n\}.$$ Moreover, ${\MP}_f(p_i)=\alpha_i$ for every $1\leq i\leq n$.

Now define the mapping $f_\times:\Delta_{\mathcal{X}}\rightarrow\Delta_{\mathcal{X}}\times \Delta_{\mathcal{X}}$ as $f_\times(p)=(p,f(p))$, and define the measure $\gamma_f$ on $\Delta_{\mathcal{X}}\times\Delta_{\mathcal{X}}$ as the push-forward of ${\MP}$ by $f_\times$, i.e., $\gamma_f(B)=\MP(f_{\times}^{-1}(B))$ for every Borel set $B$ of $\Delta_{\mathcal{X}}\times \Delta_{\mathcal{X}}$. It is easy to see that the marginals of $\gamma_f$ on the first and second factors are ${\MP}$ and ${\MP}_f$ respectively. Therefore, $\gamma_f$ is a coupling between ${\MP}$ and ${\MP}_f$, hence
\begin{align*}
W_1({\MP},{\MP}_f)&=\inf_{\gamma\in\Gamma({\MP},{\MP}_f)}\int_{\Delta_{\mathcal{X}}\times \Delta_{\mathcal{X}}}d(p,p')\cdot d\gamma(p,p')\leq \int_{\Delta_{\mathcal{X}}\times \Delta_{\mathcal{X}}}d(p,p')\cdot d\gamma_f(p,p')\\
&\stackrel{(a)}{=}\int_{\Delta_{\mathcal{X}}}d(p,f(p))\cdot d{\MP}(p)\stackrel{(b)}{\leq} \frac{\epsilon}{2},
\end{align*}
where (a) follows from the fact that $\gamma_f$ is the push-forward of ${\MP}$ by $f_{\times}(p)=(p,f(p))$. (b) follows from the fact that $d(p,f(p))<\frac{\epsilon}{2}$ for every $p\in \Delta_{\mathcal{X}}$. Therefore, ${\MP}_f$ well approximates ${\MP}$ and it is finitely supported. However, ${\MP}_f$ may not be balanced, so more work needs to be done in order to find a balanced and finitely supported meta-probability measure that well approximates ${\MP}$.

For every $x\in\mathcal{X}$, we have:
\begin{align*}
\int_{\Delta_{\mathcal{X}}} p(x)\cdot d{\MP}_f(p)\stackrel{(a)}{=}\int_{\Delta_{\mathcal{X}}} (f(p))(x)\cdot d{\MP}(p)=\sum_{i=1}^n p_i(x){\MP}(A_i)\stackrel{(b)}{\geq} \sum_{i=1}^n l_{i,x}{\MP}(A_i)\geq\frac{1}{|\mathcal{X}|}-\frac{\epsilon}{12|\mathcal{X}|},
\end{align*}
where (a) follows from the fact that ${\MP}_f$ is the push-forward of ${\MP}$ by $f$. (b) follows from the fact that $p_i\in A_i$ and so $p_i(x)\geq l_{i,x}$ for every $1\leq i\leq n$. Similarly, we have
\begin{align*}
\int_{\Delta_{\mathcal{X}}} p(x)\cdot d{\MP}_f(p)=\sum_{i=1}^n p_i(x){\MP}(A_i)\stackrel{(c)}{\leq} \sum_{i=1}^n u_{i,x}{\MP}(A_i)\leq\frac{1}{|\mathcal{X}|}+\frac{\epsilon}{12|\mathcal{X}|},
\end{align*}
where (c) follows from the fact that $p_i\in A_i$ and so $p_i(x)\leq u_{i,x}$ for every $1\leq i\leq n$. We conclude that for every $x\in\mathcal{X}$, we have
\begin{align*}
\left|\pi_{\mathcal{X}}(x)-\int_{\Delta_{\mathcal{X}}} p(x)\cdot d{\MP}_f(p)\right|\leq \frac{\epsilon}{12|\mathcal{X}|},
\end{align*}
where $\pi_{\mathcal{X}}$ is the uniform distribution on $\mathcal{X}$. Define $\tilde{p}\in\Delta_{\mathcal{X}}$ as:
$$\tilde{p}=\int_{\Delta_{\mathcal{X}}} p\cdot d{\MP}_f(p).$$
For every $x\in\mathcal{X}$, define
$$p'(x)=\frac{6(\pi_{\mathcal{X}}(x)-\tilde{p}(x))}{\epsilon}+\tilde{p}(x).$$
Clearly, $\displaystyle\sum_{x\in\mathcal{X}}p'(x)=1$. Moreover,
\begin{align*}
p'(x)&=\frac{6(\pi_{\mathcal{X}}(x)-\tilde{p}(x))}{\epsilon}+\tilde{p}(x)\\
&\stackrel{(a)}{\geq} \frac{\displaystyle 6\left(\pi_{\mathcal{X}}(x)-\pi_{\mathcal{X}}(x)-\frac{\epsilon}{12|\mathcal{X}|}\right)}{\epsilon}+\frac{1}{|\mathcal{X}|}-\frac{\epsilon}{12|\mathcal{X}|}=\frac{1}{2|\mathcal{X}|}-\frac{\epsilon}{12|\mathcal{X}|}\geq 0,
\end{align*}
Where (a) follows from the fact that $\displaystyle|\pi_{\mathcal{X}}(x)-\tilde{p}(x)|\leq \frac{\epsilon}{12|\mathcal{X}|}$. We conclude that $p'\in\Delta_{\mathcal{X}}$. Now define the meta-probability measure ${\MP}'$ as follows:
$$\displaystyle{\MP}'=\frac{\epsilon}{6}\cdot\delta_{p'}+\left(1-\frac{\epsilon}{6}\right){\MP}_f,$$ where $\delta_{p'}$ is a Dirac measure centered at $p'$.

For every $x\in\mathcal{X}$, we have
\begin{align*}
\int_{\Delta_{\mathcal{X}}} p(x)\cdot d{\MP}'(p)&=\frac{\epsilon}{6}\cdot p'(x)+\left(1-\frac{\epsilon}{6}\right)\int_{\Delta_{\mathcal{X}}} p(x)\cdot d{\MP}_f(p)=\frac{\epsilon}{6}\cdot p'(x)+\left(1-\frac{\epsilon}{6}\right)\cdot\tilde{p}(x)\\
&=\pi_{\mathcal{X}}(x)-\tilde{p}(x)+\frac{\epsilon}{6}\cdot\tilde{p}(x) + \left(1-\frac{\epsilon}{6}\right)\tilde{p}(x)=\pi_{\mathcal{X}}(x).
\end{align*}
Therefore, ${\MP}'$ is balanced and finitely supported. Moreover,
\begin{align*}
W_1({\MP},{\MP}')&\leq W_1({\MP},{\MP}_f)+W_1({\MP}_f,{\MP}')\stackrel{(a)}{\leq} \frac{\epsilon}{2} + \|{\MP}_f-{\MP}'\|_{TV}\\
&=\frac{\epsilon}{2} + \left\|{\MP}_f-  \left(1-\frac{\epsilon}{6}\right){\MP}_f - \frac{\epsilon}{6}\cdot\delta_{p'}\right\|_{TV}\leq \frac{\epsilon}{2} + \left\|\frac{\epsilon}{6}\cdot{\MP}_f\right\|_{TV} + \left\|\frac{\epsilon}{6}\delta_{p'}\right\|_{TV}\\
&=\frac{\epsilon}{2} + \frac{\epsilon}{6}+\frac{\epsilon}{6}<\epsilon,
\end{align*}
where (a) follows from the fact that the $1^{st}$ Wasserstein metric is upper bounded by the total variation multiplied by the diameter of $\Delta_{\mathcal{X}}$ (which is equal to 1 in our case) \cite{WassersteinMetric}. We conclude that $\mathcal{MP}_{bf}(\mathcal{X})$ is dense in $\mathcal{MP}_{b}(\mathcal{X})$ which is weakly-$\ast$ closed. Therefore, $\mathcal{MP}_{b}(\mathcal{X})$ is the weak-$\ast$ closure of $\mathcal{MP}_{bf}(\mathcal{X})$.

\bibliographystyle{IEEEtran}
\bibliography{bibliofile}
\end{document}